\documentclass{amsart}
\usepackage[T1]{fontenc}
\usepackage[utf8]{inputenc}
\usepackage{lmodern}
\usepackage{amsmath,amssymb,amsthm,mathtools}
\usepackage[margin=1.1in]{geometry}
\usepackage{xcolor}
\definecolor{linkblue}{RGB}{38,68,105}
\definecolor{citegreen}{RGB}{45,88,68}
\usepackage[colorlinks=true,linkcolor=linkblue,citecolor=citegreen,urlcolor=linkblue,pdfborder={0 0 0},linktoc=all]{hyperref}
\hypersetup{pdftitle={Localization and recursion for weighted Goldman volume},pdfauthor={Shiyu Cao}}

\newcommand{\R}{\mathbb R}
\newcommand{\C}{\mathbb C}
\newcommand{\Z}{\mathbb Z}

\newcommand{\cM}{\mathcal M}
\newcommand{\cH}{\mathcal H}
\newcommand{\cC}{\mathcal C}
\newcommand{\bM}{\overline{\mathcal M}}
\newcommand{\bC}{\overline{\mathcal C}}
\newcommand{\bE}{\overline{\mathbb E}}
\newcommand{\MG}{\cM_g}
\newcommand{\HG}{\cH_g}
\newcommand{\Eg}{\mathbb E_3}
\newcommand{\PP}{\mathbb P}
\newcommand{\OO}{\mathcal O}
\newcommand{\cR}{\mathcal R}
\newcommand{\dd}{\mathrm d}
\newcommand{\dc}{\dd^c}
\newcommand{\dM}{\dd_{\mathcal M}}
\newcommand{\dP}{\dd_{\PP}}
\newcommand{\dY}{\dd_Y}
\newcommand{\dMc}{\dd^c_{\mathcal M}}
\newcommand{\dPc}{\dd^c_{\PP}}
\newcommand{\dA}{\dd A}
\newcommand{\dV}{\dd V}
\newcommand{\dnu}{\dd\nu}
\newcommand{\ii}{\mathrm i}
\newcommand{\WP}{\mathrm{WP}}
\newcommand{\rad}{\mathrm{rad}}
\newcommand{\rot}{\mathrm{rot}}
\newcommand{\aff}{\mathrm{aff}}
\newcommand{\vol}{\mathrm{vol}}
\newcommand{\flatarea}{\mathrm{flat}}
\newcommand{\refn}{\mathrm{ref}}
\newcommand{\hyp}{\mathrm{hyp}}
\newcommand{\sm}{\mathrm{sm}}
\newcommand{\coeff}{\mathrm{coeff}}
\newcommand{\norm}[1]{\left\lVert#1\right\rVert}
\newcommand{\pair}[2]{\left\langle#1,#2\right\rangle}
\DeclareMathOperator{\Hit}{Hit}
\DeclareMathOperator{\Mod}{Mod}
\DeclareMathOperator{\Tot}{Tot}
\DeclareMathOperator{\Vol}{Vol}
\DeclareMathOperator{\tr}{tr}
\DeclareMathOperator{\ch}{ch}
\DeclareMathOperator{\Seg}{Seg}
\DeclareMathOperator{\End}{End}
\DeclareMathOperator{\Id}{Id}
\DeclareMathOperator{\Hess}{Hess}
\DeclareMathOperator{\Area}{Area}
\DeclareMathOperator{\Sym}{Sym}
\DeclareMathOperator{\diver}{div}
\DeclareMathOperator{\Aut}{Aut}
\DeclareMathOperator{\rank}{rank}

\DeclareMathOperator{\Pf}{Pf}
\DeclareMathOperator{\grad}{grad}
\DeclareMathOperator{\supp}{supp}

\theoremstyle{plain}
\newtheorem{thm}{Theorem}[section]
\newtheorem{prop}[thm]{Proposition}
\newtheorem{lem}[thm]{Lemma}
\newtheorem{cor}[thm]{Corollary}
\theoremstyle{definition}
\newtheorem{defn}[thm]{Definition}
\theoremstyle{remark}
\newtheorem{rem}[thm]{Remark}
\numberwithin{equation}{section}

\title[Localization and recursion of Goldman volume]{Localization and recursion of weighted Goldman volume on  $\mathrm{SL}(3,\R)$-Hitchin component}
\author{Shiyu Cao}
\address{Chern Institute of Mathematics and LPCM, Nankai University}
\email{shiyucao@126.com}
\date{}

\begin{document}

\begin{abstract}
We compute the Goldman volume of the $\mathrm{SL}(3,\R)$-Hitchin moduli space of a closed oriented surface of genus $g\ge 2$, weighted by $e^{-s\mu}$ for $s>0$, where $\mu$ is the shifted Blaschke area. The integral converges to a polynomial in $s^{-1}$ whose coefficients are intersection numbers of $\kappa_1$ and signed Segre classes of the bundle of regular cubic differentials on $\bM_g$. Using the Labourie--Loftin parametrization, we prove this by equivariant localization along the zero section, establishing absolute integrability and vanishing boundary integrals under simultaneous nodal degeneration. Extending these coefficients to all stable pairs $(g,n)$, we express the underlying characteristic class via Hodge and $\kappa$-classes and obtain an effective recursion through Chiodo classes and topological recursion, yielding explicit volumes in genera two and three and the polynomial density of the shifted Blaschke area.
\end{abstract}

\maketitle
\setcounter{tocdepth}{1}
\tableofcontents

\section{Introduction}\label{sec:introduction}

We compute the exponentially weighted Goldman volume of the $\mathrm{SL}(3,\R)$-Hitchin component with respect to the shifted Blaschke area and establish a recursion for every coefficient.

Throughout the paper, $S$ is a closed oriented surface of genus $g\geq2$. The holonomy of a marked hyperbolic structure identifies Teichm\"uller space with a connected component of the $\mathrm{PSL}(2,\R)$-character variety. Hitchin constructed analogous components for split real groups using Higgs bundles \cite{Hitchin1992}. For $\mathrm{PSL}(n,\R)$, a Hitchin component contains the representations obtained by composing Fuchsian representations with the principal embedding $\mathrm{PSL}(2,\R)\to\mathrm{PSL}(n,\R)$. Its representations are discrete and faithful \cite{LabourieAnosov2006,FockGoncharov2006}; see \cite[\S\S1--3 and \S9]{Wienhard2018} for geometric background and symplectic geometry in higher Teichm\"uller theory.

Nonabelian Hodge theory uses harmonic metrics to relate reductive representations to polystable Higgs bundles \cite{Hitchin1987,Donaldson1987,Corlette1988,Simpson1988,Simpson1992}. For a complex structure $X$ on $S$, the Hitchin section parametrizes the component by
$$
 \bigoplus_{j=2}^n H^0(X,K_X^j),
$$
where $K_X$ is the canonical bundle and the parameter space has real dimension $(n^2-1)(2g-2)$.

For $n=3$, the component also has a concrete description in projective geometry. By Goldman's study of convex real projective structures \cite{Goldman1990} and the holonomy theorem of Choi and Goldman \cite{ChoiGoldman1993}, the holonomy map identifies the deformation space of marked convex $\mathbb{RP}^2$-structures on $S$ with $\Hit_3(S)$. Each such structure realizes $S$ as the quotient of a properly convex open domain in $\mathbb{RP}^2$ by its holonomy group. The natural homomorphism $\mathrm{SL}(3,\R)\to\mathrm{PGL}(3,\R)$ is an isomorphism of real Lie groups, so the holonomy takes values canonically in $\mathrm{SL}(3,\R)$. The Fuchsian locus consists of the structures obtained from hyperbolic metrics, for which the developing domain is the interior of a conic.

Hyperbolic affine spheres give a parametrization in which the complex structure varies. By the existence theory of Cheng and Yau \cite{ChengYau1986} and Wang's structure equations \cite{Wang1991}, a convex projective surface has an associated Blaschke metric and holomorphic cubic differential. Loftin \cite{Loftin01} and Labourie \cite{Labourie2006} proved that these data give a mapping class group equivariant parametrization of $\Hit_3(S)$ by the total space of the holomorphic vector bundle over Teichm\"uller space with fiber $H^0(X,K_X^3)$. The complex structure $X$ is the conformal structure of the Blaschke metric. In the Higgs bundle description, the corresponding harmonic map to $\mathrm{SL}(3,\R)/\mathrm{SO}(3)$ is conformal, so its Hopf differential vanishes and the Higgs field is cyclic. Labourie extended uniqueness of the equivariant minimal surface to Hitchin representations into split real groups of rank two \cite{Labourie2017}. Under this parametrization, the zero section corresponds to the Fuchsian locus.

The equations for the hyperbolic affine sphere reduce the harmonic metric to a scalar elliptic equation on $X$, and the Blaschke area is bounded below by the hyperbolic area. Estimates of Ouyang and Tamburelli \cite[Lemmas~3.3--3.4]{OT20} compare it with $\int_X|q|^{2/3}$, the area of the singular flat metric defined by a nonzero cubic differential. The comparison controls area growth for large cubic differentials with arbitrary zero multiplicities. We use Wang's equation to bound radial derivatives and variations with respect to the cubic differential and the complex structure.

Goldman defined a symplectic form on the smooth locus of the character variety using the cup product and an invariant pairing on the Lie algebra \cite{Goldman84}. Rungi and Tamburelli \cite{RungiTamburelli2025} constructed a compatible closed two-form and proved nondegeneracy near the Fuchsian locus. Global compatibility of the Goldman form was established through complex harmonic maps by El Emam and Sagman \cite{ElEmamSagman2025} and through joint moduli spaces of Higgs bundles by Collier, Toulisse, and Wentworth \cite{CTW2025}; see also the independent universal construction of \cite{AGGT2025}. The real metric $(A,B)\mapsto\omega_G(A,IB)$, where $I$ is the bundle complex structure, is pseudo-K\"ahler and restricts to a negative definite metric on the fibers of $\mathbb E_3$ for the sign of $\omega_G$ specified below. By Tamburelli's comparison theorem \cite{Tamburelli26}, the closed form of Rungi and Tamburelli equals the Goldman form for the corresponding invariant pairings. In particular, fiberwise rotation of the cubic differential is Hamiltonian, with moment map an affine function of the Blaschke area. We also use the horizontal projection formula of \cite{CTW2025} to obtain a metric estimate controlling the Goldman volume density.

Regular cubic differentials also govern the degeneration of convex $\mathbb{RP}^2$-structures near the boundary of $\bM_g$. Loftin \cite{Loftin2019} described convex projective structures under neck separation using regular cubic differentials on stable curves. On the normalization, these have poles of order at most three at the preimages of each node, with opposite leading coefficients on the two branches. The associated bundle over the Deligne--Mumford compactification supplies the characteristic classes in our volume formula.

\subsection{Weighted Goldman volume}\label{sec:weighted-goldman}

The Labourie--Loftin parametrization is equivariant under the mapping class group $\Mod_g$ and descends to an identification
\begin{equation}\label{eq:intro-cubic-bundle}
 p\colon \cH_g:=\Hit_3(S)/\Mod_g
 \simeq \Tot(\mathbb E_3)\longrightarrow\cM_g,
 \qquad (\mathbb E_3)_X=H^0(X,K_X^3).
\end{equation}
Here $\cM_g$ is the moduli orbifold of smooth genus-$g$ curves of complex dimension $d$, $\mathbb E_3$ has complex rank $r$, and $\cH_g$ has real dimension $2N$, where
$$
 d=3g-3,\qquad r=5g-5,\qquad N=d+r=8g-8.
$$

The form $\omega_G$ is defined by the invariant pairing $(A,B)\mapsto\tr_{\R^3}(AB)$ on $\mathfrak{sl}(3,\R)$. We use the clockwise action
\begin{equation}\label{eq:intro-circle-action}
 \Psi_\theta(X,q)=(X,e^{-\ii\theta}q),
 \qquad \theta\in\R/(2\pi\Z),
\end{equation}
and denote its generator by $V_{\mathrm{rot}}$. We denote by $q$ the Pick differential multiplied by $\sqrt{2}$. If $A_B$ denotes Blaschke area and $i_0$ denotes the zero section, the moment map and the restriction of $\omega_G$ satisfy
\begin{equation}\label{eq:intro-hamiltonian-data}
 \begin{aligned}
 A_0&=4\pi(g-1),&\qquad \mu&=A_B-A_0\geq0,\\
 \iota_{V_{\mathrm{rot}}}\omega_G&=-\dd\mu,
 &\qquad \omega_0&=i_0^*\omega_G=2\omega_{\WP}.
 \end{aligned}
\end{equation}
These identities follow from Propositions~\ref{ch01:prop:import-goldman} and~\ref{ch01:prop:normal-form}. The normalization of the Weil--Petersson symplectic form agrees with \cite[discussion after equation~(7) and Section~5]{Wolpert2007Cusps} and \cite[Theorem~7 and \S3.1]{Do2011WP}.

Integrating over $\cH_g$ with respect to its Goldman symplectic orientation, we set
\begin{equation}\label{eq:intro-weighted-integral}
 \dd V_G=\frac{\omega_G^N}{N!},\qquad
 Z_g(s)=\int_{\cH_g}e^{-s\mu}\,\dd V_G
 \quad(s>0).
\end{equation}
Integrals over moduli spaces are understood in the orbifold (stack) sense: integration on a local quotient chart is divided by the order of its finite uniformizing group, including any subgroup acting trivially, and integration on a finite orbifold cover is divided by its degree. The base $\cM_g$ and its compactification have their complex orientations.

The coefficients of the volume formula are defined on the stack $\bM_g$ of stable curves. We write
\begin{equation}\label{eq:intro-stable-cubic-bundle}
 \begin{gathered}
 \overline\pi\colon\overline{\cC}_g\longrightarrow\bM_g,\\
 \overline{\mathbb E}_3=\overline\pi_*\omega_{\overline\pi}^{\otimes3},
 \qquad
 \kappa_1=\overline\pi_*c_1(\omega_{\overline\pi})^2.
 \end{gathered}
\end{equation}
Here $\omega_{\overline\pi}$ is the relative dualizing sheaf; over the universal curve $\pi\colon\cC_g\to\cM_g$ it restricts to $K_\pi$, so $\mathbb E_3=\pi_*K_\pi^{\otimes3}$. By Proposition~\ref{ch02:prop:cubic-extension}, $\overline{\mathbb E}_3$ is a rank-$r$ vector bundle extending $\mathbb E_3$, compatible with arbitrary scheme base change in stable families and with fibers consisting of regular cubic differentials. Its characteristic classes are computed by Chiodo's formula~\cite[Theorem~1.1.1]{Chiodo08} applied to $\omega_{\overline\pi}^{\otimes3}$, as recorded in Proposition~\ref{ch02:prop:cubic-grr}.

The signed Segre classes $h_k(\overline{\mathbb E}_3)$ and Segre classes $\Seg_k(\overline{\mathbb E}_3)$ are defined in Definition~\ref{ch02:defn:symmetric-classes}. The class $h_k$ has real cohomological degree $2k$, $h_0=1$, and $h_k=(-1)^k\Seg_k$, where $\Seg_k$ is the degree-$2k$ component of $c(\overline{\mathbb E}_3)^{-1}$. For $0\leq k\leq d$, we define
\begin{equation}\label{eq:intersection-coefficients}
 A_{g,k}
 =\frac{(2\pi)^{r+k}(4\pi^2)^{d-k}}{(d-k)!}
   \int_{\bM_g}
      \kappa_1^{d-k}h_k(\overline{\mathbb E}_3).
\end{equation}
The Weil--Petersson cohomology class extends to $\bM_g$ with
$$
 [\omega_{\WP}]=2\pi^2\kappa_1;
$$
see \cite[Theorem~7 and \S3.1]{Do2011WP} and the normalization discussion in \cite[\S5]{Wolpert2007Cusps}. The Weil--Petersson volume coefficients are intersection numbers on the compactified moduli stack by Mirzakhani's volume formula \cite[Theorem~4.4]{Mirzakhani2007}.

\begin{thm}\label{thm:weighted-volume}
For every $g\geq2$ and $s>0$, the weighted Goldman integral converges and satisfies
\begin{equation}\label{eq:weighted-volume}
 0<Z_g(s)
 =\sum_{k=0}^d A_{g,k}s^{-(r+k)}
 <\infty.
\end{equation}
The equality may be differentiated any finite number of times in $s$, with differentiation under the integral locally uniform on $(0,\infty)$.
\end{thm}

The analytic proof establishes absolute integrability over $\cH_g$ and the vanishing of the boundary integrals of the transgression form constructed in Section~\ref{ch03:sec:fiber-integration}. Here boundary limits are taken after integration over the boundary with its induced orientation. The estimates hold uniformly under simultaneous stable degeneration with independent plumbing parameters and arbitrary cubic differentials, including the zero differential and differentials with multiple zeros; for each genus, they are locally uniform in $s>0$.

Sun~\cite{Sun2020} studied Goldman volumes of subsets defined by bounded projective invariants and relations between several notions of bounded geometry for convex projective structures.

Recently, Li, Liu, Su, and Sun~\cite[Theorem~1.1 and Corollary~1.2]{LiLiuSuSun2026} proved polynomial upper bounds for the Goldman volumes of canonical Hilbert--Busemann area sublevel sets on moduli spaces of unmarked convex $\mathbb{RP}^2$ structures, for closed surfaces and for bordered surfaces with prescribed positive simple-root boundary lengths. They also established finiteness and explicit upper bounds for the corresponding exponentially weighted volumes for every positive parameter.

\subsection{A recursion on marked curves}

The coefficients of~\eqref{eq:weighted-volume} belong to a family on the moduli spaces $\bM_{g,n}$ of stable $n$-pointed curves. For~$2g-2+n>0$, we write~$\pi\colon\overline{\cC}_{g,n}\to\bM_{g,n}$ for the universal curve and~$D_{\mathrm{mark}}$ for the sum of its sections. The bundle
\begin{equation}\label{eq:intro-marked-bundle}
 \mathbb F_{g,n}=\pi_*\omega_\pi^{\otimes3}(2D_{\mathrm{mark}})
\end{equation}
has rank~$5g-5+2n$. We use logarithmic~$\kappa$ classes, cotangent line classes~$\psi_i$, and the Hodge classes~$\lambda_j$. With~$\Lambda_g(t)=\sum_{j=0}^g t^j\lambda_j$, the inverse Chern polynomial of $\mathbb F_{g,n}$ is
\begin{equation}\label{eq:intro-Hodge-expression}
 C_{g,n}(t):=c_{-t}(\mathbb F_{g,n})^{-1}
 =\Lambda_g(t)\exp\left(\sum_{m\ge1}\frac{1+2^m}{m}t^m\kappa_m\right).
\end{equation}
We define
\begin{equation}\label{eq:intro-marked-polynomial}
 \mathcal V_{g,n}(a,t;\mathbf x)
 =\int_{\bM_{g,n}}e^{a\kappa_1+\sum_i x_i\psi_i}C_{g,n}(t).
\end{equation}
These are homogeneous polynomials of degree~$3g-3+n$. In this notation, the volume formula~\eqref{eq:weighted-volume} is the unmarked specialization ($n=0$, $g\ge2$):
\begin{equation}\label{eq:intro-recursive-volume}
 Z_g(s)=\left(\frac{2\pi}{s}\right)^{5g-5}
 \mathcal V_{g,0}\left(4\pi^2,\frac{2\pi}{s}\right).
\end{equation}

The class in~\eqref{eq:intro-Hodge-expression} is~$\Omega_{g,n}^{[-t]}(1,3;1^n)$. The shift and forgetting identities for~$\Omega$ classes are established by Giacchetto, Lewa\'nski, and Norbury~\cite[Theorem~4.1 and Appendix~A]{GLN}. With root parameter one and exponent parameter three, we apply the correspondence of Lewa\'nski, Popolitov, Shadrin, and Zvonkine~\cite[Remark~2 and Theorem~4.5]{LPSZ}. We use the normalization
\begin{equation}\label{eq:intro-spectral-curve}
 X(z)=\log z-z,\qquad Y(z)=z^3/3,\qquad
 B(z_1,z_2)=\frac{\dd z_1\,\dd z_2}{(z_1-z_2)^2}.
\end{equation}
The associated recursion computes all descendant integrals of~$C_{g,n}$ from residues at~$z=1$. We incorporate~$e^{a\kappa_1}$ using~\eqref{eq:marked-full-assembly} and the series
\begin{equation}\label{eq:intro-insertion-series}
 S_{a,t}(w)=w(1-tw)(1-2tw)(1-e^{-aw}).
\end{equation}
For a descendant tuple~$\mathbf d\ge0$ with~$|\mathbf d|\le3g-3+n$, at most~$3g-3+n-|\mathbf d|$ additional markings occur. Since $2g-2+n$ decreases on every stable factor on the right-hand side of the recursion and~\eqref{eq:marked-unmarked-J} recovers the unmarked case from one-point descendants, every coefficient, including those in genera zero and one, is determined recursively from rational initial data by Theorem~\ref{thm:marked-complete-recursion}.

The factors~$(1-tw)(1-2tw)$ also enter the forgetting identities and account for the absence of a flat unit for $C_{g,n}(t)$ under the marking convention~\eqref{eq:intro-marked-bundle}. We prove the insertion formula in the presence of arbitrary $\psi$-classes at the original markings, derive the resulting string and dilaton identities, and give an independent computation by Mumford's formula and the Witten--Kontsevich theorem. These formulas evaluate~\eqref{eq:intro-recursive-volume} in every genus.

\subsection*{Relation to volume recursions on moduli spaces of curves}

At~$t=0$, the specialization~$a=2\pi^2$ and~$x_i=L_i^2/2$ recovers the Weil--Petersson volume polynomial with full stabilizer weights, by Mirzakhani's intersection formula~\cite{Mirzakhani2007}; the convention is specified in~\eqref{eq:WP-stack-definition}. The conversion between~$\kappa$ insertions and additional descendants belongs to the higher Weil--Petersson theory of Kaufmann, Manin, and Zagier~\cite{KMZ}. We apply this conversion to the full marked polynomial $\mathcal V_{g,n}$, which in particular yields the unmarked volume $Z_g(s)$.

A close precedent is the expression of principal-stratum quadratic Masur--Veech volumes through Segre classes and linear Hodge integrals by Chen, M\"oller, and Sauvaget~\cite{CMS}. The appendix by Borot, Giacchetto, and Lewa\'nski treats the family of spectral curves $X=-z+\alpha\log z$, $Y=z^b$ for $(\alpha,b)\in\C^*\times\Z^*$, which includes~\eqref{eq:intro-spectral-curve} up to rescaling $Y$. We use this spectral correspondence and the Hodge expression with our normalization and marked insertion formula to evaluate the characteristic classes arising from Goldman localization.

The shorter recursions for specialized linear Hodge integrals in~\cite{Kazarian,YZZ} provide further comparisons. Even at $a=0$, expressing $C_{g,n}(t)$ through descendant insertions involves both $\psi^2$ and $\psi^3$, while general $a$ introduces arbitrarily high powers of $\psi$; the insertion formula keeps all these powers explicit.

\subsection*{Equivariant localization}

For actions of compact tori on compact oriented manifolds without boundary, equivariant localization expresses the integral of an equivariantly closed form in terms of data on the fixed locus. Berline and Vergne proved a localization formula for nondegenerate isolated zeros \cite[Theorem~1.6]{BerlineVergne1983}. Atiyah and Bott expressed the contributions of fixed components by division by the equivariant Euler class of the normal bundle \cite[\S3]{AtiyahBott1984Moment}. For compact Hamiltonian torus manifolds, the pushforward of symplectic measure has a polynomial density on each chamber of regular values, by the Duistermaat--Heckman theorem \cite{DH82}.

Mathai and Quillen constructed Gaussian differential forms representing Thom classes \cite{MathaiQuillen1986}. Bismut gave Gaussian proofs of equivariant localization, including a fiber-integration formula modulo exact forms for compact equivariant submersions, and related these formulas to the family index theorem \cite[Theorems~1.3 and~1.9, Section~III]{Bismut1986Localization}. In the smooth setting of Higgs moduli spaces over a fixed curve, Liu used the circle action and its proper moment map to obtain relations among integrals over fixed components \cite[Section~6]{Liu1997Heat}.

For the bundle $\mathbb E_3$ in \eqref{eq:intro-cubic-bundle}, the circle action fixes the zero section $\cM_g$, whose normal bundle is $\mathbb E_3$. The inverse equivariant Euler class accounts for the characteristic classes in \eqref{eq:intersection-coefficients}. The Hamiltonian $\mu$ is proper on each fiber of $\mathbb E_3$, but $\mu^{-1}(0)=\cM_g$ is noncompact. We first prove a localization identity of differential forms over $\cM_g$, with an explicit transgression form. Estimates under simultaneous plumbing degeneration then give absolute convergence and vanishing boundary integrals. These estimates justify integration over the base and yield \eqref{eq:weighted-volume}.

\subsection*{Moment sublevels}

The pushforward of Goldman measure by $\mu$ is the positive Borel measure
$$
 \nu_g(B)=\int_{\mu^{-1}(B)}\dd V_G
 \qquad\text{for every Borel set }B\subset[0,\infty).
$$
The Laplace transform in Theorem~\ref{thm:weighted-volume} determines its density.

\begin{cor}\label{cor:moment-density}
The measure $\nu_g$ is locally finite and absolutely continuous, with
\begin{equation}\label{eq:moment-density}
 \dd\nu_g(\tau)=p_g(\tau)\,\dd\tau,\qquad
 p_g(\tau)=\sum_{k=0}^d
       A_{g,k}\frac{\tau^{r+k-1}}{(r+k-1)!}.
\end{equation}
For every $\tau\geq0$, $p_g(\tau)\geq0$.
\end{cor}

\begin{cor}\label{cor:sublevel-volume}
For every real $\Lambda\geq0$,
\begin{equation}\label{eq:sublevel-volume}
 \Vol_G\{\mu\leq\Lambda\}
 =\Vol_G\{A_B\leq4\pi(g-1)+\Lambda\}
 =\sum_{k=0}^d A_{g,k}\frac{\Lambda^{r+k}}{(r+k)!}.
\end{equation}
The polynomial in \eqref{eq:sublevel-volume} has degree at most $N$, and $p_g$ has degree at most $N-1$.
\end{cor}

The coefficient at the smallest power of $\Lambda$ is determined by the volume of the zero section. The Weil--Petersson volume is denoted by
$$
 V_{g,0}^{\WP}=\frac1{d!}\int_{\cM_g}\omega_{\WP}^d.
$$

\begin{cor}\label{cor:wp-leading}
The first coefficient is
$$
 A_{g,0}=(2\pi)^r2^dV_{g,0}^{\WP}.
$$
Consequently, for each $g$ and as $\Lambda\downarrow0$,
$$
 \Vol_G\{\mu\leq\Lambda\}
 =\frac{(2\pi)^r2^dV_{g,0}^{\WP}}{r!}\Lambda^r
   +O(\Lambda^{r+1}).
$$
\end{cor}

We prove Corollaries~\ref{cor:moment-density}--\ref{cor:wp-leading} from \eqref{eq:weighted-volume} by exponential tilting and uniqueness of the Laplace transform for finite positive measures \cite[Proposition~1.2, p.~2]{SSV2010}.

\subsection*{Organization of the paper}

The Hamiltonian identities and estimates on the fibers of $\mathbb E_3$ are proved in Section~\ref{ch01:sec:cubic}. The extension to stable pointed curves and the Hodge expression occupy Section~\ref{ch02:sec:stable-cubics}, followed by the complete marked recursion in Section~\ref{sec:marked-recursion}. Fiber integration in Section~\ref{ch03:sec:fiber-integration} produces the characteristic form and its transgression. The simultaneous plumbing estimates and global Stokes argument are given in Sections~\ref{ch04:sec:plumbing} and~\ref{ch05:sec:boundary}. The explicit coefficients and area distribution are obtained in Section~\ref{sec:volume-formulas}. The appendices contain the stable-graph calculation, the rational Hodge recurrences, and localized boundary estimates.

\subsection*{Acknowledgment}

This work was inspired by K.~Liu's expectation, expressed in \cite[Section~6]{Liu1997Heat}, that localization formulas on Higgs moduli spaces could be useful for computations in Teichm\"uller theory in connection with Goldman's work. The author began this project in the summer of 2026 after becoming aware of Tamburelli's work \cite{Tamburelli26} on the Goldman symplectic form and the Hamiltonian circle action on the $\mathrm{SL}(3,\R)$-Hitchin component, and the preparation of the manuscript continued into the autumn.

During the final editing of this manuscript, the author learned of the work of Li, Liu, Su, and Sun~\cite{LiLiuSuSun2026}. The two works were carried out entirely independently. They address a common subject through substantially different methods: Li, Liu, Su, and Sun use techniques from Teichm\"uller theory, whereas the present work uses localization to obtain exact volume formulas.

\subsection*{AI-use disclosure}
In developing this work, the author benefited substantially from repeated conversations with generative large language models, which helped him deepen his understanding of affine differential geometry and higher Teichm\"uller theory. Notably, after repeated rounds of discussions with ChatGPT 6, the author figured out the blueprint of the recursive formula which was not included in the preliminary version after it recommended  references such as \cite{CMS,GLN,KMZ,Kazarian,LPSZ,YZZ}. ChatGPT 6 Astra also served as an auxiliary tool in some routine calculations and language editing. The author takes full responsibility for the content of this paper.

\section{Cubic differentials and the Goldman form}
\label{ch01:sec:cubic}

The Blaschke area controls the growth of the coordinate components of $\omega_G$ on the fibers of $\mathbb E_3$, uniformly over compact subsets of $\cM_g$.

\subsection{Geometric data}
\label{ch01:setup:cubic}

We use the cubic differential $q$, the area functions $A_0,A_B,\mu$, and the clockwise circle action introduced in Subsection~\ref{sec:weighted-goldman}. The zero section is denoted by $i_0$.

For $X\in\MG$, let $\sigma_X$ be its hyperbolic metric of curvature $-1$. We use the Laplacian
$$
\Delta_\sigma=-\diver_\sigma\nabla.
$$
The Blaschke metric is denoted by $h=e^u\sigma_X$. In a local coordinate with $\sigma_X=\rho^2|\dd z|^2$ and $q=f(\dd z)^3$, our scalar norm is $|q|_{\sigma_X}^2=|f|^2\rho^{-6}$. We put
$$
Q=|q|_{\sigma_X}^2.
$$
The Hermitian $L^2$ pairing on cubic differentials is
$$
\langle\psi,\eta\rangle_{L^2(\sigma_X)}
=\int_X\langle\psi,\eta\rangle_{\sigma_X}\,\dA_{\sigma_X}.
$$

The vertical radial field is $R_{\rad}$, and $I$ denotes multiplication by $\ii$ on a fiber of $\mathbb E_3$. Thus
$$
V_{\rot}=-IR_{\rad}.
$$
The Weil--Petersson form is
$$
\omega_{\WP}=\sum_{j=1}^{d}\dd \ell_j\wedge \dd \vartheta_j,
$$
where a positive Dehn twist has $\vartheta_j$-period $\ell_j$.

The rank of $\mathbb E_3$ follows from Riemann--Roch for the line bundle $K_X^3$ on the smooth projective curve $X$ \cite[Tag~0BS6, Lemma~53.5.2]{StacksR1}. By Serre duality \cite[Tag~0BS4, Remark~53.4.3]{StacksR1}, $H^1(X,K_X^3)$ is isomorphic to the dual of $H^0(X,K_X^{-2})$, which vanishes because $K_X^{-2}$ has negative degree. Hence
$$
h^0(X,K_X^3)=6g-6+1-g=5g-5.
$$

\begin{prop}
\label{ch01:prop:import-parametrization}
The marked $\mathrm{SL}(3,\R)$-Hitchin component is identified, equivariantly under orientation-preserving mapping classes, with the bundle of holomorphic cubic differentials over Teichm\"uller space. After quotienting by the mapping class group, this identification gives the bundle $p\colon\HG\to\MG$.

For each pair $(X,q)$, the hyperbolic affine sphere construction gives a unique Blaschke metric conformal to $X$, satisfying
$$
K_h-|q|_h^2=-1.
$$
The connection
\begin{equation}
\nabla^B=\nabla^h+\frac1{\sqrt2}h^{-1}\operatorname{Re}(q)
\label{ch01:eq:blaschke-connection}
\end{equation}
is torsion-free, projectively flat, and preserves $\dA_h$. On $TS\oplus\R$, the connection
\begin{equation}
D_{\aff,U}(Z,a)
=\bigl(\nabla^B_UZ+aU,\ U(a)+h(U,Z)\bigr)
\label{ch01:eq:affine-flat-connection}
\end{equation}
is flat and has the holonomy associated to $(X,q)$.
\end{prop}

\begin{proof}
The holonomy identification is due to Choi and Goldman \cite{ChoiGoldman1993}. The parametrization and the construction by hyperbolic affine spheres are proved in \cite[Theorems~1.0.2, 4.1.1, and 4.2.1]{Labourie2006} and \cite[Theorem~2]{Loftin01}. The flat connection on $TS\oplus\R$ and its projective holonomy are described in \cite[Section~2.2.2, Proposition~2.2.2, and Section~8]{Labourie2006}.
\end{proof}

The holonomy map carries the Goldman form to the bundle of holomorphic cubic differentials, where scalar rotation satisfies the Hamiltonian identity below.

\begin{prop}
\label{ch01:prop:import-goldman}
On the smooth Hitchin locus, the trace pairing defines a smooth closed nondegenerate form $\omega_G$. For covariantly closed tangent representatives $A$ and $B$,
\begin{equation}
\omega_G([A],[B])=\int_S\tr_{\R^3}(A\wedge B).
\label{ch01:eq:goldman-cup-product}
\end{equation}
This form is invariant under orientation-preserving mapping classes. The Hamiltonian identity is
\begin{equation}
\iota_{V_{\rot}}\omega_G=-\dd \mu.
\label{ch01:eq:hamiltonian}
\end{equation}
\end{prop}

\begin{proof}
The cup-product construction is \cite[Sections~1.7--1.8]{Goldman84}. Let $\omega_f$ denote the closed two-form of Rungi--Tamburelli in the trace normalization of \cite[Section~1.1]{Tamburelli26}, and let
$$
 \mathcal E(X,q)=\frac12\int_X|q|_h^2\,\dA_h
$$
be the energy functional in \cite[equation~(2.7)]{Tamburelli26}. By \cite[Theorem~1.1]{Tamburelli26}, $\omega_f=\omega_G$ on the entire bundle of cubic differentials. For clockwise rotation,
$$
\iota_{V_{\rot}}\omega_f=-2\,\dd \mathcal E
$$
by \cite[Corollary~2.2 and Proposition~4.3]{Tamburelli26}. Since the Pick differential is $q/\sqrt2$, the area identity in \cite[Section~1.1 and Remark~2.3]{Tamburelli26} gives
$$
2\mathcal E=A_B+2\pi\chi(S)=A_B-A_0=\mu.
$$
This proves \eqref{ch01:eq:hamiltonian}. By Cartan's formula and the closedness of $\omega_G$,
$$
\mathcal L_{V_{\rot}}\omega_G
=\dd \iota_{V_{\rot}}\omega_G+\iota_{V_{\rot}}\dd \omega_G=0.
$$
\end{proof}

\subsection{The Wang equation and area estimates}

\begin{prop}
\label{ch01:prop:wang}
The conformal factor satisfies
\begin{equation}
\Delta_\sigma u+2e^u-2Qe^{-2u}-2=0.
\label{ch01:eq:wang}
\end{equation}
Moreover,
\begin{equation}
\mu
=\int_X(e^u-1)\,\dA_\sigma
=\int_XQe^{-2u}\,\dA_\sigma,
\label{ch01:eq:area}
\end{equation}
and
\begin{equation}
u\geq0,\qquad
1\leq e^u\leq1+\|Q\|_\infty^{1/3}.
\label{ch01:eq:pointwise}
\end{equation}
In particular, $\mu\geq0$, with equality exactly on the zero section.
\end{prop}

\begin{proof}
The curvature identity in Proposition~\ref{ch01:prop:import-parametrization} is
$$
K_h-|q|_h^2=-1.
$$
The conformal curvature formula and the scaling of the norm of the cubic differential are
$$
K_h=e^{-u}\left(-1+\frac12\Delta_\sigma u\right),
\qquad
|q|_h^2=e^{-3u}Q.
$$
Substitution, followed by multiplication by $2e^u$, gives \eqref{ch01:eq:wang}.

The integral of $\Delta_\sigma u$ vanishes on the closed surface. Integrating \eqref{ch01:eq:wang} therefore gives
$$
\int_X(e^u-1)\,\dA_\sigma
=\int_XQe^{-2u}\,\dA_\sigma.
$$
Since $\dA_h=e^u \dA_\sigma$, the left side is $A_B-A_0$.

At a minimum of $u$, the Laplacian satisfies $\Delta_\sigma u\leq0$. A negative minimum would give
$$
\Delta_\sigma u=2Qe^{-2u}+2(1-e^u)>0
$$
at that point. Hence $u\geq0$.

At a maximum of $u$, equation \eqref{ch01:eq:wang} and $u\geq0$ give
$$
(e^u-1)^3\leq e^{2u}(e^u-1)\leq Q\leq\|Q\|_\infty.
$$
This proves \eqref{ch01:eq:pointwise}.

Both integrands in \eqref{ch01:eq:area} are nonnegative. If $\mu=0$, continuity gives $e^u-1=0$ everywhere, and the second integral gives $Q=0$. Conversely, $Q=0$ makes the second integral vanish.
\end{proof}

To quantify this growth, we compare the Blaschke area with the area determined by the cubic differential.

\begin{defn}
\label{ch01:def:flat-area}
The area of the singular flat metric associated to a holomorphic cubic differential is
\begin{equation}
A_{\flatarea}(q)
=\int_XQ^{1/3}\,\dA_\sigma
=\int_X|q|^{2/3}.
\label{ch01:eq:flat-metric-area}
\end{equation}
\end{defn}

For $\sigma=\rho^2|\dd z|^2$ and $q=f(\dd z)^3$, this definition reads
$$
Q^{1/3}\dA_\sigma
=(|f|^2\rho^{-6})^{1/3}\rho^2\,\dd x\,\dd y
=|f|^{2/3}\,\dd x\,\dd y.
$$
Thus $A_{\flatarea}$ is independent of the conformal background. The area comparison below applies to arbitrary holomorphic cubic differentials.

\begin{prop}
\label{ch01:prop:blaschke-comparison}
For every cubic differential,
$$
0\leq |q|_h^2<1,\qquad K_h<0,
$$
and
\begin{equation}
A_{\flatarea}(q)\leq A_B(X,q)\leq A_0+A_{\flatarea}(q).
\label{ch01:eq:flat-area}
\end{equation}
Consequently,
$$
\max\{0,A_{\flatarea}(q)-A_0\}\leq\mu(X,q)\leq A_{\flatarea}(q).
$$
\end{prop}

\begin{proof}
For $q\neq0$, the metric and area comparisons are
$$
h>2^{1/3}|q/\sqrt2|^{2/3},
\qquad
2^{1/3}\int_X|q/\sqrt2|^{2/3}
\leq A_B
\leq 2^{1/3}\int_X|q/\sqrt2|^{2/3}+2\pi|\chi(S)|.
$$
These comparisons hold for every holomorphic cubic differential on a closed surface of genus at least two by \cite[Lemmas~3.3--3.4 and Corollary~3.5]{OT20}; see also \cite[Lemmas~3.1--3.2]{LTW22}.

Since the Pick differential is $q/\sqrt2$,
$$
2^{1/3}|q/\sqrt2|^{2/3}=|q|^{2/3},
\qquad
2\pi|\chi(S)|=A_0.
$$
The area bounds therefore give \eqref{ch01:eq:flat-area}. The metric bound gives $e^u>Q^{1/3}$ wherever $q\neq0$, so
$$
|q|_h^2=Qe^{-3u}<1.
$$
At a zero of $q$, this squared norm is zero. The curvature identity gives $K_h=|q|_h^2-1<0$.

When $q=0$, we have $u=0$ by Proposition~\ref{ch01:prop:wang}, so $h=\sigma$. This also proves all the assertions in that case. Finally, subtracting $A_0$ from \eqref{ch01:eq:flat-area} and using $\mu\geq0$ gives the bounds for $\mu$.
\end{proof}

To make these bounds uniform over compact subsets of moduli space, we first recall the smooth dependence of the hyperbolic background on the curve.

\begin{lem}
\label{ch01:lem:hyperbolic-family}
In a smooth trivialization of a local holomorphic family of closed curves of genus $g\geq2$, the hyperbolic metrics $\sigma_t$ depend smoothly on the surface point and the real base coordinates $t$. On a compact subset of the smooth base, any finite collection of derivatives in the surface and parameter variables of $\sigma_t$, $\sigma_t^{-1}$, and $\dA_{\sigma_t}$ is bounded. The metrics are uniformly positive definite relative to an auxiliary metric independent of the parameters.
\end{lem}

\begin{proof}
Choose smooth compatible metrics $g_t$ near a base point $t_0$, with $g_{t_0}=\sigma_{t_0}$. The equation for $e^{2w}g_t$ to have curvature $-1$ is
$$
\mathcal H(t,w)=\Delta_{g_t}w+K_{g_t}+e^{2w}=0,
\qquad
D_w\mathcal H(t_0,0)=\Delta_{\sigma_{t_0}}+2.
$$
The linearization is an isomorphism from $C^{k+2,\alpha}(S)$ to $C^{k,\alpha}(S)$ for every $k\geq0$ and $0<\alpha<1$, by coercivity and elliptic regularity \cite[Sections~3 and~4.2]{Donaldson2008Analysis}. The hyperbolic metric exists uniquely at every point of the parameter chart. For each fixed $k$, we apply the implicit function theorem \cite[Section~5.4, Theorem~5.7 and Corollary~5.1]{Sideris2013ODE} at every such point; uniqueness makes each local solution map equal to the hyperbolic family. Thus this family is smooth on the parameter chart as a map into $C^{k+2,\alpha}(S)$ for every $k$, which proves joint smoothness in the surface and parameter variables. Compactness of the parameter set and of $S$ bounds the stated derivatives and gives uniform positive definiteness; the inverse metrics and area forms satisfy the same conclusion by smooth matrix operations.
\end{proof}

The flat-area comparison now becomes a coercive estimate in any smooth Hermitian norm on $\mathbb E_3$.

\begin{cor}
\label{ch01:cor:coercivity}
For each compact $K\subset\MG$ and each smooth Hermitian norm on $\mathbb E_3$ over a neighborhood of $K$, there are constants $c_K,C_K>0$ such that
\begin{equation}
c_K\|q\|^{2/3}-A_0
\leq\mu(X,q)\leq C_K\|q\|^{2/3}
\label{ch01:eq:coercivity}
\end{equation}
for every $X\in K$ and every $q\in(\mathbb E_3)_X$.
\end{cor}

\begin{proof}
A finite collection of orbifold charts over which the bundle is trivial covers $K$. In smooth markings, the hyperbolic metrics and their area forms vary smoothly by Lemma~\ref{ch01:lem:hyperbolic-family}. Since the local frames of $\mathbb E_3$ and the density $|q|^{2/3}$ vary continuously, including at zeros of $q$, the integral $A_{\flatarea}$ is continuous on the bundle.

Since the unit sphere bundle over $K$ is compact and the flat area of a nonzero holomorphic cubic differential is strictly positive, $A_{\flatarea}$ attains a positive minimum and a finite maximum on this sphere bundle, with
$$
0<c_K\leq C_K<\infty.
$$
The identity
$$
A_{\flatarea}(aq)=|a|^{2/3}A_{\flatarea}(q)
$$
gives
$$
c_K\|q\|^{2/3}\leq A_{\flatarea}(q)\leq C_K\|q\|^{2/3}.
$$
The estimate \eqref{ch01:eq:coercivity} now follows from Proposition~\ref{ch01:prop:blaschke-comparison}.
\end{proof}

\subsection{Linearization and level sets of $\mu$}

\begin{lem}
\label{ch01:lem:linearization}
The linearization of \eqref{ch01:eq:wang} in $u$ is
\begin{equation}
L_u=\Delta_\sigma+V_u,
\qquad
V_u=2e^u+4Qe^{-2u}.
\label{ch01:eq:linearization}
\end{equation}
It satisfies
\begin{equation}
L_u\geq\Delta_\sigma+2,
\qquad
e^{-u}L_u=\Delta_h+2+4|q|_h^2,
\label{ch01:eq:rescaled-linearization}
\end{equation}
where $2+4|q|_h^2\in[2,6)$.

For $0<\alpha<1$, the operator
$$
L_u\colon C^{2,\alpha}(X)\longrightarrow C^{0,\alpha}(X)
$$
is an isomorphism. For real $f$,
\begin{equation}
\|L_u^{-1}f\|_\infty\leq\frac12\|f\|_\infty,
\qquad
\|L_u^{-1}f\|_{L^2(\dA_\sigma)}
\leq\frac12\|f\|_{L^2(\dA_\sigma)}.
\label{ch01:eq:inverse}
\end{equation}
The Wang solution is unique. On each marked family chart, it is jointly smooth in the surface point, the complex structure, and the real fiber coordinates, including the zero section.
\end{lem}

\begin{proof}
Differentiating \eqref{ch01:eq:wang} gives \eqref{ch01:eq:linearization}. The bounds $e^u\geq1$ and $0\leq |q|_h^2<1$, together with $e^{-u}\Delta_\sigma=\Delta_h$, give \eqref{ch01:eq:rescaled-linearization}.

For a real smooth function $w$,
$$
\int_XwL_uw\,\dA_\sigma
=\int_X\bigl(|\nabla w|_\sigma^2+V_uw^2\bigr)\,\dA_\sigma
\geq2\|w\|_2^2.
$$
The operator is elliptic and self-adjoint on the closed surface, so this inequality and elliptic solvability give an inverse. Applied to $L_uw=f$, it yields the $L^2$ bound. At a positive maximum of $w$, $f\geq2w$; at a negative minimum, $f\leq2w$. These inequalities give the supremum bound. By Schauder regularity, $L_u\colon C^{k+2,\alpha}(X)\to C^{k,\alpha}(X)$ is an isomorphism at every finite order \cite[Sections~3 and~4.2]{Donaldson2008Analysis}.

If $u_1,u_2$ are two Wang solutions and $w=u_1-u_2$, then
$$
\Delta_\sigma w+C(x)w=0,\qquad
C(x)=2\int_0^1e^{u_2+tw}\,\dd t+4Q\int_0^1e^{-2(u_2+tw)}\,\dd t>0.
$$
By the maximum principle, $w=0$.

In a marked family, the smooth hyperbolic metrics of Lemma~\ref{ch01:lem:hyperbolic-family} and smooth frames of cubic differentials make the Wang equation a smooth map from the real parameters and $C^{k+2,\alpha}(S)$ to $C^{k,\alpha}(S)$. A solution exists at every parameter point by Proposition~\ref{ch01:prop:import-parametrization}, and its linearization $L_u$ is an isomorphism at every finite order. For each $k$, the implicit function theorem therefore applies at every parameter point. By uniqueness, its local solution maps are restrictions of the same Wang solution map, which is consequently smooth throughout the chart as a map into $C^{k+2,\alpha}(S)$. Since $k$ is arbitrary, the solution is jointly smooth, including at $q=0$, where $u=0$ and $L_u=\Delta_\sigma+2$.
\end{proof}

By applying the maximum principle to the radial derivative of the Wang equation, we obtain quantitative monotonicity of $\mu$.

\begin{prop}
\label{ch01:prop:radial}
For a nonzero $\psi\in H^0(X,K_X^3)$ and $q=a\psi$ with $a>0$, the function $v=a\partial_a u$ satisfies
\begin{equation}
0<v<\frac23,
\qquad
\frac23\mu<a\partial_a\mu<2\mu.
\label{ch01:eq:radial}
\end{equation}
The area derivative is
$$
a\partial_a\mu
=\int_X(2-2v)Qe^{-2u}\,\dA_\sigma.
$$
\end{prop}

\begin{proof}
Since $Q=a^2|\psi|_\sigma^2$, differentiation of \eqref{ch01:eq:wang} gives
$$
L_uv=4Qe^{-2u}.
$$
The right side is nonnegative and is positive on an open set. The operator
$$
-L_u=\diver_\sigma\nabla-V_u
$$
has positive definite principal part and zeroth-order coefficient $-V_u\leq-2$, as required by the maximum principles in \cite[Lecture~13, Theorems~1 and~3]{SimonPDE}. At a negative minimum of $v$, the equation would give $L_uv<0$, so $v\geq0$. If $v$ vanished somewhere, $-v$ would attain its nonnegative maximum $0$ at an interior point and would satisfy
$$
(-L_u)(-v)=L_uv=4Qe^{-2u}\geq0.
$$
By the strong maximum principle, $v$ would be constant, and hence identically zero, on the connected surface, contradicting the positivity of $Qe^{-2u}$ on an open set. Thus $v>0$.

For $z=\frac23-v$, the linearized equation gives
$$
L_uz
=\frac43\bigl(e^u-Qe^{-2u}\bigr)
=\frac43e^u(1-|q|_h^2)>0.
$$
At a nonpositive minimum of $z$, the left side would be nonpositive. Hence $z>0$, proving the first inequality in \eqref{ch01:eq:radial}.

Differentiation of \eqref{ch01:eq:area} under the integral gives
$$
a\partial_a\mu
=\int_Xa\partial_a(Qe^{-2u})\,\dA_\sigma
=\int_X(2-2v)Qe^{-2u}\,\dA_\sigma.
$$
The factor $2-2v$ lies strictly between $2/3$ and $2$. The weight $Qe^{-2u}$ is nonnegative, is positive on an open set, and has integral $\mu>0$. Integration proves the remaining inequalities.
\end{proof}

Together with coercivity, radial monotonicity gives equivariant diffeomorphisms from the sphere bundle to the positive level sets of $\mu$, whose circle quotients are the projective bundle.

\begin{cor}
\label{ch01:cor:levels}
Every positive value of $\mu$ is regular. Every complex line in a fiber of $\mathbb E_3$ meets each positive level set of $\mu$ in exactly one circle.

For a smooth orbifold Hermitian norm, $Y=S(\mathbb E_3)$ denotes the unit sphere bundle. For $\tau>0$ and $\zeta=(X,\psi)\in Y$, there is a unique smooth positive circle-invariant function $a_{\rad}(\tau,\zeta)$ satisfying
$$
\mu(X,a_{\rad}(\tau,\zeta)\psi)=\tau.
$$
The map
$$
\ell_\tau\colon Y\longrightarrow\mu^{-1}(\tau),
\qquad
\ell_\tau(X,\psi)=(X,a_{\rad}(\tau,X,\psi)\psi)
$$
is a circle-equivariant diffeomorphism. It identifies the circle quotient of $\mu^{-1}(\tau)$ with the projective bundle $\mathbb P(\mathbb E_3)$ as an orbifold stack.

For every finite $L>0$, the sublevel $\mu\leq L$ is locally a smooth disk bundle with boundary. Its restriction over a compact base set is compact.
\end{cor}

\begin{proof}
The function $Q=|ae^{\ii\theta}\psi|_\sigma^2$ is independent of $\theta$. By uniqueness in Lemma~\ref{ch01:lem:linearization}, $u$ and $\mu$ are invariant under multiplication of $q$ by $e^{\ii\theta}$.

The area bounds give
$$
0\leq\mu(X,a\psi)\leq a^{2/3}A_{\flatarea}(\psi)\longrightarrow0
\quad\text{as }a\downarrow0,
$$
and
$$
\mu(X,a\psi)\geq a^{2/3}A_{\flatarea}(\psi)-A_0\longrightarrow\infty
\quad\text{as }a\to\infty.
$$
By Proposition~\ref{ch01:prop:radial}, the function $a\mapsto\mu(X,a\psi)$ is strictly increasing for $a>0$. Together with the limits as $a\downarrow0$ and $a\to\infty$, this shows that each positive value is attained at a unique radius.

At $q\neq0$, we have, by Proposition~\ref{ch01:prop:radial},
$$
\dd \mu(R_{\rad})>\frac23\mu>0.
$$
Thus every positive value of $\mu$ is regular, and $a_{\rad}$ is jointly smooth in $\tau$, the base, and the sphere variables by the implicit function theorem; uniqueness gives its $S^1$-invariance. Since $q\mapsto q/\|q\|$ is a smooth inverse, $\ell_\tau$ is an equivariant diffeomorphism. On each lifted chart, the free scalar action on the sphere has the projective fiber as quotient, and the construction is equivariant under the finite uniformizing group. It therefore descends to the asserted orbifold stack quotient.

If $X\in K$ and $\mu(X,q)\leq L$, \eqref{ch01:eq:coercivity} gives
$$
\|q\|\leq\left(\frac{L+A_0}{c_K}\right)^{3/2}.
$$
The sublevel over $K$ is consequently closed in a compact disk bundle.

On a relatively compact base chart, a smooth disk-bundle parametrization can be constructed with a constant
$$
0<c<\min_{\zeta}a_{\rad}(L,\zeta).
$$
Let $k$ be a smooth nonnegative function on $[0,1]$ that vanishes near zero and satisfies $\int_0^1k(s)\,\dd s=1$. Define
$$
K_0(s)=\int_0^s k(v)\,\dd v,\qquad
r(s,\zeta)=cs+(a_{\rad}(L,\zeta)-c)K_0(s).
$$
Then $r(0,\zeta)=0$, $r(1,\zeta)=a_{\rad}(L,\zeta)$, and
$$
\partial_s r=c+(a_{\rad}(L,\zeta)-c)k(s)>0.
$$
The map
$$
(X,s\psi)\longmapsto(X,r(s,X,\psi)\psi)
$$
takes the unit disk bundle onto $\mu\leq L$. Near the zero section it is the linear map $q\mapsto cq$, so it is smooth there. Away from zero, the positive radial derivative gives a smooth inverse, including in boundary charts at $s=1$. The norm and the radius function are invariant under the finite local uniformizing group, so this local trivialization descends.
\end{proof}

\subsection{Expansion and symplectic orientation at the zero section}

The quadratic expansion of the area Hamiltonian determines the restriction of $\omega_G$ to the fiber tangent spaces along the zero section and its symplectic orientation.

\begin{prop}
\label{ch01:prop:zero-expansion}
For $q=a\psi$ with $a\in\R$, as $a\to0$,
$$
u(X,a\psi)=2a^2(\Delta_\sigma+2)^{-1}|\psi|_\sigma^2+O(a^4),
$$
and
\begin{equation}
\mu(X,a\psi)=a^2\|\psi\|_{L^2(\sigma)}^2+O(a^4).
\label{ch01:eq:area-taylor}
\end{equation}
The first expansion holds in $C^{k+2,\alpha}(S)$ for every $k\geq0$ and $0<\alpha<1$. Both remainders are uniform on compact parameter sets after any finite collection of derivatives in the base and $\psi$ variables. In particular,
$$
\Hess_{\R}\mu_{(X,0)}(\psi,\eta)
=2\operatorname{Re}\langle\psi,\eta\rangle_{L^2(\sigma)}.
$$
\end{prop}

\begin{proof}
Put $t=a^2$. The equation
$$
\Delta_\sigma u+2e^u-2t|\psi|_\sigma^2e^{-2u}-2=0
$$
has derivative $\Delta_\sigma+2$ in $u$ at $t=0$. By the implicit function theorem, there is a solution for real $t$ near zero in $C^{2,\alpha}(S)$, with the base and $\psi$ as additional parameters. On a sufficiently small parameter neighborhood, its linearization remains coercive. Elliptic regularity makes each solution smooth on $S$, and the pointwise implicit-function argument of Lemma~\ref{ch01:lem:linearization} gives smooth parameter dependence at every finite order on this neighborhood. Its derivative at zero is
$$
\left.\partial_tu\right|_{t=0}
=2(\Delta_\sigma+2)^{-1}|\psi|_\sigma^2.
$$
By Taylor's theorem, the first expansion holds with the stated parameter bounds. Substitution into \eqref{ch01:eq:area} gives the expansion of $\mu$. Uniformity on the unit sphere in a local fiber frame yields $\mu(X,q)=\|q\|_{L^2(\sigma_X)}^2+O(\|q\|^4)$; polarization gives the Hessian.
\end{proof}

The Hessian in Proposition~\ref{ch01:prop:zero-expansion}, combined with the Hamiltonian identity, determines the restriction of $\omega_G$ to the normal bundle. The restriction along the zero section is obtained from the principal $\mathrm{SL}(2,\R)$ embedding.

\begin{prop}
\label{ch01:prop:normal-form}
The restriction to the zero section is
\begin{equation}
\omega_0:=i_0^*\omega_G=2\omega_{\WP}.
\label{ch01:eq:fuchsian-wp-form}
\end{equation}
The base and normal tangent spaces are symplectically orthogonal at the zero section. For normal vectors $v,w$,
$$
\omega_G(Iv,w)=2\operatorname{Re}\langle v,w\rangle_{L^2(\sigma)}.
$$
In an $L^2$-unit complex normal coordinate $z=x+\ii y$, the quadratic moment and the restriction of $\omega_G$ are
\begin{equation}
\mu_2=x^2+y^2,\qquad
\omega_{\mathrm{nor}}=-2\,\dd x\wedge \dd y.
\label{ch01:eq:normal-form}
\end{equation}
The Goldman orientation is $(-1)^r$ times the product of the complex orientations of the base and the fiber. The Gaussian integral on each complex normal line, taken with its symplectic orientation, is $2\pi/s$ for every $s>0$.
\end{prop}

\begin{proof}
We denote by $\omega_{\mathrm{ABG}}$ and $\omega_{\mathrm{ABG}}^{\tr_{\R^2}}$ the Atiyah--Bott--Goldman forms on the Fuchsian locus defined by the Killing form of $\mathfrak{sl}_2$ and by the trace form in its standard real two-dimensional representation, respectively. By \cite[Proposition~6.1]{Kawazumi25},
$$
2\omega_{\WP}
=\omega_{\mathrm{ABG}}
=4\omega_{\mathrm{ABG}}^{\tr_{\R^2}}.
$$
Under the principal representation $\iota\colon\mathfrak{sl}_2\to\mathfrak{sl}_3$, the element $H=\operatorname{diag}(1,-1)$ acts on $\Sym^2(\R^2)$ as
$$
\iota(H)=\operatorname{diag}(2,0,-2),
$$
and therefore
$$
\frac{\tr_{\R^3}(\iota(H)^2)}{\tr_{\R^2}(H^2)}
=\frac82=4.
$$
Invariant symmetric bilinear forms on the simple Lie algebra $\mathfrak{sl}_2$ are proportional, so
$$
\iota^*\tr_{\R^3}=4\tr_{\R^2}.
$$
The cup-product construction is linear in the invariant pairing, so
$$
 i_0^*\omega_G
 =4\omega_{\mathrm{ABG}}^{\tr_{\R^2}}
 =2\omega_{\WP}.
$$

The twist coordinate $\widehat{\tau}_j$ in \cite[equation~(1)]{Kawazumi25} is related to the twist coordinate in Subsection~\ref{ch01:setup:cubic} by $\vartheta_j=-\widehat{\tau}_j$. Consequently,
$$
\dd \widehat{\tau}_j\wedge \dd \ell_j
=-\dd \vartheta_j\wedge \dd \ell_j
=\dd \ell_j\wedge \dd \vartheta_j.
$$
This fixes the sign in \eqref{ch01:eq:fuchsian-wp-form}.

At a zero-section point, both $V_{\rot}$ and $\dd \mu$ vanish. Differentiating \eqref{ch01:eq:hamiltonian} in a normal direction $v$, and evaluating on a normal vector $w$, gives
$$
\omega_G(-Iv,w)
=-\Hess_{\R}\mu(v,w).
$$
By Proposition~\ref{ch01:prop:zero-expansion},
$$
\omega_G(Iv,w)=2\operatorname{Re}\langle v,w\rangle_{L^2(\sigma)}.
$$
In particular,
$$
\omega_G(v,Iv)=-2\|v\|_{L^2(\sigma)}^2,
$$
which proves \eqref{ch01:eq:normal-form}.

The circle fixes the base tangent space and rotates the normal space. For a base vector $b$ and a normal vector $v$, invariance gives
$$
\omega_G(b,v)=\omega_G(b,e^{-\ii\theta}v)
$$
for every $\theta$. Averaging over the circle makes the right side zero, proving orthogonality.

A complex-oriented base basis is positive for $2\omega_{\WP}$. In an $L^2$-unit complex fiber basis, every normal symplectic block has value $-2$. Hence the top exterior power of $\omega_G$ has sign $(-1)^r$ in the product complex orientation at zero. Nondegeneracy and continuity preserve this sign along each radial segment, proving the assertion on all of $\HG$.

Finally, \eqref{ch01:eq:normal-form} gives the positive symplectic density $2\,\dd x\,\dd y$ when expressed in standard Euclidean coordinates. Thus
$$
\int_{\R^2}2e^{-s(x^2+y^2)}\,\dd x\,\dd y
=4\pi\int_0^\infty \rho e^{-s\rho^2}\,\dd\rho
=\frac{2\pi}{s}.
$$
\end{proof}

\subsection{Polynomial estimates in smooth marked families}
\label{ch01:setup:parameters}

Let $K\subset\MG$ be compact. We use a finite cover by relatively compact lifted family charts, with smooth markings by $S$, smooth local frames of $\mathbb E_3$, and a smooth Hermitian fiber norm. These data extend to neighborhoods of the compact chart subsets covering $K$.

In one such chart, let $t$ be real base coordinates and $z$ real linear fiber coordinates, and write $\lambda=(t,z)$. We choose a finite coordinate atlas on $S$, with nested coordinate balls and an auxiliary metric, independently of $\lambda$. The norm $\|\cdot\|_{k,\theta}$ is the resulting $C^{k,\theta}(S)$ norm: it includes the supremum norms of coefficient derivatives through order $k$, and the $\theta$-H\"older seminorm of derivatives of order $k$. Tensor norms use local frames independent of $\lambda$. We write $\|\cdot\|_0$ for the supremum norm.

Derivatives on $S$ may equivalently be taken with an auxiliary connection independent of $\lambda$. At any finite order, the coordinate, frame, and marking changes have bounded derivatives on the compact chart subsets. The hyperbolic metric is denoted by $\sigma_t$. The following background estimates provide the uniform elliptic constants used in the growth argument.

\begin{lem}
\label{ch01:lem:background-estimates}
With the family data of Subsection~\ref{ch01:setup:parameters}, the background Laplacians satisfy, uniformly in the base parameter,
\begin{equation}
\|v\|_{2,1/2}
\leq C\bigl(\|\Delta_{\sigma_t}v\|_{0,1/2}+\|v\|_0\bigr),
\label{ch01:eq:schauder-low}
\end{equation}
and, for each integer $k\geq1$,
\begin{equation}
\|v\|_{k+2,1/2}
\leq C_k\bigl(\|\Delta_{\sigma_t}v\|_{k,1/2}
+\|v\|_{k+1,1/2}\bigr).
\label{ch01:eq:schauder-high}
\end{equation}
There is $\delta_0>0$, independent of the base parameter, such that for $0<\delta\leq\delta_0$,
\begin{equation}
\|v\|_{0,1/2}
\leq C\delta^{3/2}\|v\|_{2,1/2}
+C\delta^{-1/2}\|v\|_0.
\label{ch01:eq:interpolation}
\end{equation}
These constants are independent of the fiber coordinate.
\end{lem}

\begin{proof}
By Lemma~\ref{ch01:lem:hyperbolic-family}, the background Laplacians on the compact family have uniformly bounded coefficients and a uniform ellipticity constant. Interior Schauder estimates on a finite collection of nested coordinate balls give \eqref{ch01:eq:schauder-low} \cite[Lecture~12, Theorem~2]{SimonPDE}. Applying them to derivatives of order $k$ in the coordinates on $S$ gives \eqref{ch01:eq:schauder-high}, since the commutator with $\Delta_{\sigma_t}$ involves derivatives of $v$ of order at most $k+1$.

For the interpolation estimate, we have, by Taylor's formula,
$$
\|Dv\|_0\leq C\bigl(\delta\|D^2v\|_0+\delta^{-1}\|v\|_0\bigr).
$$
For points at distance at most $\delta$, the $1/2$-H\"older quotient is bounded by $C\delta^{1/2}\|Dv\|_0$; for points farther apart, it is at most $2\delta^{-1/2}\|v\|_0$. A finite cover gives \eqref{ch01:eq:interpolation}, with $\delta_0$ independent of the parameters.
\end{proof}

We use the pointwise bound to begin an induction on derivatives in the surface and parameter variables. At each step, the interpolation estimate gives polynomial dependence on $\|q\|$.

\begin{prop}
\label{ch01:prop:scalar-growth}
With the family data of Subsection~\ref{ch01:setup:parameters}, for every real $c$, integer $k\geq0$, and parameter multiindex $\beta$, there are finite constants $C>0$ and $M\geq0$ such that
\begin{equation}
\|\partial_\lambda^\beta u\|_{k,1/2}
+\|\partial_\lambda^\beta e^{cu}\|_{k,1/2}
\leq C(1+\|q\|)^M
\label{ch01:eq:scalar-growth}
\end{equation}
for every $X\in K$ and every $q\in(\mathbb E_3)_X$.

The constants may depend on $g,K,c,k,\beta$ and the family data of Subsection~\ref{ch01:setup:parameters}. A single bound may be chosen for any finite collection of derivatives. The conclusion includes coefficients of $h,h^{-1}$, their induced tensor metrics, and inverse area factors.
\end{prop}

\begin{proof}
On a compact lifted chart, write $q(t,z;x)=\sum_jz_je_j(t;x)$ in a smooth real frame and put $T=1+|z|$. The coordinate norm and the Hermitian fiber norm are uniformly equivalent. Since $Q=|q|_{\sigma_t}^2$ is quadratic in $z$,
$$
\|\partial_\lambda^\gamma Q\|_{k,1/2}\leq C(1+|z|)^2
$$
for every prescribed mixed derivative; a fiber derivative of order greater than two vanishes. By Proposition~\ref{ch01:prop:wang},
$$
\|u\|_0\leq C\log(2+T),\qquad
1\leq e^u\leq CT^{2/3},\qquad 0<e^{-u}\leq1.
$$

We first bound derivatives in the coordinates on $S$. By the product inequality for H\"older seminorms and the mean-value theorem,
$$
[e^u]_{1/2}\leq CT^{2/3}[u]_{1/2},\qquad
[e^{-2u}]_{1/2}\leq2[u]_{1/2},
$$
and hence
$$
\|\Delta_{\sigma_t}u\|_{0,1/2}
=\|2Qe^{-2u}+2-2e^u\|_{0,1/2}
\leq CT^2(1+\|u\|_{0,1/2}).
$$
Applying \eqref{ch01:eq:schauder-low} and \eqref{ch01:eq:interpolation}, we obtain
$$
\|u\|_{2,1/2}
\leq CT^2+CT^2\delta^{3/2}\|u\|_{2,1/2}
+C(T^2\delta^{-1/2}+1)\|u\|_0.
$$
Choosing $\delta=c_0T^{-4/3}$ with $c_0>0$ sufficiently small absorbs the middle term and gives
$$
\|u\|_{2,1/2}\leq CT^{8/3}(1+\log(2+T))\leq C'T^4.
$$

Each derivative of $e^{cu}$ is $e^{cu}$ times a finite polynomial in derivatives of $u$, and
$$
\sup e^{cu}\leq
\begin{cases}
C_cT^{2c/3},&c\geq0,\\
1,&c<0.
\end{cases}
$$
By the H\"older product inequality, polynomial bounds for $u$ give polynomial bounds for $e^{cu}$. If $u$ is controlled in $C^{k+1,1/2}$, the Wang equation bounds $\Delta_{\sigma_t}u$ in $C^{k,1/2}$, and \eqref{ch01:eq:schauder-high} bounds $u$ in $C^{k+2,1/2}$. This proves the bounds for $u$, $e^{cu}$, and $V_u$ at every order of differentiation on $S$.

We now induct on parameter order, retaining the estimates for all derivatives on $S$ at each step. For $\beta\neq0$, put $u_\beta=\partial_\lambda^\beta u$. Differentiating the Wang equation gives
\begin{equation}
\begin{aligned}
L_uu_\beta&=f_\beta,\\
f_\beta&=-\sum_{0<\gamma\leq\beta}\binom{\beta}{\gamma}
(\partial_\lambda^\gamma\Delta_{\sigma_t})u_{\beta-\gamma}
+\mathcal R_\beta.
\end{aligned}
\label{ch01:eq:mixed-linearization}
\end{equation}
Here $\partial_\lambda^\gamma\Delta_{\sigma_t}$ means differentiation of its coefficients. The remainder $\mathcal R_\beta$ is a finite sum of products of $e^u$, $e^{-2u}$, derivatives of $Q$ of order at most $|\beta|$, and parameter derivatives of $u$ of order strictly below $|\beta|$. This follows from the chain rule after the terms $2e^uu_\beta+4Qe^{-2u}u_\beta$ are moved to the left. In particular,
$$
L_u(\partial_\lambda u)
=-(\partial_\lambda\Delta_{\sigma_t})u
+2(\partial_\lambda Q)e^{-2u}.
$$
The induction hypothesis bounds $f_\beta$ in every $C^{k,1/2}(S)$ norm. The differentiated Laplacian requires two additional derivatives on $S$ of a lower parameter derivative, which are available at the preceding stage.

The supremum estimate in \eqref{ch01:eq:inverse} gives $\|u_\beta\|_0\leq\tfrac12\|f_\beta\|_0$. Rewriting the equation as $\Delta_{\sigma_t}u_\beta=f_\beta-V_uu_\beta$ and applying \eqref{ch01:eq:schauder-low} gives
$$
\|u_\beta\|_{2,1/2}
\leq C\bigl(\|f_\beta\|_{0,1/2}+\|u_\beta\|_0\bigr)
+A(T)\|u_\beta\|_{0,1/2},
$$
where $A(T)\geq1$ is a polynomial bound for a constant multiple of $\|V_u\|_{0,1/2}$. In \eqref{ch01:eq:interpolation}, choose
$$
\delta=\min\{\delta_0,c_1A(T)^{-2/3}\}.
$$
For sufficiently small $c_1$, absorption yields
$$
\|u_\beta\|_{2,1/2}
\leq C\left(\|f_\beta\|_{0,1/2}
+(1+A(T)^{4/3})\|u_\beta\|_0\right).
$$
This is polynomial in $T$. By \eqref{ch01:eq:schauder-high}, the bounds hold at every higher order of differentiation on $S$, completing the parameter induction. The chain rule gives the same bounds for $\partial_\lambda^\beta e^{cu}$.

The tensor estimates follow from $h=e^u\sigma_t$, $h^{-1}=e^{-u}\sigma_t^{-1}$ and the corresponding tensor powers. Changes of fiber frame have the form $z'=L(t)z$, with bounded derivatives of $L,L^{-1}$ on compact overlaps; changes of marking have bounded derivatives as well. The chain rule preserves polynomial growth under these changes. Finite maxima over the charts give the asserted uniform bounds.
\end{proof}

\subsection{Flat connections and the Goldman pairing}

We express the coordinate components of $\omega_G$ through derivatives of the Blaschke connection, to which the bounds in Proposition~\ref{ch01:prop:scalar-growth} apply.

\begin{defn}
\label{ch01:def:affine-connection}
In a marked chart, we choose a smooth positive area form $\nu$ on $S$, independent of the parameters. For the flat connection $D_{\aff}$ of \eqref{ch01:eq:affine-flat-connection} on $E_{\aff}=TS\oplus(S\times\R)$, we define
\begin{equation}
f_{\vol}=\log\frac{\dA_h}{\nu},\qquad
D_{\vol}=D_{\aff}-\frac13\dd_Sf_{\vol}\,\Id_{E_{\aff}}.
\label{ch01:eq:volume-preserving-connection}
\end{equation}
Here $\dd_S$ denotes exterior differentiation on the marked surface. The area form $\nu$ defines the parameter-independent volume form
$$
\vol_0((Z,0),(W,0),(0,1))=\nu(Z,W)
$$
on $E_{\aff}$.
\end{defn}

For each real coordinate $\lambda^A$, put
$$
\alpha_A=\partial_{\lambda^A}\nabla^B,
\qquad
\dot D_A=\partial_{\lambda^A}D_{\vol}.
$$
The forms $\dot D_A$ represent tangent vectors under the holonomy map. The trace pairing reduces to the variations of the Blaschke connection on $TS$.

\begin{prop}
\label{ch01:prop:affine-pairing}
The connection $D_{\vol}$ is flat, preserves $\vol_0$, and has the $\mathrm{SL}(3,\R)$-holonomy associated to $(X,q)$. Each $\dot D_A$ is a covariantly closed, trace-free one-form representing the corresponding coordinate tangent vector under the holonomy differential. The coordinate component of $\omega_G$ is
\begin{equation}
\begin{aligned}
W_{AB}
&:=\omega_G(\partial_{\lambda^A},\partial_{\lambda^B})\\
&=\int_S\tr_3(\dot D_A\wedge \dot D_B)\\
&=\int_S\tr_2(\alpha_A\wedge\alpha_B).
\end{aligned}
\label{ch01:eq:trace-pairing}
\end{equation}

With the family data of Subsection~\ref{ch01:setup:parameters}, all mixed derivatives in the surface and parameter variables of $\nabla^B,D_{\vol}$, and all parameter derivatives of $W_{AB}$, have polynomial bounds in $1+\|q\|$, uniformly over $K$. The same holds for coefficients of exterior powers of $\omega_G$ and their contractions with coordinate vector fields.
\end{prop}

\begin{proof}
By Proposition~\ref{ch01:prop:import-parametrization}, $D_{\aff}$ is flat and $\nabla^B$ preserves $\dA_h$. The off-diagonal blocks of $D_{\aff}$ have zero trace, so $D_{\aff}$ preserves the volume form $\vol_\lambda=e^{f_{\vol}}\vol_0$. In a $\vol_0$-unit frame, writing $D_{\aff}=\dd_S+\mathcal A_{\aff}$, preservation of $\vol_\lambda$ means $\tr_3\mathcal A_{\aff}=\dd_Sf_{\vol}$. The scalar gauge
$$
G=e^{-f_{\vol}/3}\Id_{E_{\aff}}
$$
satisfies $G^*\vol_\lambda=\vol_0$ and
$$
G^{-1}D_{\aff}G=D_{\vol}.
$$
Thus $D_{\vol}$ is flat, preserves $\vol_0$, and has the prescribed holonomy. Differentiating flatness and preservation of $\vol_0$ gives
$$
\dd_{D_{\vol}}\dot D_A=0,\qquad \tr_3\dot D_A=0.
$$
The holonomy differential is represented by $[\dot D_A]$, so the first integral in \eqref{ch01:eq:trace-pairing} follows from Proposition~\ref{ch01:prop:import-goldman}. A change of gauge adds a covariant exact form; its pairing with a closed representative integrates to zero on $S$. The formula is therefore independent of the chosen gauge and marking within its isotopy class.

Put $T_A=\partial_{\lambda^A}D_{\aff}$ and $s_A=\dd_S\partial_{\lambda^A}f_{\vol}$. The upper-right block of $D_{\aff}$ is parameter-independent, and hence
$$
T_A=
\begin{pmatrix}
\alpha_A&0\\
\partial_{\lambda^A}h&0
\end{pmatrix},\qquad
\tr_3T_A=s_A,\qquad
\dot D_A=T_A-\tfrac13s_A\Id.
$$
Block multiplication gives
$$
\tr_3(\dot D_A\wedge \dot D_B)
=\tr_2(\alpha_A\wedge\alpha_B)-\tfrac13s_A\wedge s_B.
$$
The last term integrates to zero because
$$
s_A\wedge s_B=\dd_S\bigl((\partial_{\lambda^A}f_{\vol})s_B\bigr).
$$
This proves the second integral identity.

For the estimates, in real coordinates on $S$ we have
$$
(\nabla^B)^k_{ij}
=(\nabla^\sigma)^k_{ij}
+\frac12\left(\delta_i^k\partial_ju+\delta_j^k\partial_iu
-\sigma_{ij}\sigma^{k\ell}\partial_\ell u\right)
+\frac1{\sqrt2}e^{-u}\sigma^{k\ell}(\operatorname{Re}q)_{ij\ell}.
$$
Every mixed derivative of this expression is bounded by a power of $1+\|q\|$ by Proposition~\ref{ch01:prop:scalar-growth}. The same argument applies to $D_{\vol}$, since $h=e^u\sigma$ and $f_{\vol}=u+\log(\dA_\sigma/\nu)$. Differentiating \eqref{ch01:eq:trace-pairing} under the integral over the compact marked surface gives the bounds for $W_{AB}$. Exterior products and contractions involve finite sums of products of these coefficients, so their derivatives have the same form of bound. Changes of marking and frame preserve the estimates, as in Proposition~\ref{ch01:prop:scalar-growth}.
\end{proof}

\subsection{Coefficient bounds for fiber integration}

In a complex-oriented real fiber frame in Subsection~\ref{ch01:setup:parameters}, write
$$
\left[\frac{\omega_G^{r+k}}{(r+k)!}\right]_{\text{vertical degree }2r}
=
\sum_{|I|=2k}
F_{k,I}(t,z)\,
\dd t^I\wedge \dd z^1\wedge\cdots\wedge \dd z^{2r},
\qquad 0\leq k\leq d.
$$

\begin{cor}
\label{ch01:prop:coefficient-growth}
With the family data of Subsection~\ref{ch01:setup:parameters}, the coefficients of $\omega_G^j/j!$, $0\leq j\leq N$, and all their finite parameter derivatives have polynomial bounds in $1+\|q\|$, uniformly over $K$. In particular, for every parameter multiindex $\beta$,
$$
|\partial_\lambda^\beta F_{k,I}(t,z)|\leq C(1+\|q\|)^M,
\qquad 0\leq k\leq d.
$$
The constants depend on the derivative order and the compact family data. The same bounds hold in the Goldman orientation, whose conversion factor is $(-1)^r$.
\end{cor}

\begin{proof}
Each coefficient is a finite sum of products of coordinate components of $\omega_G$. The bounds therefore follow from Proposition~\ref{ch01:prop:affine-pairing} and the product rule.
\end{proof}

\section{Bundles of cubic differentials on stable curves}\label{ch02:sec:stable-cubics}

The bundle of holomorphic cubic differentials extends across the Deligne--Mumford compactification $\bM_g$, and its pointed analogue allows poles of order at most two at the markings. These bundles place the localization coefficients in a family indexed by stable pairs $(g,n)$. Their behavior under clutching is described by nodal residues, and their characteristic classes admit an expression in Hodge and logarithmic $\kappa$ classes.

All algebraic constructions are over~$\C$, and characteristic-class identities hold in the rational Chow rings of Deligne--Mumford stacks; integration means the stack-theoretic degree of a zero-cycle, with stabilizer weights, while degrees in the intersection formulas mean complex codimensions. For nonnegative integers~$g,n$ satisfying~$2g-2+n>0$, the universal pointed curve and its marking divisor are
$$
 \pi_{g,n}\colon\overline{\cC}_{g,n}\longrightarrow\bM_{g,n},
 \qquad D_{\mathrm{mark}}=\sum_{i=1}^n p_i,
 \qquad D_{g,n}=3g-3+n.
$$
We use the cotangent line classes~$\psi_i=c_1(p_i^*\omega_{\pi_{g,n}})$ and the logarithmic $\kappa$ classes
\begin{equation}\label{eq:marked-kappa-definition}
 \kappa_m=(\pi_{g,n})_*c_1\bigl(\omega_{\pi_{g,n}}(D_{\mathrm{mark}})\bigr)^{m+1}.
\end{equation}
In particular,~$\kappa_m=f_*\psi_{n+1}^{m+1}$ for the morphism~$f\colon\bM_{g,n+1}\to\bM_{g,n}$ that forgets the last marking and stabilizes the curve. On~$\bM_g$, these classes agree with those in~\eqref{eq:intro-stable-cubic-bundle}.

\subsection{The bundle of cubic differentials}
\label{ch02:sec:cubic-extension}

We define
\begin{equation}\label{eq:marked-cubic-definition}
 \mathbb F_{g,n}=(\pi_{g,n})_*\omega_{\pi_{g,n}}^{\otimes3}(2D_{\mathrm{mark}}),
 \qquad r_{g,n}=5g-5+2n.
\end{equation}
For~$n=0$, the universal curve is denoted by~$\bar\pi\colon\bC_g\to\bM_g$, and
\begin{equation}\label{ch02:eq:stable-cubic-bundle}
 \bE_3=\mathbb F_{g,0}=\bar\pi_*\omega_{\bar\pi}^{\otimes3}.
\end{equation}

\begin{prop}\label{ch02:prop:cubic-extension}
For every stable~$(g,n)$, the sheaf~$\mathbb F_{g,n}$ is locally free of rank~$r_{g,n}$, and
$$
 R(\pi_{g,n})_*\omega_{\pi_{g,n}}^{\otimes3}(2D_{\mathrm{mark}})
 \simeq\mathbb F_{g,n}[0].
$$
For every scheme~$S\to\bM_{g,n}$, with pointed stable family~$f\colon C\to S$ and marking divisor~$D_S$, the line bundle~$L_f=\omega_f^{\otimes3}(2D_S)$ satisfies~$R^if_*L_f=0$ for~$i>0$, and~$f_*L_f$ is locally free of rank~$r_{g,n}$. For every morphism~$h\colon T\to S$, with induced maps~$u\colon C_T\to C$ and~$f_T\colon C_T\to T$, the canonical maps
\begin{equation}\label{ch02:eq:ordinary-cubic-base-change}
 u^*\omega_f\xrightarrow{\sim}\omega_{f_T},
 \qquad h^*f_*L_f\xrightarrow{\sim}(f_T)_*L_{f_T}
\end{equation}
are isomorphisms compatible with composition and descent. Moreover,~$\bE_3|_{\cM_g}\simeq\mathbb E_3$.
\end{prop}

\begin{proof}
A pointed stable family is proper, flat, and of finite presentation, with Gorenstein geometric fibers. Its relative dualizing sheaf is invertible and commutes with arbitrary base change~\cite[Tag~0E6R]{StacksR1}; the disjoint marked sections are relative Cartier divisors and also pull back under base change.

The marking divisor on each geometric fiber~$C_s$ is~$D_s=D_S|_{C_s}$. By stability,~$\omega_{C_s}(D_s)$ is ample, so its inverse square has negative degree on the normalization of every irreducible component and admits no nonzero global sections. By Serre duality and Riemann--Roch,
\begin{equation}\label{ch02:eq:stable-cubic-vanishing}
 H^1(C_s,L_f|_{C_s})^\vee
 =H^0\bigl(C_s,\omega_{C_s}(D_s)^{-2}\bigr)=0,
 \qquad h^0(C_s,L_f|_{C_s})=5g-5+2n;
\end{equation}
see~\cite[Tags~0BS4 and~0BS6]{StacksR1}. The complex~$Rf_*L_f$ is perfect, and the canonical map
\begin{equation}\label{ch02:eq:derived-base-change}
 Lh^*Rf_*L_f\xrightarrow{\sim}R(f_T)_*Lu^*L_f
\end{equation}
is an isomorphism for arbitrary~$h$~\cite[Tag~0CTM, part~(1)]{StacksR1}. Since every geometric fiber has cohomology only in degree zero, of dimension~$r_{g,n}$, we have $Rf_*L_f\simeq(f_*L_f)[0]$ with~$f_*L_f$ locally free of rank~$r_{g,n}$ by cohomology and base change~\cite[Tag~0E6A]{StacksR1}. Taking degree-zero cohomology in~\eqref{ch02:eq:derived-base-change} proves~\eqref{ch02:eq:ordinary-cubic-base-change}, with compatibility under composition and descent inherited from the canonical derived base-change maps. On the unmarked smooth locus,~$\omega_f=K_f$, so the restriction is~$\mathbb E_3$.
\end{proof}

For a vector bundle~$W$, the total characteristic class used below is~$c_{-t}(W)^{-1}=\prod_\alpha(1-t\alpha)^{-1}$, where the~$\alpha$ are formal Chern roots. The geometric localization expression is
\begin{equation}\label{ch02:eq:characteristic-expression}
 \frac{\exp([\omega_0])}
 {\prod_{j=1}^{r}(s/(2\pi)-x_j(\mathbb E_3))},\qquad s>0,
\end{equation}
with~$\omega_0=2\omega_{\WP}$ and~$c_1=(\ii/(2\pi))\tr F$. We derive this expression by integration over the fibers of $\mathbb E_3$ in Section~\ref{ch03:sec:fiber-integration}.

\subsection{Clutching and residues at the nodes}
\label{ch02:sec:cubic-residues}

Let $\Gamma$ be a connected stable graph with no legs, with vertex genera $g_v\geq0$, half-edge set $H(\Gamma)$, a fixed-point-free involution $h\mapsto\bar h$, and an incidence map from half-edges to vertices. Write $E(\Gamma)$ for the set of pairs $\{h,\bar h\}$, put $e_\Gamma=|E(\Gamma)|$, and let $n_v$ be the valence of $v$. The stability and genus conditions are
$$
2g_v-2+n_v>0,
\qquad
g=\sum_vg_v+e_\Gamma-|V(\Gamma)|+1.
$$
Set
$$
B_\Gamma=\prod_{v\in V(\Gamma)}\bM_{g_v,n_v},
\qquad
\xi_\Gamma\colon B_\Gamma\longrightarrow\bM_g.
$$
The half-edge markings on each factor are labeled, and $\xi_\Gamma$ glues paired markings. On the pulled-back universal vertex family $\pi_v$, let $D_v$ be their sum and define
$$
\mathbb F_v=(\pi_v)_*\bigl(\omega_{\pi_v}^{\otimes 3}(2D_v)\bigr).
$$

A cubic differential on a branch with a pole of order at most three has a local expression $a(z)(\dd z/z)^3$; its leading coefficient $a(0)$ is the cubic residue at that branch. The residue bundle is
\begin{equation}
\cR_\Gamma=
\left\{(R_h)\in\OO_{B_\Gamma}^{H(\Gamma)}
\ \middle|\ R_h+R_{\bar h}=0\text{ for every edge}\right\}.
\label{ch02:eq:residue-bundle}
\end{equation}
Let $G_\Gamma=\Aut(\Gamma)$, including exchanges of the branches of a loop, and put $Q_\Gamma=[B_\Gamma/G_\Gamma]$.

For a vector bundle $W$ of rank $n$, we use the characteristic polynomial
\begin{equation}
D_u(W)=\sum_{i=0}^{n}(-1)^ic_i(W)u^{n-i}
=\prod_{j=1}^{n}(u-x_j(W)).
\label{ch02:eq:chern-polynomial}
\end{equation}
All roots in this notation are formal.

\begin{prop}
\label{ch02:prop:cubic-residues}
Each $\mathbb F_v$ is locally free of rank $5g_v-5+2n_v$, and
$$
R(\pi_v)_*\bigl(\omega_{\pi_v}^{\otimes3}(2D_v)\bigr)\simeq\mathbb F_v[0].
$$
There is a canonical exact sequence
\begin{equation}
0\longrightarrow\bigoplus_v\mathbb F_v
\longrightarrow\xi_\Gamma^*\bE_3
\xrightarrow{\operatorname{Res}}\cR_\Gamma
\longrightarrow0.
\label{ch02:eq:residue-sequence}
\end{equation}
The bundles and maps commute with arbitrary scheme base change and are $G_\Gamma$-equivariant. The Chern classes of the descended residue bundle vanish in positive degree in $A^*(Q_\Gamma)_{\mathbb Q}$ and in real cohomology. Consequently,
\begin{equation}
\xi_\Gamma^*D_u(\bE_3)
=u^{e_\Gamma}\prod_vD_u(\mathbb F_v),
\label{ch02:eq:clutching-denominator}
\end{equation}
and this identity descends to $Q_\Gamma$.
\end{prop}

\begin{proof}
Let $S\to B_\Gamma$ be a scheme and $Y_v\to S$ its vertex families. Gluing the paired sections gives a stable family $f\colon X\to S$, and the finite map $\nu\colon\bigsqcup_vY_v\to X$ separates precisely these designated nodes, leaving any additional nodes in a vertex fiber as internal nodes of $Y_v$. All these constructions commute with base change.

At a designated node, an étale local model over a base ring $A$ is
$$
A[x,y]/(xy)
\simeq
\left\{(a(x),b(y))\in A[x]\oplus A[y]\ \middle|\ a(0)=b(0)\right\}.
$$
The equality follows by writing an element uniquely as a constant plus a multiple of $x$ and a multiple of $y$. This description is compatible with arbitrary base change $A\to A'$.

The dualizing generator pulls back to
$$
\eta=\left(\frac{\dd x}{x},-\frac{\dd y}{y}\right);
$$
this is the local nodal description in \cite[Section~1, pp.~2--3]{Wolpert2012}. The invertibility and base-change compatibility of the relative dualizing sheaves follow from \cite[Tag 0E6R]{StacksR1}. The local description gives the intrinsic identification
\begin{equation}
\nu^*\omega_f|_{Y_v}\simeq\omega_{Y_v/S}(D_v).
\label{ch02:eq:partial-normalization-dualizing}
\end{equation}

A section $a\eta^3$ has branch expressions
$$
a_+(x)\left(\frac{\dd x}{x}\right)^3,
\qquad
-a_-(y)\left(\frac{\dd y}{y}\right)^3,
\qquad
a_+(0)=a_-(0).
$$
Its residues therefore satisfy $R_++R_-=0$. Conversely, a pair of cubic differentials on the branches satisfying this relation determines a pair of coefficients with equal constants and hence a section of $\omega_f^{\otimes 3}$. Thus $R_++R_-$ is exactly the obstruction to matching arbitrary cubic differentials on the two branches.

For sections of $\omega_f^{\otimes3}$, selecting one branch identifies the residue quotient with $A$ by $a\eta^3\mapsto a_+(0)$, with kernel consisting of coefficients vanishing at both marked branches. By~\eqref{ch02:eq:partial-normalization-dualizing}, this lowers the allowed pole order from three to two and gives $\omega_{Y_v/S}^{\otimes3}(2D_v)$. The quotient map to $A$ has an $A$-linear section given by constant coefficients, so its exactness is preserved under arbitrary base change.

The residue is independent of the local parameter at the branch. For a second relative parameter $z'=az+O(z^2)$, with $a$ invertible on the base,
$$
\frac{\dd z'}{z'}=(1+O(z))\frac{\dd z}{z}.
$$
Cubing leaves the leading coefficient unchanged. This proves invariance under every automorphism preserving the markings, including over nonreduced bases.

The designated node sections are disjoint. After choosing one half-edge of each edge, the local sequences consequently glue to
\begin{equation}
0\longrightarrow
\bigoplus_v\nu_{v,*}\omega_{Y_v/S}^{\otimes 3}(2D_v)
\longrightarrow\omega_f^{\otimes 3}
\longrightarrow
\bigoplus_{a\in E(\Gamma)}\sigma_{a,*}\OO_S
\longrightarrow0,
\label{ch02:eq:sheaf-residue-sequence}
\end{equation}
where $\sigma_a$ is the node section. Intrinsically, the last map records the branch residues, which satisfy $R_h=-R_{\bar h}$. Partial normalization along the designated node sections makes this construction compatible with arbitrary base schemes.

Each pointed vertex family is stable, including those of genus zero and one. By Proposition~\ref{ch02:prop:cubic-extension}, we have
$$
R(\pi_v)_*\omega_{\pi_v}^{\otimes3}(2D_v)\simeq\mathbb F_v[0],
\qquad \rank\mathbb F_v=5g_v-5+2n_v,
$$
and formation of~$\mathbb F_v$ commutes with arbitrary scheme base change. The ranks for~$(g_v,n_v)=(0,3)$ and~$(1,1)$ are respectively~$1$ and~$2$.

Since $\nu$ is finite, it has no higher coherent direct images. The vanishing just proved makes the pushforward of \eqref{ch02:eq:sheaf-residue-sequence} right-exact. By Proposition~\ref{ch02:prop:cubic-extension}, its middle term is the pullback of $\bE_3$. This proves \eqref{ch02:eq:residue-sequence} on $S$. The intrinsic residue maps commute with isomorphisms and with composition of base changes, so they descend to $B_\Gamma$. In particular, every vector of matched node residues occurs on each geometric fiber.

Choosing one half-edge of each edge identifies $\cR_\Gamma\simeq\OO_{B_\Gamma}^{e_\Gamma}$. The rank identity follows from $\sum_vn_v=2e_\Gamma$:
$$
\sum_v\rank \mathbb F_v+e_\Gamma
=5\sum_vg_v-5|V(\Gamma)|+5e_\Gamma
=5g-5.
$$

A graph automorphism permutes the intrinsic residue coordinates $R_h$ and acts by a signed permutation in a chosen half-edge basis; exchanging the two half-edges of a loop reverses the corresponding residue. Vertex stabilizers preserving the markings fix every residue coordinate, so the sequence is $G_\Gamma$-equivariant and descends to $Q_\Gamma$.

The quotient map $q_\Gamma\colon B_\Gamma\to Q_\Gamma$ is a finite étale $G_\Gamma$-torsor of degree $|G_\Gamma|$, even when the action on geometric points has fixed points. By the projection formula in rational Chow theory,
$$
(q_\Gamma)_*q_\Gamma^*\alpha=|G_\Gamma|\alpha.
$$
The transfer map gives this identity in real cohomology as well, so pullback is injective with either set of coefficients. The descended residue bundle $\overline{\cR}_\Gamma$ on~$Q_\Gamma$ satisfies $q_\Gamma^*\overline{\cR}_\Gamma\simeq\cR_\Gamma\simeq\OO_{B_\Gamma}^{e_\Gamma}$, so its positive-degree Chern classes vanish with these coefficients. The residue bundle may retain its signed permutation action and torsion integral Chern classes.

The quotient of the product of smooth vertex moduli stacks is the open stratum of curves with exact dual graph $\Gamma$. The compactified quotient $Q_\Gamma$ is the normalization of its closed image, by \cite[Definition~3.3, Proposition~4.4, and Theorem~4.5]{EdidinFulghesu2010}. Here $G_\Gamma$ includes exchanges of the two half-edges of a loop. For any top-degree class on $Q_\Gamma$, transfer gives
\begin{equation}
\int_{Q_\Gamma}\alpha
=\frac{1}{|G_\Gamma|}
\int_{B_\Gamma}q_\Gamma^*\alpha.
\label{ch02:eq:quotient-integration}
\end{equation}
Each integral on a vertex moduli stack includes its stabilizer weights. A different choice of branches changes the basis of $\cR_\Gamma$ and preserves the stack degree in \eqref{ch02:eq:quotient-integration}.

Finally, multiplicativity of Chern classes in \eqref{ch02:eq:residue-sequence} proves \eqref{ch02:eq:clutching-denominator}. The descended identity follows from the vanishing of the rational or real Chern classes of the residue quotient. After adjoining $u^{-1}$, all inverse series truncate by dimension. For $u=s/(2\pi)$, the identity becomes
$$
\xi_\Gamma^*D_{s/(2\pi)}(\bE_3)^{-1}
=
\left(\frac{2\pi}{s}\right)^{e_\Gamma}
\prod_vD_{s/(2\pi)}(\mathbb F_v)^{-1}.
$$
This is the restriction of the denominator in \eqref{ch02:eq:characteristic-expression}.
\end{proof}

\begin{rem}
Residues are linear in the cubic differential and therefore have rotation weight $-1$, while curve coordinates, including every plumbing parameter, have weight zero.
\end{rem}

\subsection{Nodal Riemann--Roch and classes~$h_k$}
\label{ch02:sec:cubic-grr}

Put
$$
K=c_1(\omega_{\bar\pi}),
\qquad
\kappa_m=\bar\pi_*(K^{m+1})\quad(m\geq0).
$$
Let $Z_{\mathrm{node}}\subset\bC_g$ be the relative node locus. Its degree-two finite étale cover
$$
e_{\mathrm{ord}}\colon Z_{\mathrm{node}}^{\mathrm{ord}}
\longrightarrow Z_{\mathrm{node}}
$$
orders the branches of the selected node. Write $j\colon Z_{\mathrm{node}}^{\mathrm{ord}}\to\bM_g$ for the composite map. Let $L_+$ and $L_-$ be the branch cotangent line bundles and put
$$
\psi_+=c_1(L_+),\qquad
\psi_-=c_1(L_-),\qquad
\delta=\frac12j_*1,\qquad
\lambda_1=c_1(\bar\pi_*\omega_{\bar\pi}).
$$
A point of the ordered node cover selects one node even when the curve has several.

The Bernoulli polynomials are defined by
$$
\frac{t e^{xt}}{e^t-1}
=\sum_{n\geq0}B_n(x)\frac{t^n}{n!},
\qquad
B_n=B_n(0),
\qquad
B_1=-\frac12.
$$
For $m\geq1$, set
$$
\gamma_{m-1}(u,v)
=\sum_{a=0}^{m-1}(-1)^a u^a v^{m-1-a},
\qquad
\gamma_{-1}=0.
$$
The node correction in Riemann--Roch is expressed by these polynomials. Its pushforward gives the Chern character of $\bE_3$.

\begin{prop}
\label{ch02:prop:cubic-grr}
The degree-zero Chern character is $\ch_0(\bE_3)=B_1(3)\kappa_0=5g-5$. For every $m\geq1$,
\begin{equation}
\ch_m(\bE_3)
=
\frac{B_{m+1}(3)}{(m+1)!}\kappa_m
+
\frac{B_{m+1}}{2(m+1)!}
j_*\gamma_{m-1}(\psi_+,\psi_-).
\label{ch02:eq:cubic-grr}
\end{equation}
These identities hold in the rational Chow ring of $\bM_g$.
\end{prop}

\begin{proof}
We apply~\cite[Theorem~1.1.1]{Chiodo08} with root parameter one, exponent parameter three, and no markings. The universal root line bundle is then $\omega_{\bar\pi}^{\otimes3}$, whose derived direct image is $\bE_3[0]$ by Proposition~\ref{ch02:prop:cubic-extension}. The only node character is zero. Chiodo's ordered node map is $j$, and his node polynomial is $\gamma_{m-1}(\psi_+,\psi_-)$; the factor $1/2$ accounts for ordering the two branches. With these substitutions, \eqref{ch02:eq:cubic-grr} follows from the cited theorem, and the degree-zero Chern character is $B_1(3)\kappa_0=\tfrac52(2g-2)=r$.
\end{proof}

The localization denominator expands in signed Segre classes. We specify their relation to Segre classes before computing the first coefficients.

\begin{defn}
\label{ch02:defn:symmetric-classes}
For formal Chern roots $x_1,\ldots,x_r$ of $\bE_3$, the classes $h_k$ and $\Seg_k$ are defined by
\begin{equation}
\sum_{k\geq0}h_kz^k
=\prod_{i=1}^{r}(1-zx_i)^{-1},
\qquad
\sum_{k\geq0}\Seg_kz^k
=\prod_{i=1}^{r}(1+zx_i)^{-1}.
\label{ch02:eq:symmetric-segre-series}
\end{equation}
Their coefficients are symmetric polynomials in the formal Chern roots.
\end{defn}

The Bernoulli values in Proposition~\ref{ch02:prop:cubic-grr} give the first three Chern characters. Their generating function then determines the classes~$h_k$.

\begin{cor}
\label{ch02:cor:cubic-characteristic-classes}
The relation $12\lambda_1=\kappa_1+\delta$ and the Chern character formulas give
\begin{equation}
\begin{aligned}
\ch_1(\bE_3)&=\frac{37}{12}\kappa_1+\frac{\delta}{12}
=37\lambda_1-3\delta,\\
\ch_2(\bE_3)&=\frac52\kappa_2,\\
\ch_3(\bE_3)&=\frac{1079}{720}\kappa_3
-\frac1{1440}j_*(\psi_+^2-\psi_+\psi_-+\psi_-^2).
\end{aligned}
\label{ch02:eq:cubic-low-chern-characters}
\end{equation}
For every $k\geq0$, $h_k=(-1)^k\Seg_k$, and
\begin{equation}
\log\left(\sum_{k\geq0}h_kz^k\right)
=\sum_{m\geq1}(m-1)!\ch_m(\bE_3)z^m.
\label{ch02:eq:symmetric-logarithm}
\end{equation}
In particular,
$$
\begin{aligned}
h_1={}&\ch_1(\bE_3),\\
h_2={}&\frac{\ch_1(\bE_3)^2}{2}+\frac52\kappa_2,\\
h_3={}&\frac{\ch_1(\bE_3)^3}{6}+\frac52\ch_1(\bE_3)\kappa_2
+\frac{1079}{360}\kappa_3
-\frac1{720}j_*(\psi_+^2-\psi_+\psi_-+\psi_-^2).
\end{aligned}
$$
\end{cor}

\begin{proof}
By Mumford's formula, $\lambda_1=\kappa_1/12+j_*1/24$, hence $12\lambda_1=\kappa_1+\delta$ \cite[Section~5, formula~(5.2)]{Mumford83}. Substituting
$$
B_2(3)=\frac{37}{6},\qquad B_3(3)=15,\qquad
B_4(3)=\frac{1079}{30},\qquad B_2=\frac16,\quad B_4=-\frac1{30}
$$
into \eqref{ch02:eq:cubic-grr} gives the Chern characters. For the classes~$h_k$, $\ch_m=(m!)^{-1}\sum_i x_i^m$ implies
$$
\log\prod_i(1-zx_i)^{-1}
=\sum_{m\geq1}\frac{\sum_i x_i^m}{m}z^m
=\sum_{m\geq1}(m-1)!\ch_m z^m.
$$
This proves \eqref{ch02:eq:symmetric-logarithm}. Expanding its exponential through degree three gives the displayed $h_k$, and replacing $z$ by $-z$ in the Segre series gives $h_k=(-1)^k\Seg_k$.
\end{proof}

For a stable pair~$(g,n)$,~$j_{g,n}$ denotes the corresponding ordered node cover. Its selected internal node is disjoint from the markings. The same Riemann--Roch calculation gives the following marked formula.

\begin{prop}
\label{ch02:prop:pointed-cubic-grr}
The Chern character of $\mathbb F_{g,n}$ satisfies
\begin{equation}
\begin{split}
\ch_m(\mathbb F_{g,n})
={}&
\frac{B_{m+1}(3)}{(m+1)!}\kappa_m
-\frac{B_{m+1}(1)}{(m+1)!}\sum_i\psi_i^m\\
&+
\frac{B_{m+1}}{2(m+1)!}
(j_{g,n})_*\gamma_{m-1}(\psi_+,\psi_-),
\qquad m\geq0.
\end{split}
\label{ch02:eq:pointed-cubic-grr}
\end{equation}
\end{prop}

\begin{proof}
In~\cite[Theorem~1.1.1]{Chiodo08}, we take root parameter one, exponent parameter three, and marking multiplicity one at every marking. With $\omega_{\log}=\omega_{\pi_{g,n}}(D_{\mathrm{mark}})$, the root line bundle is
$$
\omega_{\log}^{\otimes3}(-D_{\mathrm{mark}})=\omega_{\pi_{g,n}}^{\otimes3}(2D_{\mathrm{mark}}).
$$
Its derived direct image is $\mathbb F_{g,n}[0]$ by Proposition~\ref{ch02:prop:cubic-extension}. The marking coefficient is consequently $B_{m+1}(1)$, and the sole node character is zero, giving \eqref{ch02:eq:pointed-cubic-grr}. In degrees zero and one, the formula reads
$$
\rank\mathbb F_{g,n}=5g-5+2n,\qquad
\ch_1(\mathbb F_{g,n})=\frac{37}{12}\kappa_1
-\frac1{12}\sum_i\psi_i+\frac1{24}(j_{g,n})_*1.
$$
The argument applies to every stable pair~$(g,n)$, including genus zero and genus one.
\end{proof}

\subsection{The Hodge expression}

The Hodge bundle on~$\bM_{g,n}$ is~$\mathbb E=(\pi_{g,n})_*\omega_{\pi_{g,n}}$, with Chern classes~$\lambda_j=c_j(\mathbb E)$. We use the series
$$
 \Lambda_g(t)=\sum_{j=0}^g t^j\lambda_j,
 \qquad C_{g,n}(t)=c_{-t}(\mathbb F_{g,n})^{-1}.
$$
For root parameter one, the class of Giacchetto, Lewa\'nski, and Norbury~\cite[Section~2]{GLN} specializes to
$$
 \Omega_{g,n}^{[u]}(1,3;1^n)
 =c_u\bigl(-R^\bullet(\pi_{g,n})_*
           \omega_{\pi_{g,n}}^{\otimes3}(2D_{\mathrm{mark}})\bigr).
$$
Its degree-zero component is one. The marking and node terms in Riemann--Roch cancel when $\mathbb F_{g,n}$ is compared with the Hodge bundle, leaving a formula involving only~$\lambda$ and~$\kappa$ classes.

\begin{prop}\label{prop:marked-cubic-hodge}
For every stable~$(g,n)$, the following identity holds in~$A^*(\bM_{g,n})_{\mathbb Q}[t]$:
\begin{equation}\label{eq:marked-cubic-hodge}
 C_{g,n}(t)
 =\Omega_{g,n}^{[-t]}(1,3;1^n)
 =\Lambda_g(t)\exp\left(\sum_{m\ge1}
                    \frac{1+2^m}{m}t^m\kappa_m\right).
\end{equation}
The exponential is truncated above codimension~$D_{g,n}$.
\end{prop}

\begin{proof}
The first equality follows from Proposition~\ref{ch02:prop:cubic-extension}. In Chiodo's formula~\cite[Theorem~1.1.1]{Chiodo08}, exponent three and marking multiplicities~$1^n$ give~$\omega_{\log}^{\otimes3}(-D_{\mathrm{mark}})=\omega^{\otimes3}(2D_{\mathrm{mark}})$, whereas exponent one with the same markings gives~$\omega$. Their marking and node terms agree. Since
$$
 B_{m+1}(3)-B_{m+1}(1)=(m+1)(1+2^m),
$$
we obtain, for~$m\ge1$,
$$
 \ch_m(\mathbb F_{g,n})-\ch_m(\mathbb E)
 =\frac{1+2^m}{m!}\kappa_m.
$$
In~$K^0(\bM_{g,n})$, we have $[R^\bullet(\pi_{g,n})_*\omega_{\pi_{g,n}}]=[\mathbb E]-[\OO]$, so the positive-degree components of the Chern character of this derived direct image agree with those of~$\mathbb E$. The logarithm of the inverse Chern polynomial is
$$
 \log c_{-t}(W)^{-1}
 =\sum_{m\ge1}(m-1)!t^m\ch_m(W).
$$
For~$g\ge2$, Mumford's relation~$c_t(\mathbb E)c_{-t}(\mathbb E)=1$~\cite[Section~5]{Mumford83} pulls back from~$\bM_g$ to~$\bM_{g,n}$. For~$g=0$, the Hodge bundle is zero. For~$g=1$, it pulls back from~$\bM_{1,1}$, whose dimension is one; hence~$\lambda_1^2=0$ and the same relation holds. Consequently~$c_{-t}(\mathbb E)^{-1}=\Lambda_g(t)$ in every stable case, which proves~\eqref{eq:marked-cubic-hodge}. The same specialization follows from the identity for changing the exponent in~\cite[Theorem~4.1(i)]{GLN}.
\end{proof}

We define the marked polynomials by
\begin{equation}\label{eq:marked-volume-definition}
 \mathcal V_{g,n}(a,t;\mathbf x)
 =\int_{\bM_{g,n}}
 e^{a\kappa_1+\sum_{i=1}^n x_i\psi_i}C_{g,n}(t).
\end{equation}
Their descendant coefficients are
\begin{equation}\label{eq:marked-F-definition}
 \begin{aligned}
 F_{g,n}(a,t;\mathbf d)
 &=\int_{\bM_{g,n}}e^{a\kappa_1}C_{g,n}(t)
                             \prod_{i=1}^n\psi_i^{d_i},\\
 \mathcal V_{g,n}(a,t;\mathbf x)
 &=\sum_{\mathbf d\ge0}\frac{\mathbf x^{\mathbf d}}{\mathbf d!}
                         F_{g,n}(a,t;\mathbf d).
 \end{aligned}
\end{equation}
We assign value zero when a descendant exponent is negative or~$|\mathbf d|>D_{g,n}$. The parameters~$a,t,\mathbf x$ are formal variables, so these formulas define polynomials over~$\mathbb Q$ and also permit arbitrary real or complex substitutions.

\subsection{The unmarked intersection expression}

To assemble the localization expression, we use the compactified Weil--Petersson class
\begin{equation}
[\omega_{\WP}]=2\pi^2\kappa_1.
\label{ch02:eq:compactified-wp-class}
\end{equation}
Its relation to compactified intersection theory is recorded in \cite[Theorem~7]{Do2011WP}, with the normalization in \cite[\S5]{Wolpert2007Cusps}. For $s>0$, define the intersection expression
\begin{equation}
P_g(s)=
\int_{\bM_g}
\left[
\frac{\exp(2[\omega_{\WP}])}
{\prod_{j=1}^{r}\bigl(s/(2\pi)-x_j(\bE_3)\bigr)}
\right]_{2d},
\label{ch02:eq:compactified-fixed-degree}
\end{equation}
where the subscript denotes real cohomological degree $2d$. By Proposition~\ref{ch02:prop:cubic-extension}, its integrand restricts to the degree-$2d$ part of \eqref{ch02:eq:characteristic-expression} on $\cM_g$.

\begin{prop}
\label{ch02:prop:cubic-coefficients}
For every $s>0$,
\begin{align}
P_g(s)
&=
\sum_{k=0}^{d}
\left(\frac{2\pi}{s}\right)^{r+k}
\frac{2^{d-k}}{(d-k)!}
\int_{\bM_g}[\omega_{\WP}]^{d-k}h_k,
\label{ch02:eq:fixed-expansion-wp}\\
&=
\sum_{k=0}^{d}
\frac{(2\pi)^{r+k}(4\pi^2)^{d-k}}{(d-k)!}
s^{-(r+k)}
\int_{\bM_g}\kappa_1^{d-k}h_k.
\label{ch02:eq:fixed-expansion-kappa}
\end{align}
Thus its possible powers are $s^{-r},\ldots,s^{-(r+d)}$. After the displayed powers of $\pi$ are removed, every coefficient is a rational tautological intersection number. Every coefficient can be evaluated from rational initial data by the marked recursion of Section~\ref{sec:marked-recursion}; the equivalent expansion in stable graphs is given in Appendix~\ref{ch02:sec:cubic-coefficients}.
\end{prop}

\begin{proof}
The symmetric-class generating function gives
$$
\prod_{j=1}^{r}\left(\frac{s}{2\pi}-x_j\right)^{-1}
=\sum_{k\geq0}\left(\frac{2\pi}{s}\right)^{r+k}h_k.
$$
Since $h_k$ has real degree $2k$, its complementary term in $\exp(2[\omega_{\WP}])$ is $2^{d-k}[\omega_{\WP}]^{d-k}/(d-k)!$. Selecting degree $2d$ proves \eqref{ch02:eq:fixed-expansion-wp}; substituting \eqref{ch02:eq:compactified-wp-class} gives \eqref{ch02:eq:fixed-expansion-kappa}. Rationality follows from integration in the rational Chow ring. The stable-graph evaluation follows from Proposition~\ref{ch02:prop:finite-graph-evaluation}. Under clutching, \eqref{ch02:eq:clutching-denominator} supplies one factor $2\pi/s$ for each designated node.
\end{proof}

The marked polynomial gives the same expression in the form
\begin{equation}\label{eq:marked-return-P}
 P_g(s)=t^{5g-5}\mathcal V_{g,0}(4\pi^2,t),\qquad t=\frac{2\pi}{s}.
\end{equation}
By the global Stokes identity proved in Section~\ref{ch05:sec:boundary}, this expression equals~$Z_g(s)$.

\section{Recursion for the marked polynomials}
\label{sec:marked-recursion}

We evaluate all descendant integrals of $C_{g,n}(t)$ by Eynard--Orantin topological recursion and incorporate the factor $e^{a\kappa_1}$ by inserting additional markings.

\subsection{Topological recursion}

We use the spectral data
\begin{equation}\label{eq:marked-spectral-data}
 X(z)=\log z-z,\qquad Y(z)=\frac{z^3}{3},\qquad
 B(z_1,z_2)=\frac{\dd z_1\,\dd z_2}{(z_1-z_2)^2}.
\end{equation}
The differential $\dd X$ has a simple zero at $z=1$. We denote by $\varsigma$ its local involution, determined by $X(\varsigma(z))=X(z)$ and $\varsigma(z)\ne z$. The initial bidifferential is $W_{0,2}=B$. For stable $(g,n)$ with $n\ge1$, we write $I=\{2,\ldots,n\}$ and define
\begin{equation}\label{eq:marked-EO}
\begin{split}
 W_{g,n}(p,z_I)=\operatorname*{Res}_{z=1}K(p,z)
 \bigg(&W_{g-1,n+1}(z,\varsigma(z),z_I)\\
 &+\sum_{\substack{g_1+g_2=g\\ A\sqcup B=I}}^{\prime}
 W_{g_1,|A|+1}(z,z_A)
 W_{g_2,|B|+1}(\varsigma(z),z_B)\bigg),
\end{split}
\end{equation}
where
\begin{equation}\label{eq:marked-EO-kernel}
\begin{split}
 K(p,z)
 &=\frac{\int_{\varsigma(z)}^z B(p,u)}
 {2\bigl(Y(z)-Y(\varsigma(z))\bigr)\dd X(z)}\\
 &=\frac{3z}
 {2(1-z)(z^2+z\varsigma(z)+\varsigma(z)^2)(p-z)(p-\varsigma(z))}
 \frac{\dd p}{\dd z}.
\end{split}
\end{equation}
The splitting is ordered, and the prime excludes factors of type $(0,1)$ while retaining factors of type $(0,2)$; negative-genus correlators vanish by convention. Evaluation at $\varsigma(z)$ includes the pullback of its differential, and the integration limits in~\eqref{eq:marked-EO-kernel} fix the sign of the kernel.

For $d\ge0$, we set
\begin{equation}\label{eq:marked-xi}
 \xi_d(z)=\left(\frac{z}{1-z}\frac{\partial}{\partial z}\right)^d
                  \frac1{1-z}.
\end{equation}
The coefficients in the expansion
\begin{equation}\label{eq:marked-J-extraction}
 W_{g,n}(z_1,\ldots,z_n)
 =\sum_{|\mathbf d|\le D_{g,n}}J_{g,n}(\mathbf d)
                    \prod_{i=1}^n\dd\xi_{d_i}(z_i)
\end{equation}
are uniquely determined. We assign $J_{g,n}(\mathbf d)=0$ if an entry of $\mathbf d$ is negative or if $|\mathbf d|>D_{g,n}$.

\begin{prop}\label{prop:marked-spectral-dictionary}
For every stable $(g,n)$ with $n\ge1$, the recursion \eqref{eq:marked-EO} has the finite expansion \eqref{eq:marked-J-extraction}, and
\begin{equation}\label{eq:marked-J-intersection}
 J_{g,n}(\mathbf d)
 =\int_{\bM_{g,n}}
      \Omega_{g,n}^{[1]}(1,3;1^n)\prod_i\psi_i^{d_i}.
\end{equation}
For~$\mathbf d\in\Z_{\ge0}^n$ with~$|\mathbf d|\le D_{g,n}$,
\begin{equation}\label{eq:marked-J-scaling}
 F_{g,n}(0,t;\mathbf d)
 =(-t)^{D_{g,n}-|\mathbf d|}J_{g,n}(\mathbf d).
\end{equation}
\end{prop}

\begin{proof}
We apply~\cite[Remark~2 and Theorem~4.5]{LPSZ} with root parameter one and exponent parameter three. The case in which the exponent exceeds the root parameter is included in the proof through the comparison of the~$\kappa$ classes. For~$Y=z^3$, the prefactor is~$3^{-(2g-2+n)}$ and the one-form basis is~$\dd\xi_d$; the possible additive constant in~$\xi_0$ disappears on differentiation. Replacing~$Y$ by~$Y/3$ multiplies the kernel by three, so induction in~\eqref{eq:marked-EO} multiplies~$W_{g,n}$ by~$3^{2g-2+n}$ and proves~\eqref{eq:marked-J-intersection}. The dimension bound makes the expansion finite, and the degree parameter~$[-t]$ in Proposition~\ref{prop:marked-cubic-hodge} gives~\eqref{eq:marked-J-scaling}.
\end{proof}

The first two residues are
\begin{equation}\label{eq:marked-initial-W}
 W_{0,3}=\prod_{i=1}^3\frac{\dd z_i}{(1-z_i)^2},
 \qquad
 W_{1,1}(z)=\frac{-3+10z-4z^2}{24(1-z)^4}\,\dd z.
\end{equation}
Thus
\begin{equation}\label{eq:marked-initial-J}
 J_{0,3}(0,0,0)=1,\qquad
 J_{1,1}(0)=-\frac16,\qquad J_{1,1}(1)=\frac1{24}.
\end{equation}

\subsection{The complete formula}

We define
\begin{equation}\label{eq:marked-S-definition}
 P_t(w)=(1-tw)(1-2tw),\qquad
 S_{a,t}(w)=wP_t(w)(1-e^{-aw})
          =\sum_{k\ge1}s_k(a,t)w^{k+1}.
\end{equation}
The coefficients are
\begin{equation}\label{eq:marked-s-coefficient}
 \begin{aligned}
 s_1&=a,\qquad s_2=-a^2/2-3at,\\
 s_k&=(-1)^{k-1}\left(
 \frac{a^k}{k!}+\frac{3ta^{k-1}}{(k-1)!}
 +\frac{2t^2a^{k-2}}{(k-2)!}\right),\qquad k\ge3.
 \end{aligned}
\end{equation}
In particular, $s_3=a^3/6+3ta^2/2+2t^2a$.

For $h\ge2$, we include the unmarked numbers by the definition
\begin{equation}\label{eq:marked-unmarked-J}
 J_{h,0}=\frac{J_{h,1}(1)+3J_{h,1}(2)+2J_{h,1}(3)}{2h-2}.
\end{equation}
Together with~\eqref{eq:marked-EO}--\eqref{eq:marked-J-extraction}, this determines every~$J$ used below.

\begin{thm}\label{thm:marked-complete-recursion}
For every stable $(g,n)$, the polynomial $\mathcal V_{g,n}(a,t;\mathbf x)$ is symmetric in $\mathbf x$ and homogeneous of degree $D_{g,n}$ in $a,t,x_1,\ldots,x_n$. If $\mathbf d\ge0$ and $d_0=D_{g,n}-|\mathbf d|\ge0$, then
\begin{equation}\label{eq:marked-full-assembly}
\begin{split}
 F_{g,n}(a,t;\mathbf d)
 =\sum_{\ell=0}^{d_0}\frac1{\ell!}
 \sum_{\substack{k_1,\ldots,k_\ell\ge1\\
                         k_1+\cdots+k_\ell\le d_0}}
 &(-t)^{d_0-\sum_j k_j}
 \left(\prod_{j=1}^{\ell}s_{k_j}(a,t)\right)\\
 &\mathrel{}\times J_{g,n+\ell}
       (\mathbf d,k_1+1,\ldots,k_\ell+1).
\end{split}
\end{equation}
For $\ell=0$, the inner sum has one empty tuple and its term is $(-t)^{d_0} J_{g,n}(\mathbf d)$. For $d_0<0$, $F_{g,n}(a,t;\mathbf d)=0$.

Every coefficient is obtained by finitely many rational operations. For prescribed~$\mathbf d$, the calculation uses only differentials~$W_{h,m}$ with~$2h-2+m\le5g-5+2n-|\mathbf d|$. Each stable differential on the right-hand side of~\eqref{eq:marked-EO} has a smaller value of~$2h-2+m$ than~$W_{g,n}$ on the left.
\end{thm}

\subsection{Proof of the insertion formula}

For a symmetric collection of classes $A_{g,n}$ and a series $U(w)\in w^2\mathbb Q[a,t][[w]]$, we define
\begin{equation}\label{eq:marked-translation}
 (\operatorname{Tr}_U A)_{g,n}
 =\sum_{\ell\ge0}\frac1{\ell!}(f_\ell)_*
 \left(A_{g,n+\ell}\prod_{j=1}^{\ell}U(\psi_{n+j})\right),
\end{equation}
where $f_\ell$ forgets the last $\ell$ markings and $f_0$ is the identity. Since every new insertion has degree at least two, this is finite in each cohomological degree. We put
\begin{equation}\label{eq:marked-T-definition}
 T_{a,t}(w)=w\bigl[1-P_t(w)e^{-aw}\bigr].
\end{equation}

\begin{lem}\label{lem:marked-translations}
For every stable $(g,n)$,
\begin{equation}\label{eq:marked-translation-identities}
 e^{a\kappa_1}C_{g,n}(t)
 =(\operatorname{Tr}_{T_{a,t}}\Lambda(t))_{g,n}
 =(\operatorname{Tr}_{S_{a,t}}C(t))_{g,n}.
\end{equation}
\end{lem}

\begin{proof}
For $k_i\ge1$, the forgetful permutation identity is
\begin{equation}\label{eq:marked-cycle-identity}
 (f_\ell)_*\prod_{i=1}^{\ell}\psi_{n+i}^{k_i+1}
 =\sum_{\rho\in\mathfrak S_\ell}
   \prod_{\text{cycles }c\text{ of }\rho}\kappa_{\sum_{i\in c}k_i}.
\end{equation}
It follows by induction from $f_*\psi_{n+1}^{m+1}=\kappa_m$ and $\kappa_m^{(n+1)}=f^*\kappa_m^{(n)}+\psi_{n+1}^m$; the cycle containing the last label records the successive correction terms. This is the identity used in \cite[Lemma~A.1]{GLN} and in the calculations of higher Weil--Petersson volumes in \cite{KMZ}. For $T(w)/w=\sum_{k\ge1}p_kw^k$, grouping the permutations by their cycles gives
\begin{equation}\label{eq:marked-cycle-exponential}
 \sum_{\ell\ge0}\frac1{\ell!}(f_\ell)_*
                    \prod_{i=1}^{\ell}T(\psi_{n+i})
 =\exp\left(\sum_{m\ge1}[w^m]
       \{-\log(1-T(w)/w)\}\,\kappa_m\right).
\end{equation}
For $T=T_{a,t}$, the logarithm on the right is
$$
 aw-\log P_t(w)
 =aw+\sum_{m\ge1}\frac{1+2^m}{m}t^mw^m.
$$
Since the Hodge bundle pulls back under the forgetful morphism, the projection formula and~\eqref{eq:marked-cubic-hodge} give the first equality in~\eqref{eq:marked-translation-identities}. Although~\cite[Lemma~A.1]{GLN} is stated for $n>0$, the same argument applies to $n=0$, $g\ge2$, because stability is preserved after forgetting markings and~\eqref{eq:marked-cycle-identity} remains valid.

We next establish the additive law
\begin{equation}\label{eq:marked-translation-additive}
 \operatorname{Tr}_U\operatorname{Tr}_V A
 =\operatorname{Tr}_{U+V}A,
 \qquad U,V\in w^2\mathbb Q[a,t][[w]].
\end{equation}
In a term with $m$ outer and $k$ inner markings, the projection formula combines the two forgetful pushforwards, so an outer factor becomes $U(f_k^*\psi_i)$ on the space with $m+k$ added markings. The difference $U(f_k^*\psi_i)-U(\psi_i)$ is supported on boundary divisors whose rational tail contains the outer marking $i$ and a nonempty subset $J$ of the inner markings. Such a tail has $|J|+2$ special points and dimension $|J|-1$, whereas $\prod_{j\in J}V(\psi_j)$ restricts to a class of degree at least $2|J|$ and therefore vanishes on the tail. The same degree argument kills its products with powers of the supporting boundary divisors. We may thus replace each outer factor by $U(\psi_i)$; summing over $m,k$ gives the binomial expansion of $\prod(U+V)$ with coefficient $1/(m+k)!$, proving~\eqref{eq:marked-translation-additive}.

At $a=0$, the first equality gives $C(t)=\operatorname{Tr}_{T_{0,t}}\Lambda(t)$. Since $S_{a,t}+T_{0,t}=T_{a,t}$, the second equality in \eqref{eq:marked-translation-identities} follows from \eqref{eq:marked-translation-additive}.
\end{proof}

\begin{proof}[Proof of Theorem~\ref{thm:marked-complete-recursion}]
We first prove~\eqref{eq:marked-full-assembly} for $n\ge1$ and for the terms with $\ell\ge1$ when $n=0$. Multiplication of the second equality in~\eqref{eq:marked-translation-identities} by $\prod_i\psi_i^{d_i}$ and the projection formula produce the factors $\prod_i(f_\ell^*\psi_i)^{d_i}$ upstairs. Their difference from $\prod_i\psi_i^{d_i}$ is supported on rational tails containing one original marking and a nonempty set $J$ of new markings. As in the proof of Lemma~\ref{lem:marked-translations}, each such tail has dimension $|J|-1$, while the product of the factors $S_{a,t}(\psi_j)$ for $j\in J$ has degree at least $2|J|$ there; hence all these corrections vanish.

We can therefore expand the inserted factors using~\eqref{eq:marked-S-definition}. A term with $\ell$ new markings and indices $k_1,\ldots,k_\ell$ has descendant degree $|\mathbf d|+\ell+\sum_jk_j$ on a space of dimension $D_{g,n}+\ell$. It vanishes unless $\sum_jk_j\le d_0$, which also forces $\ell\le d_0$ since $k_j\ge1$; equation~\eqref{eq:marked-J-scaling} then gives~\eqref{eq:marked-full-assembly}.

For the unmarked term, the Hodge expression and the forgetful formula for $\kappa_m$ give
\begin{equation}\label{eq:marked-C-forgetting}
 f^*C_{g,n}(t)=C_{g,n+1}(t)P_t(\psi_{n+1}).
\end{equation}
For~$n=0$, we set~$t=-1$, multiply~\eqref{eq:marked-C-forgetting} by~$\psi_1$, and integrate. Since~$f_*\psi_1=2g-2$, we obtain~\eqref{eq:marked-unmarked-J}, equal to~$\int_{\bM_g}\Omega_{g,0}^{[1]}(1,3)$. Degree scaling then proves the~$\ell=0$ term.

Symmetry follows from permutation of the original markings, and each $s_k$ has degree $k$ in $a,t$. Thus~\eqref{eq:marked-full-assembly} has degree $d_0$ in these variables, and multiplication by $\mathbf x^{\mathbf d}$ gives total degree $D_{g,n}$.

In~\eqref{eq:marked-full-assembly}, the number of markings is~$n+\ell$ with~$\ell\le d_0$, so the negative Euler characteristic is at most~$2g-2+n+d_0=5g-5+2n-|\mathbf d|$. For~$n=0$, the one-point differential required in~\eqref{eq:marked-unmarked-J} satisfies the same bound. In~\eqref{eq:marked-EO}, the nonseparating term decreases~$2g-2+n$ by one, and the sum of the two corresponding integers in a splitting term is also one less. Each factor has a nonnegative value, with zero occurring only for the initial bidifferential~$W_{0,2}$. Thus every stable factor has a strictly smaller value of~$2h-2+m$ than the differential on the left. The effective rational computation below completes the termination proof.
\end{proof}

\subsection{String and dilaton identities}

We define
\begin{equation}\label{eq:marked-Q-definition}
 Q_{a,t}(w)=P_t(w)e^{-aw}=\sum_{k\ge0}q_k(a,t)w^k.
\end{equation}
Thus $q_0=1$, $q_1=-a-3t$, and $q_2=a^2/2+3at+2t^2$. The class $\Theta_{g,n}=e^{a\kappa_1}C_{g,n}(t)$ satisfies
\begin{equation}\label{eq:marked-Theta-forgetting}
 f^*\Theta_{g,n}
 =\Theta_{g,n+1}Q_{a,t}(\psi_{n+1}).
\end{equation}

\begin{prop}\label{prop:marked-string-dilaton}
For every stable $(g,n)$ and every $\mathbf d\ge0$,
\begin{align}
 \sum_{k\ge0}q_k(a,t)F_{g,n+1}(a,t;\mathbf d,k)
 &=\sum_{i=1}^nF_{g,n}(a,t;\mathbf d-\mathbf e_i),
 \label{eq:marked-string}\\
 \sum_{k\ge0}q_k(a,t)F_{g,n+1}(a,t;\mathbf d,k+1)
 &=(2g-2+n)F_{g,n}(a,t;\mathbf d).
 \label{eq:marked-dilaton}
\end{align}
Both sums are finite by dimension. Equivalently,
\begin{align}
 \left.Q_{a,t}(\partial_y)
       \mathcal V_{g,n+1}(a,t;\mathbf x,y)\right|_{y=0}
 &=\sum_{i=1}^n\int_0^{x_i}
       \mathcal V_{g,n}(a,t;x_1,\ldots,v,\ldots,x_n)\,\dd v,
 \label{eq:marked-polynomial-string}\\
 \left.Q_{a,t}(\partial_y)\partial_y
       \mathcal V_{g,n+1}(a,t;\mathbf x,y)\right|_{y=0}
 &=(2g-2+n)\mathcal V_{g,n}(a,t;\mathbf x).
 \label{eq:marked-polynomial-dilaton}
\end{align}
\end{prop}

\begin{proof}
Equation~\eqref{eq:marked-Theta-forgetting} follows from \eqref{eq:marked-C-forgetting} and $\kappa_1^{(n+1)}=f^*\kappa_1^{(n)}+\psi_{n+1}$. For the string equation, apply the projection formula and
\begin{equation}\label{eq:marked-psi-string-pushforward}
 f_*\prod_{i=1}^n\psi_i^{d_i}
 =\sum_{i:d_i>0}\psi_i^{d_i-1}\prod_{j\ne i}\psi_j^{d_j}.
\end{equation}
For completeness, let $D_i$ be the divisor whose rational tail contains the markings $i,n+1$. The relations $\psi_i=f^*\psi_i+D_i$, $D_iD_j=0$ for $i\ne j$, and $D_i^k=(-f^*\psi_i)^{k-1}D_i$ give~\eqref{eq:marked-psi-string-pushforward}, because $f_*D_i=1$, $\sum_{k=1}^{d_i}\binom{d_i}{k}(-1)^{k-1}=1$, and the pure pullback term has zero pushforward.

For the dilaton equation, multiplication by $\psi_{n+1}$ kills every $D_i$ correction, since its restriction to the three-pointed rational tail is zero. We may therefore replace each original cotangent line class $\psi_i$, $1\le i\le n$, by its pullback and use $f_*\psi_{n+1}=2g-2+n$. This proves~\eqref{eq:marked-string} and~\eqref{eq:marked-dilaton}. Multiplication by $\mathbf x^{\mathbf d}/\mathbf d!$ and summation give~\eqref{eq:marked-polynomial-string} and~\eqref{eq:marked-polynomial-dilaton}; the integral on the right-hand side of~\eqref{eq:marked-polynomial-string} records the factorial in this exponential generating series.
\end{proof}

\subsection{Effective rational computation}

We make the local residue and coefficient extraction explicit. We choose $z=Z(w)$ by
\begin{equation}\label{eq:marked-local-coordinate}
 X(Z(w))=-1-\frac{w^2}{2},\qquad
 Z(w)=1+\sum_{\nu\ge1}c_{\nu}w^\nu,\qquad c_1=1.
\end{equation}
Differentiation gives $(Z-1)Z'=wZ$, hence
\begin{equation}\label{eq:marked-coordinate-recursion}
 c_{\nu}=\frac{c_{\nu-1}}{\nu+1}
       -\frac12\sum_{i=2}^{\nu-1}c_i c_{\nu+1-i},\qquad \nu\ge2.
\end{equation}
In particular, $Z(w)=1+w+w^2/3+w^3/36-w^4/270+O(w^5)$ and $\varsigma(Z(w))=Z(-w)$.

For~$W_{g,n}$, we write~$D=D_{g,n}$ within this subsection. The bracket in~\eqref{eq:marked-EO}, as a coefficient of~$\dd w^2$, has pole order at most~$2D$. In a stable nonseparating term, the lower dimension is~$D-2$, so the two internal factors have combined pole order at most~$2(D-2)+4=2D$; for two stable splitting factors, the sum of dimensions is again~$D-2$. A factor~$B(Z(w),z_j)$ or~$B(Z(-w),z_j)$, with~$j\in I$, is holomorphic at~$w=0$ as a coefficient of~$\dd w\,\dd z_j$, and its complementary factor has dimension~$D-1$. The initial case~$(0,3)$ has a regular bracket, while the bracket at~$(1,1)$ is~$B(Z(w),Z(-w))$ and has a pole of order two. Since the kernel has a simple pole as a coefficient of~$\dd p/\dd w$, the residue integrand has pole order at most~$2D+1$.

It is enough to compute $Z$ through degree $2D+2$. An omitted term is $O(w^{2D+3})$; differentiation or removal of a simple factor $w$ changes the regular factors only by $O(w^{2D+2})$. Multiplication by a pole of order at most $2D+1$ cannot alter the residue. The remaining arguments belong to a field of rational functions, so $p-Z(w)$ is a unit in the local power-series ring.

For coefficient extraction, we use $u=z-1$ and write
\begin{equation}\label{eq:marked-b-recursion}
 \frac{\dd\xi_d}{\dd z}=\phi_d(u)
 =\sum_kb_{d,k}u^{-k},\qquad b_{0,2}=1,\qquad
 b_{d+1,k}=(k-1)(b_{d,k-1}+b_{d,k-2}).
\end{equation}
All unspecified initial entries are zero. The recurrence follows from $\phi_{d+1}=\phi_d/u^2-(1+u)\phi_d'/u$, which also shows that $b_{d,k}$ is supported on $d+2\le k\le2d+2$, with leading coefficient $b_{d,2d+2}=(2d+1)!!$. We extract $J_{g,n}(\mathbf d)$ from $W_{g,n}/\prod_i\dd z_i$ in decreasing order of $|\mathbf d|$: at each step, the coefficient of $\prod_i u_i^{-2d_i-2}$ in the remainder, divided by $\prod_i(2d_i+1)!!$, is the coefficient of $\prod_i\phi_{d_i}$ to be subtracted. Indeed, a tuple $\mathbf e$ contributes to this monomial only if $e_i\ge d_i$ for all $i$, so equality of total degrees forces $\mathbf e=\mathbf d$. This proves uniqueness and gives a finite extraction algorithm over $\mathbb Q$.

\subsection{Hodge evaluation}

The first equality in \eqref{eq:marked-translation-identities} gives a second finite expression for the descendant coefficients. We expand $T_{a,t}(w)=\sum_{k\ge1}p_k(a,t)w^{k+1}$, where
\begin{equation}\label{eq:marked-p-coefficient}
 \begin{aligned}
 p_1(a,t)&=a+3t,\\
 p_k(a,t)&=(-1)^{k-1}\left(
 \frac{a^k}{k!}+\frac{3ta^{k-1}}{(k-1)!}
 +\frac{2t^2a^{k-2}}{(k-2)!}\right),\qquad k\ge2.
 \end{aligned}
\end{equation}
For $d_0=D_{g,n}-|\mathbf d|\ge0$,
\begin{equation}\label{eq:marked-Hodge-assembly}
\begin{split}
 F_{g,n}(a,t;\mathbf d)
 =\sum_{j=0}^{\min(g,d_0)}t^j
 \sum_{\substack{m_1,\ldots,m_{d_0}\ge0\\
                    \sum k m_k=d_0-j}}
 &\left(\prod_{k=1}^{d_0}\frac{p_k(a,t)^{m_k}}{m_k!}\right)\\
 &\mathrel{}\times
 \int_{\bM_{g,n+\ell}}
   \lambda_j\prod_{i=1}^n\psi_i^{d_i}
   \prod_{k=1}^{d_0}\prod_{v=1}^{m_k}\psi_{k,v}^{k+1},
 \qquad \ell=\sum_km_k.
\end{split}
\end{equation}
Here $(k,v)$ labels the added markings; at $d_0=0$, the inner sum consists of the empty multi-index. The formula follows from the first equality in~\eqref{eq:marked-translation-identities}: the boundary corrections from the original $\psi$-classes vanish as before, grouping equal powers replaces $1/\ell!$ by $\prod_k1/m_k!$, and the dimension constraint is $j+\sum km_k=d_0$. Every integral on the right-hand side is evaluated rationally from Mumford's Riemann--Roch formula and the Witten--Kontsevich theorem by the recurrences in Appendix~\ref{sec:arithmetic}, where the initial values and termination proof are also given.

\section{Fiber integration and transgression}\label{ch03:sec:fiber-integration}

We express the fiberwise pushforward of $e^{\omega_G-s\mu}$ as a characteristic form plus an exact form, using symplectic reduction and integration over projective fibers. Polynomial estimates justify removal of the sublevel cutoff $\Lambda\to\infty$, and a potential for the transgression controls its boundary integral on $\cM_g$. All estimates in this section are local on $\cM_g$ and uniform for $s$ in compact subintervals of $(0,\infty)$.

\subsection{Geometric data and the reference metric}

We use the geometric data of Subsection~\ref{ch01:setup:cubic}, and the form on the zero section and the orientation from Proposition~\ref{ch01:prop:normal-form}. By Corollary~\ref{ch01:cor:levels}, the equivariant level maps have radial function $a_{\rad}(\tau,X,\psi)$, and each sublevel set of $\mu$ is compact over a compact base set.

On a compact set $K\subset\MG$, we use the charts and frames of Subsection~\ref{ch01:setup:parameters}, with real base coordinates $x$ and real fiber coordinates $z$. Any finite collection of parameter derivatives of the solution to Wang's equation and the coefficients of the Goldman form satisfies polynomial bounds by Proposition~\ref{ch01:prop:scalar-growth} and Corollary~\ref{ch01:prop:coefficient-growth}.

For the flat area $A_{\flatarea}$ of Definition~\ref{ch01:def:flat-area}, let $c_K$ and $C_K$ be its minimum and maximum on the unit sphere bundle over the compact chart cover of $K$. By Corollary~\ref{ch01:cor:coercivity}, $0<c_K\leq C_K<\infty$, and the resulting coercive area estimate holds.

We use the extension $\bE_3$ constructed in Proposition~\ref{ch02:prop:cubic-extension}. We choose a smooth orbifold Hermitian metric $\bar h$ on this bundle and denote its Chern connection by $\bar\nabla$. Such a metric is obtained by averaging local positive Hermitian metrics over finite local uniformizing groups and patching with a smooth orbifold partition of unity. We use this metric and connection both here and in the boundary argument, with
$$
 \nabla_{\refn}=\bar\nabla|_{\MG},
 \qquad
 Y=S(\Eg,\bar h|_{\MG}).
$$
The sphere and projective bundle projections are
$$
 \varrho\colon Y\longrightarrow\PP(\Eg),\qquad
 \varpi\colon\PP(\Eg)\longrightarrow\MG.
$$
Projective fibers have their complex orientation.

In fiber integration, base covectors precede fiber covectors, and $p_*^{\C}$ denotes integration with the complex fiber orientation; the even fiber dimension makes reversal of these two blocks sign-preserving. Boundary orientations place the outward normal first, and stack integrals include the reciprocals of finite stabilizer orders. We expand exponentials of forms as finite sums and truncate base forms above real degree $2d$. The notation $[\alpha]_j$ denotes the component of real degree $j$, with value zero for $j<0$ or $j$ above the dimension of the space on which $\alpha$ is defined. The constructions below use these orbifold data and commute with their local uniformizing group actions, so their locally defined forms descend to $\MG$.

\subsection{The induced $S^1$-connection near the zero section}

The radial homotopy determines the connection for the transgression. We first construct its primitive, then use the squared radius to extend the level connection to $\tau=0$.

For the next three lemmas, let $p\colon E\to B$ be a complex vector bundle of positive rank over a manifold or an orbifold, equipped with a smooth Hermitian metric. We write $i_0\colon B\to E$ for the zero section and $R$ for the Euler vector field. The vector field $V=-IR$ generates the clockwise scalar action $\Psi_\theta$. Let $\omega$ be a smooth closed invariant real two-form, and let $\mu$ satisfy
$$
 \mu|_{i_0(B)}=0,\qquad \iota_V\omega=-\dd\mu.
$$
Assume that the real fiber Hessian of $\mu$ is positive definite at every zero vector. Set
\begin{equation}\label{ch03:eq:radial-homotopy-definition}
 \begin{gathered}
 \omega_0=i_0^*\omega,\qquad m_a(x,z)=(x,az),\\
 b=\int_0^1 a^{-1}m_a^*(\iota_R\omega)\,\dd a.
 \end{gathered}
\end{equation}

\begin{lem}\label{ch03:lem:radial-primitive}
The form $b$ is smooth and invariant, vanishes at the zero section, and satisfies
\begin{equation}\label{ch03:eq:radial-identities}
 \dd b=\omega-p^*\omega_0,\qquad
 b(R)=0,\qquad b(V)=\mu.
\end{equation}
\end{lem}

\begin{proof}
In real bundle coordinates, evaluation of the integrand on $W=(W_x,W_z)$ gives
$$
 \bigl[a^{-1}m_a^*(\iota_R\omega)\bigr]_{(x,z)}(W)
 =\omega_{(x,az)}\bigl((0,z),(W_x,aW_z)\bigr).
$$
The expression is smooth for $0\leq a\leq1$, including $a=0$. Integration therefore defines a smooth form, with all derivatives obtained under the integral. Its value is zero at $z=0$. For the homotopy $h(a,e)=m_a(e)$, write $h^*\omega=\gamma_a+\dd a\wedge\lambda_a$, where $\gamma_a=m_a^*\omega$ and $\lambda_a=a^{-1}m_a^*(\iota_R\omega)$ for $a>0$, with the smooth extension just established at $a=0$. Closedness gives $\partial_a\gamma_a=\dd_E\lambda_a$, and the endpoints are $\gamma_1=\omega$ and $\gamma_0=p^*\omega_0$. Integration from zero to one proves the first identity in \eqref{ch03:eq:radial-identities}.

The two radial arguments are proportional when the integrand is evaluated on $R$, so $b(R)=0$. The scalar action commutes with dilation and preserves $\omega$; hence it preserves $b$. The moment identity gives $\omega(R,V)=R\mu$. Consequently,
$$
 \bigl[a^{-1}m_a^*(\iota_R\omega)\bigr](V)
 =a^{-1}(R\mu)\circ m_a
 =\partial_a(\mu\circ m_a).
$$
The integral is $\mu$, and evaluation of the moment identity on $V$ gives $V\mu=0$.
\end{proof}

To study the level connection near zero, set $Y=S(E)$, $F_t(x,\psi)=(x,t\psi)$, and $\eta=b/\mu$ on $\{\mu>0\}$. For a radial parametrization $\ell_\tau(x,\psi)=(x,a_{\rad}(\tau,\psi)\psi)$ of a positive level, write
$$
 \ell(\tau,\psi)=(x,a_{\rad}(\tau,\psi)\psi),\qquad
 \widehat\eta_\tau=\ell_\tau^*\eta.
$$

\begin{lem}\label{ch03:lem:radial-connection}
Locally over $Y$, there are smooth families $\beta_w$ and $M(w,\psi)$ for real $w$ near zero such that
\begin{equation}\label{ch03:eq:radial-square-expansions}
 \begin{gathered}
 F_t^*b=t^2\beta_{t^2},\qquad
 \mu(F_t(\psi))=t^2M(t^2,\psi),\\
 M(0,\psi)=\tfrac12D_{\mathrm{fib}}^2\mu(\psi,\psi)>0.
 \end{gathered}
\end{equation}
The small positive level sets of $\mu$ have radii $a_{\rad}^2=\tau H(\tau,\psi)$, where $H$ is smooth and positive up to $\tau=0$. Their pulled-back connections
$$
 \widehat\eta_\tau
 =\frac{\beta_{\tau H(\tau,\psi)}}
        {M(\tau H(\tau,\psi),\psi)}
$$
extend smoothly to $\tau=0$, with every mixed derivative, and $\widehat\eta_\tau(V)=1$.

If $R\mu>0$ off the zero section and $\mu$ tends to infinity on every ray $\{ae:a>0\}$ with $e\ne0$, these level identifications exist smoothly for every $\tau>0$. The constructions are equivariant and descend to orbifolds.
\end{lem}

\begin{proof}
The families $F_t^*b$ and $\mu\circ F_t$ are smooth for real $t$ near zero. They are even because $F_{-t}=\Psi_\pi\circ F_t$. By the theorem on smooth even functions with parameters \cite[Theorem~1 and concluding Remark]{Whitney1943Even}, each local coefficient is the restriction to $w=t^2\geq0$ of a smooth function of real $w$ near zero, with the sphere and base coordinates as parameters. If $G(w)$ is one of these functions, then $G(0)=0$ and
$$
 G(w)=w\int_0^1 G'(vw)\,\dd v.
$$
Thus both families have a factor $t^2$. For $w\geq0$, the remaining coefficients and all their derivatives agree on overlaps. A smooth partition of unity in the parameter charts gives extensions to real $w$ near zero; in orbifold charts we average these extensions over the local uniformizing groups. Only their common restriction to $w\geq0$ enters the level connections. Differentiating twice $\mu(F_t(\psi))$ at zero gives the stated value of $M(0,\psi)$. It is uniformly positive on a compact sphere bundle over a compact base set.

The equation for the squared radius is
$$
 wM(w,\psi)=\tau.
$$
Its derivative in $w$ at zero is positive. By the parameter-dependent implicit function theorem, the equation has a smooth solution $w(\tau,\psi)$ with $w(0,\psi)=0$ and $\partial_\tau w(0,\psi)=M(0,\psi)^{-1}$. The identity
$$
 w(\tau,\psi)
 =\tau\int_0^1\partial_\tau w(v\tau,\psi)\,\dd v
$$
gives $w=\tau H$, with $H$ smooth and positive. Thus $a_{\rad}=\sqrt{\tau}\sqrt H$. Compactness gives a common small interval over each compact base set.

In differentiating $\ell$, variation of $a_{\rad}$ adds $\dd\log a_{\rad}$ times $R$ to the differential of $F_{a_{\rad}}$ taken with its radius held constant. Since $b(R)=0$, this extra term does not contribute to $\ell^*b$. It follows that
$$
 \ell^*b=a_{\rad}^2\beta_{a_{\rad}^2},\qquad
 \ell^*\eta=\frac{\beta_{a_{\rad}^2}}{M(a_{\rad}^2,\psi)}.
$$
Both pullbacks have no $\dd\tau$ component. The quotient extends smoothly to $\tau=0$ with all mixed derivatives. Invariance and the identity $\widehat\eta_\tau(V)=1$ extend there by continuity.

Under the additional ray hypotheses, each ray meets each positive level exactly once. At a solution, $\partial_a\mu=(R\mu)/a_{\rad}>0$, so the implicit function theorem gives smooth dependence on all parameters, and uniqueness glues the local solutions. The inverse of the diffeomorphism $(0,\infty)\times Y\to E\setminus i_0(B)$ is obtained by taking the moment value and dividing the vector by its Hermitian norm.
\end{proof}

The induced connection determines both the reduced form and a curvature representative of the first Chern class of the tautological line bundle. We use the projections
$$
 \varrho\colon Y\longrightarrow\PP(E),\qquad
 \varpi\colon\PP(E)\longrightarrow B.
$$

\begin{lem}\label{ch03:lem:radial-reduction}
Under the additional ray hypotheses of Lemma~\ref{ch03:lem:radial-connection}, the curvature descends to a unique smooth family of closed forms $\Omega_\tau$ on $\PP(E)$ satisfying
\begin{equation}\label{ch03:eq:reduced-curvature-class}
 \dY\widehat\eta_\tau=\varrho^*\Omega_\tau,\qquad
 [\Omega_\tau]=2\pi c_1(\OO(-1)).
\end{equation}
On $(0,\infty)\times Y$, the pullbacks are
\begin{equation}\label{ch03:eq:full-level-pullback}
 \ell^*\eta=\widehat\eta_\tau,\qquad
 \ell^*\omega
 =(p|_Y)^*\omega_0
 +\dd\tau\wedge
       \bigl(\widehat\eta_\tau+\tau\partial_\tau\widehat\eta_\tau\bigr)
 +\tau\dY\widehat\eta_\tau.
\end{equation}
The reduced form is $\omega_\tau=\varpi^*\omega_0+\tau\Omega_\tau$, and the integral of $\widehat\eta_\tau$ over one period of the $S^1$-action is $2\pi$. All assertions are equivariant and descend to orbifolds.
\end{lem}

\begin{proof}
By Cartan's formula,
$$
 \iota_V\dY\widehat\eta_\tau
 =\mathcal L_V\widehat\eta_\tau
       -\dY\bigl(\widehat\eta_\tau(V)\bigr)=0.
$$
Thus the curvature is invariant and horizontal for the circle action, and local sections of $\varrho$ define the unique smooth descended closed form $\Omega_\tau$. The connection pulls back to $\dd\theta$ along the action, so its integral over one period is $2\pi$.

The sign of its Chern class follows from the transition functions. For local unit tautological sections satisfying $s_j=e^{\ii\phi_{ij}}s_i$, put $\alpha_i=s_i^*\widehat\eta_\tau$. The clockwise action gives $\alpha_j=\alpha_i-\dd\phi_{ij}$. Therefore
$$
 \nabla s_i=-\ii\alpha_i s_i
$$
defines a unitary connection on $\OO(-1)$, whose curvature is $-\ii\Omega_\tau$. With $c_1=(\ii/(2\pi))F$, its first Chern form is $\Omega_\tau/(2\pi)$. Restricting $\omega=p^*\omega_0+\dd(\mu\eta)$ to a level gives the reduced form. Pulling the same identity back to the product and expanding $\dd(\tau\widehat\eta_\tau)$ proves \eqref{ch03:eq:full-level-pullback}.
\end{proof}

The Hamiltonian identity in Proposition~\ref{ch01:prop:import-goldman} and the positive Hessian in Proposition~\ref{ch01:prop:zero-expansion} verify the hypotheses of the preceding three lemmas for $\mathbb E_3$. The global level maps are those of Corollary~\ref{ch01:cor:levels}. We use the resulting primitive $b$, connection $\eta$, and level connections $\widehat\eta_\tau$. In this case
\begin{equation}\label{ch03:eq:cubic-zero-end}
 M(0,X,\psi)=\int_X|\psi|_{\sigma_X}^2\,\dA_{\sigma_X},
 \qquad H(0,X,\psi)=M(0,X,\psi)^{-1}.
\end{equation}
The forms on level sets of $\mu$ and their derivatives of any finite order are bounded over compact base sets and compact intervals in $\tau\geq0$.

The mixed and normal blocks at zero and the preservation of their orientation sign along radial paths are established in Proposition~\ref{ch01:prop:normal-form}. The factor converting product complex orientation to Goldman orientation is therefore $(-1)^r$ throughout the total space.

\subsection{Integration of Chern forms over projective fibers}

The Chern forms and signed Segre forms associated with $\nabla_{\refn}$ are defined by
\begin{equation}\label{ch03:eq:chern-reference-data}
 \begin{aligned}
 A_{\refn}&=\frac{\ii}{2\pi}F_{\nabla_{\refn}},\\
 \det(\Id+zA_{\refn})
   &=\sum_{j=0}^r c_j(\Eg,\nabla_{\refn})z^j,\\
 \det(\Id-zA_{\refn})^{-1}
   &=\sum_{k\geq0}h_k(\Eg,\nabla_{\refn})z^k.
 \end{aligned}
\end{equation}
Here $z$ is a formal variable, and the inverse is evaluated in the exterior algebra of each cotangent space. Orthogonal projection of $\varpi^*\nabla_{\refn}$ onto $L=\OO(-1)$ defines the connection $\nabla_L$. Its real curvature form is
$$
 \Omega_{\refn}=\ii F_{\nabla_L}.
$$
This form represents $2\pi c_1(\OO(-1))$. We next compute its projective integral and the associated Laplace transform.

\begin{prop}\label{ch03:prop:projective-integration}
For every $k\geq0$,
\begin{equation}\label{ch03:eq:projective-segre-formula}
 \varpi_*\left(-\frac{\Omega_{\refn}}{2\pi}\right)^{r-1+k}
 =(-1)^k h_k(\Eg,\nabla_{\refn}).
\end{equation}
Powers below $r-1$ push forward to zero, and degrees above $2d$ vanish.

For $s>0$, the closed form
\begin{equation}\label{ch03:eq:theta-definition}
 \Theta_s
 =
 \left[e^{\omega_0}
       \sum_{k=0}^d
       \left(\frac{2\pi}{s}\right)^{r+k}
       h_k(\Eg,\nabla_{\refn})\right]_{\leq2d}
 =
 \left[e^{\omega_0}
       \det\left(\frac{s}{2\pi}\Id-A_{\refn}\right)^{-1}
       \right]_{\leq2d}
\end{equation}
is the Laplace transform, taken coefficient by coefficient, of
$$
 2\pi(-1)^{r-1}
 \varpi_*\exp(\varpi^*\omega_0+\tau\Omega_{\refn}).
$$
Its cohomology class is $e^{[\omega_0]}\prod_{j=1}^r(s/(2\pi)-x_j)^{-1}$, truncated above degree $2d$, where the $x_j$ are formal Chern roots.
\end{prop}

\begin{proof}
The induced connection on $L=\OO(-1)$ is its Chern connection. Thus
$$
 \Xi=-\frac{\Omega_{\refn}}{2\pi}
     =c_1\bigl(\OO(1),(\bar h|_L)^{-1}\bigr).
$$
By \cite[Proposition~1.1]{Diverio2016Segre}, the differential forms satisfy
$$
 \varpi_*\Xi^{r-1+k}
 =\Seg_k(\Eg,\nabla_{\refn})
 =(-1)^kh_k(\Eg,\nabla_{\refn}).
$$
That reference also parametrizes lines, and its pointwise proof applies on a noncompact complex base. The formula therefore holds on each orbifold chart and descends to $\MG$. Fiber and base dimensions give the degree assertions. Since $\Omega_{\refn}=-2\pi\Xi$, expansion of the exponential yields
\begin{equation}\label{ch03:eq:reference-level-polynomial}
 (-1)^{r-1}\varpi_*
     e^{\varpi^*\omega_0+\tau\Omega_{\refn}}
 =
 \left[e^{\omega_0}
 \sum_{k=0}^d
 \frac{(2\pi\tau)^{r-1+k}}{(r-1+k)!}
 h_k(\Eg,\nabla_{\refn})\right]_{\leq2d}.
\end{equation}
Indeed, the signs from the reduced orientation, the power of $-\Xi$, and the Segre form multiply to
$$
 (-1)^{r-1}(-1)^{r-1+k}(-1)^k=1.
$$
Applying this calculation to $[\omega_\tau]=\varpi^*[\omega_0]+2\pi\tau c_1(\OO(-1))$ gives the corresponding polynomial in cohomology for the reduced forms.

The gamma integral $\int_0^\infty e^{-s\tau}\tau^n\,\dd\tau=n!/s^{n+1}$ gives
\begin{equation}\label{ch03:eq:laplace-factor}
 2\pi\int_0^\infty e^{-s\tau}
       \frac{(2\pi\tau)^{r-1+k}}{(r-1+k)!}\,\dd\tau
 =\left(\frac{2\pi}{s}\right)^{r+k}.
\end{equation}
The first $2\pi$ is the circle period from Lemma~\ref{ch03:lem:radial-reduction}. Applying \eqref{ch03:eq:laplace-factor} to the finite polynomial \eqref{ch03:eq:reference-level-polynomial} proves \eqref{ch03:eq:theta-definition}. If $t=s/(2\pi)$, the determinant has constant term $t^r$, so its inverse exists in the truncated form algebra and equals $\sum_{k=0}^dt^{-r-k}h_k$.
\end{proof}

\begin{rem}
The determinant factor also has an equivariant interpretation. For $\dd_s=\dd-s\iota_{V_{\rot}}$, the form $\omega_G-s\mu$ is an equivariantly closed extension of $\omega_G$. With the complex orientation, the equivariant Euler class is
\begin{equation}\label{ch03:eq:equivariant-euler-class}
 e_{S^1}(\Eg)=\prod_{j=1}^r\left(x_j-\frac{s}{2\pi}\right).
\end{equation}
This follows from multiplicativity of equivariant Euler classes and the weight normalization in \cite[Section~3, equation~(3.7)]{AtiyahBott1984Moment}, applied to the circle action of weight $-1$. The equivariant de Rham model and the extension determined by a moment map are described in \cite[Section~4 and Proposition~6.18]{AtiyahBott1984Moment}. The induced pullback action on sections has infinitesimal operator $+\ii\Id$, since the circle acts trivially on the base. The equivariant curvature is $F_{\nabla_{\refn}}+s\ii\Id$; multiplying by $\ii/(2\pi)$ gives $A_{\refn}-s\Id/(2\pi)$. Multiplication of the normal orientation by $(-1)^r$, as in Proposition~\ref{ch01:prop:normal-form}, changes its inverse Euler class to $\prod_j(s/(2\pi)-x_j)^{-1}$.
\end{rem}

\subsection{Fiberwise sublevel sets and transgression}

Let $\alpha_{\refn}$ be the real circle connection on $Y$ induced by $\nabla_L$. For the clockwise action it satisfies
\begin{equation}\label{ch03:eq:reference-circle-connection}
 \alpha_{\refn}(V_{\rot})=1,\qquad
 \dY\alpha_{\refn}=\varrho^*\Omega_{\refn}.
\end{equation}
The difference $\widehat\eta_\tau-\alpha_{\refn}$ is invariant and horizontal for the circle action. It descends to a smooth one-form $\gamma_\tau$ on $\PP(\Eg)$, including at $\tau=0$, with
\begin{equation}\label{ch03:eq:connection-difference}
 \varrho^*\gamma_\tau=\widehat\eta_\tau-\alpha_{\refn},
 \qquad
 \Omega_\tau=\Omega_{\refn}+\dP\gamma_\tau.
\end{equation}
Here $\dP$ is the exterior differential on the total space $\PP(\Eg)$, including its base directions.

Define the even base forms
$$
 Q_\tau=(-1)^{r-1}\varpi_*e^{\omega_\tau},
 \qquad
 P_\tau^{\refn}
 =
 \left[e^{\omega_0}\sum_{k=0}^d
 \frac{(2\pi\tau)^{r-1+k}}{(r-1+k)!}
 h_k(\Eg,\nabla_{\refn})\right]_{\leq2d},
$$
and the odd base form
\begin{equation}\label{ch03:eq:level-transgression}
 U_\tau=(-1)^{r-1}\varpi_*
 \left[
 \tau\gamma_\tau\wedge
 \int_0^1
 e^{\,\varpi^*\omega_0+
       \tau(\Omega_{\refn}+v\dP\gamma_\tau)}\,\dd v
 \right].
\end{equation}
For $s>0$ and $0<\Lambda<\infty$, define the sublevel pushforward by
$$
 M_\Lambda=\{\mu\leq\Lambda\},\qquad
 p_\Lambda=p|_{M_\Lambda},\qquad
 \Phi_{s,\Lambda}=(-1)^r(p_\Lambda)_*^\C
                         e^{\omega_G-s\mu},
$$
and the truncated reference and transgression forms by
$$
 \Theta_{s,\Lambda}
       =2\pi\int_0^\Lambda e^{-s\tau}P_\tau^{\refn}\,\dd\tau,
 \qquad
 T_{s,\Lambda}
       =2\pi\int_0^\Lambda e^{-s\tau}U_\tau\,\dd\tau.
$$
For $0<\epsilon<\Lambda$, the notation $\Phi_{s;\epsilon,\Lambda}$ denotes the pushforward with the Goldman symplectic orientation over $\{\epsilon\leq\mu\leq\Lambda\}$.

The reference contribution has the explicit form
\begin{equation}\label{ch03:eq:finite-reference}
 \Theta_{s,\Lambda}
 =
 \left[
 e^{\omega_0}\sum_{k=0}^d
 \left(\frac{2\pi}{s}\right)^{r+k}
 \left(
 1-e^{-s\Lambda}
   \sum_{j=0}^{r+k-1}\frac{(s\Lambda)^j}{j!}
 \right)h_k(\Eg,\nabla_{\refn})
 \right]_{\leq2d}.
\end{equation}
We next express the sublevel pushforward as the reference form plus an exact form and describe its behavior as $\tau\downarrow0$.

\begin{prop}\label{ch03:prop:finite-localization}
For $s>0$ and $0<\epsilon<\Lambda$, the forms above are smooth and satisfy
\begin{equation}\label{ch03:eq:finite-localization}
 \begin{aligned}
 Q_\tau&=P_\tau^{\refn}+\dM U_\tau,\\
 \Phi_{s;\epsilon,\Lambda}
   &=2\pi\int_\epsilon^\Lambda e^{-s\tau}Q_\tau\,\dd\tau,\\
 \Phi_{s,\Lambda}
   &=2\pi\int_0^\Lambda e^{-s\tau}Q_\tau\,\dd\tau
     =\Theta_{s,\Lambda}+\dM T_{s,\Lambda}.
\end{aligned}
\end{equation}
The forms $\Phi_{s,\Lambda}$, $\Phi_{s;\epsilon,\Lambda}$, and $\Theta_{s,\Lambda}$ are closed.

At $\tau=0$, $Q_\tau/\tau^{r-1}$ and $U_\tau/\tau^r$ extend smoothly. On each compact coordinate set in the base and for every derivative order $m$,
$$
 Q_\epsilon=O_{C^m}(\epsilon^{r-1}),\quad
 U_\epsilon=O_{C^m}(\epsilon^r),\quad
 \Phi_{s,\epsilon}=O_{C^m}(\epsilon^r),\quad
 T_{s,\epsilon}=O_{C^m}(\epsilon^{r+1}).
$$
\end{prop}

\begin{proof}
We first compare the projective curvatures using \eqref{ch03:eq:connection-difference}. For each $\tau$, the interpolating form
$$
 W_{\tau,v}
 =\varpi^*\omega_0+\tau(\Omega_{\refn}+v\dP\gamma_\tau)
$$
is $\dP$-closed. Differentiating its exterior exponential gives
$$
 \partial_v e^{W_{\tau,v}}
   =\tau\dP\gamma_\tau\wedge e^{W_{\tau,v}}
   =\dP\bigl(\tau\gamma_\tau\wedge e^{W_{\tau,v}}\bigr).
$$
Integration in $v$, followed by integration over the closed even-dimensional projective fibers, proves $Q_\tau=P_\tau^{\refn}+\dM U_\tau$, using \eqref{ch03:eq:reference-level-polynomial}.

After factoring out $e^{\omega_0}$, each term defining $U_\tau$ has the form
$$
 \frac{\tau^{m+1}}{m!}\,
 \gamma_\tau\wedge(\Omega_{\refn}+v\dP\gamma_\tau)^m.
$$
Its total degree is $2m+1$, so contributing vertical degree $2r-2$ requires $m\geq r-1$, and every surviving term carries a factor of $\tau^r$. The same degree count gives a factor of $\tau^{r-1}$ in $\varpi_*e^{\tau\Omega_\tau}$; by Lemma~\ref{ch03:lem:radial-connection} and differentiation under the compact projective integral, the remaining families extend smoothly to $\tau=0$ with all mixed derivatives.

We next compute the shell integral for $0<\epsilon<\Lambda$. On $[\epsilon,\Lambda]\times Y$, write
$$
 B_\tau=(p|_Y)^*\omega_0+\tau\dY\widehat\eta_\tau
        =\varrho^*\omega_\tau,\qquad
 c_\tau=\widehat\eta_\tau+
                      \tau\partial_\tau\widehat\eta_\tau.
$$
The product pullback \eqref{ch03:eq:full-level-pullback} yields
\begin{equation}\label{ch03:eq:shell-exponential}
 \ell^*e^{\omega_G-s\mu}
 =e^{-s\tau}
   \left(e^{B_\tau}+\dd\tau\wedge c_\tau\wedge e^{B_\tau}\right).
\end{equation}
The forms on the right are defined on the total sphere bundle, so their base and mixed components are retained. Since the circle generator on $Y$ is independent of $\tau$,
$$
 (\partial_\tau\widehat\eta_\tau)(V_{\rot})=0,
 \qquad c_\tau(V_{\rot})=1.
$$
The derivative $\partial_\tau\widehat\eta_\tau$ is invariant and horizontal, and $B_\tau$ is also basic for the circle action. Contraction of $c_\tau\wedge e^{B_\tau}$ with $V_{\rot}$ therefore gives $e^{B_\tau}$, and circle integration yields $2\pi e^{\omega_\tau}$. The term $e^{B_\tau}$ in \eqref{ch03:eq:shell-exponential} has no radial covector and does not contribute to the shell integral.

For the orientation, take a point $(a,0,\ldots,0)$ in a complex fiber, with $a>0$. Increasing radius followed by $V_{\rot}$ has the orientation of $(\partial_{a_1},-\partial_{b_1})$. The remaining projective tangent directions have their complex orientation. Hence the order
$$
 (\text{increasing radius},\text{action circle},
                         \text{projective complex directions})
$$
is the negative of the complex fiber orientation. The radius increases with $\tau$. This holds near zero by $a_{\rad}^2=\tau H$, and persists because the level parametrizations are smooth radial diffeomorphisms. Dependence of $a_{\rad}$ on angular or base coordinates adds multiples of the radial vector to the remaining tangent vectors, preserving this orientation. Consequently,
$$
 (p_{[\epsilon,\Lambda]})_*^\C e^{\omega_G-s\mu}
 =-2\pi\int_\epsilon^\Lambda
       e^{-s\tau}\varpi_*e^{\omega_\tau}\,\dd\tau.
$$
Multiplication by $(-1)^r$ gives the second identity in \eqref{ch03:eq:finite-localization}. On the outer shell boundary the complex boundary orientation is minus the orientation of the circle orbit induced by $V_{\rot}$ followed by the complex orientation of the projective fiber. On the inner boundary it is plus that product.

The order of vanishing of $Q_\tau$ at $\tau=0$ shows that the integral over $[0,\epsilon]$ is $O_{C^m}(\epsilon^r)$. The shell integrals therefore converge in every local $C^m$ norm. To compare this limit with the sublevel pushforward, choose a smooth linear bundle chart. The coefficients of $e^{\omega_G-s\mu}$ are bounded on a finite sublevel over a compact base set. The omitted domain $\{\mu<\epsilon\}$ lies inside a fiber ball of radius $C\sqrt\epsilon$, by $a_{\rad}^2=\tau H$. Its coordinate volume is $O(\epsilon^r)$. Evaluation on coordinate base lifts and a vertical frame is bounded, so every omitted coefficient integral tends to zero. Hence the smooth shell limit equals the sublevel integral.

Since $P_\tau^{\refn}$ is closed, the first identity in \eqref{ch03:eq:finite-localization} makes $Q_\tau$ closed. Its finite Laplace integrals are therefore closed as well. Integrating that identity over $[0,\Lambda]$ proves the formula for $\Phi_{s,\Lambda}$. Integration of \eqref{ch03:eq:reference-level-polynomial} gives \eqref{ch03:eq:finite-reference}.

The factor $\tau^r$ in $U_\tau$ gives $T_{s,\epsilon}=O_{C^m}(\epsilon^{r+1})$. An additional base derivative gives the same estimate for $\dM T_{s,\epsilon}$. The remaining estimates as $\epsilon\downarrow0$ follow from the factor $\tau^{r-1}$ in $Q_\tau$.
\end{proof}

\subsection{Polynomial estimates and integrable fiber densities}

To pass to the improper integrals, we need bounds in both linear fiber coordinates and coordinates on level sets of $\mu$. We use the compact charts of Subsection~\ref{ch01:setup:parameters}, enlarged slightly when taking difference quotients, together with finite angular chart covers of the sphere and projective fibers. All coordinate norms in this subsection refer to these charts.

We use the coercivity estimate of Corollary~\ref{ch01:cor:coercivity} and the strict radial bounds of Proposition~\ref{ch01:prop:radial}. We first estimate the coefficients of the radial primitive.

\begin{prop}\label{ch03:prop:homotopy-estimates}
Every mixed coordinate derivative of $b$ and $\mu$ of finite order is polynomially bounded in $1+\|q\|$, uniformly over $K$. The coefficients of $\eta=b/\mu$ and all their mixed derivatives satisfy polynomial bounds on $\{\mu\geq1\}$.
\end{prop}

\begin{proof}
The coefficient bounds in Corollary~\ref{ch01:prop:coefficient-growth} control the radial homotopy integral. In a real linear bundle chart write
$$
 \omega_G
 =\frac12H_{ij}\,\dd x^i\wedge\dd x^j
  +C_{ai}\,\dd z^a\wedge\dd x^i
  +\frac12D_{ab}\,\dd z^a\wedge\dd z^b,
 \qquad H_{ij}=-H_{ji},\quad D_{ab}=-D_{ba}.
$$
Repeated indices range over the finite base and fiber dimensions. Direct contraction and pullback give
\begin{equation}\label{ch03:eq:homotopy-coefficients}
 J_t:=t^{-1}m_t^*(\iota_{R_{\rad}}\omega_G)
 =
 z^aC_{ai}(x,tz)\,\dd x^i
 +t z^aD_{ab}(x,tz)\,\dd z^b.
\end{equation}
For any scalar coefficient $W$,
$$
 \partial_x^\alpha\partial_z^\beta
       \bigl(z^aW(x,tz)\bigr)
 =
 z^a t^{|\beta|}
       (\partial_x^\alpha\partial_z^\beta W)(x,tz)
 +\beta_a t^{|\beta|-1}
       (\partial_x^\alpha\partial_z^{\beta-e_a}W)(x,tz).
$$
The second summand is present only when $\beta_a\geq1$; its exponent is then nonnegative.

Invariance under multiplication by $-1$ gives $C_{ai}(x,-z)=-C_{ai}(x,z)$, so $C_{ai}(x,0)=0$. By the fundamental theorem of calculus, we have the stronger expression
$$
 z^aC_{ai}(x,tz)
 =
 t z^az^c\int_0^1
          (\partial_{z^c}C_{ai})(x,vtz)\,\dd v.
$$
Differentiation in $x,z$ preserves the initial factor $t$. Derivatives of $z^az^c$ have degree at most two; derivatives of the remaining coefficient contribute factors $vt$ of absolute value at most one. By Corollary~\ref{ch01:prop:coefficient-growth},
$$
 \bigl|\partial_x^\alpha\partial_z^\beta
                  (J_t)_{\mathrm{coeff}}\bigr|
 \leq Ct(1+\|q\|)^M,\qquad 0<t\leq1.
$$
The vertical coefficients in \eqref{ch03:eq:homotopy-coefficients} satisfy this bound because of the factor $t$ in their displayed expression. Each derivative extends continuously to $t=0$ with value zero. On a compact neighborhood of any finite $(x,z)$, this is an integrable majorant. Difference quotients, written as integrals of the next derivative along coordinate segments, justify repeated differentiation under $b=\int_0^1J_t\,\dd t$. Integration of the bound proves all polynomial estimates for $b$.

In a complex linear frame the coefficients of $V_{\rot}=-IR_{\rad}$ are linear in $z$, with constant real matrix $I$. Since $\mu=b(V_{\rot})$, the product rule gives all mixed polynomial estimates for $\mu$. Under a frame change $z'=G(x)z$, derivatives of $G$ and $G^{-1}$ are bounded, and the additional factors are polynomial in $z$. The estimates therefore hold on the finite chart cover. On $\{\mu\geq1\}$, an $n$-fold derivative of $\mu^{-1}$ is a sum, over partitions into $k$ blocks, of
$$
 (-1)^k k!\,\mu^{-k-1}\prod_B D_B\mu.
$$
All inverse powers are bounded there. Multiplication by the coefficients of $b$ proves the assertion for $\eta$.
\end{proof}

The strict radial bound controls the inverse of the moment map on each ray. Combining it with Proposition~\ref{ch03:prop:homotopy-estimates} gives polynomial estimates in coordinates on level sets of $\mu$.

With the flat-area bounds $c_K,C_K$ introduced above, set
$$
 c_A=C_K^{-3/2},\qquad
 C_A=((1+A_0)/c_K)^{3/2}.
$$
We denote the base and angular coordinates of $(X,\psi)$ by $\zeta$.

\begin{prop}\label{ch03:prop:level-coordinate-estimates}
For $\tau\geq1$,
\begin{equation}\label{ch03:eq:radius-estimates}
 \begin{gathered}
 c_A\tau^{3/2}\leq a_{\rad}(\tau,\zeta)\leq C_A\tau^{3/2},\\
 \frac{2}{3C_A}\tau^{-1/2}
 \leq\partial_a\mu(X,a_{\rad}\psi)
 \leq\frac2{c_A}\tau^{-1/2},\qquad
 \bigl(\partial_a\mu(X,a_{\rad}\psi)\bigr)^{-1}
 \leq\frac{3C_A}{2}\tau^{1/2}.
\end{gathered}
\end{equation}
The coefficients of $a_{\rad},\ell^*b,\widehat\eta_\tau,\gamma_\tau,\Omega_\tau,\omega_\tau$, and $\ell^*\omega_G$, together with their mixed derivatives of any finite order in $\tau$ and $\zeta$, grow polynomially in $1+\tau$. In particular, for all nonnegative integers $j,m$,
\begin{equation}\label{ch03:eq:level-polynomial-bounds}
 \|\partial_\tau^jQ_\tau\|_{C^m(K)}
 +\|\partial_\tau^jU_\tau\|_{C^m(K)}
 \leq C(1+\tau)^M,\qquad \tau\geq1.
\end{equation}
\end{prop}

\begin{proof}
We first estimate the inverse of the moment map along each ray. In the coordinates $\zeta$ introduced above, set
$$
 f(a,\zeta)=\mu(X(\zeta),a\psi(\zeta)).
$$
Every derivative of finite order of $(a,\zeta)\mapsto(X(\zeta),a\psi(\zeta))$ is $O(1+a)$. The polynomial bounds just established imply that derivatives of $f$ through any prescribed order $n$ are bounded by $C_n(1+a)^{L_n}$, with $L_n\geq0$ nondecreasing. At $a=a_{\rad}(\tau,\zeta)$, \eqref{ch01:eq:coercivity} and \eqref{ch01:eq:radial} give
$$
 \left(\frac{\tau}{C_K}\right)^{3/2}
 \leq a_{\rad}
 \leq\left(\frac{\tau+A_0}{c_K}\right)^{3/2},
 \qquad
 \frac{2\tau}{3a_{\rad}}<f_a(a_{\rad},\zeta)<\frac{2\tau}{a_{\rad}}.
$$
For $\tau\geq1$, these inequalities imply \eqref{ch03:eq:radius-estimates}.

Repeated differentiation of $f(a_{\rad}(\tau,\zeta),\zeta)=\tau$ isolates the highest derivative:
$$
 f_a\,\partial^\beta a_{\rad}
 =\partial^\beta\tau-R_\beta.
$$
Here $R_\beta$ is a finite sum of products of derivatives of $f$ and derivatives of $a_{\rad}$ of order strictly below $|\beta|$. The polynomial bounds for $f$, the bound $a_{\rad}\leq C(1+\tau)^{3/2}$, and $f_a^{-1}\leq C(1+\tau)^{1/2}$ prove the claim by induction. In particular,
$$
 \partial_\tau a_{\rad}=f_a^{-1},\qquad
 \partial_{\zeta_i}a_{\rad}=-f_{\zeta_i}/f_a.
$$

The identity $b(R_{\rad})=0$ simplifies its pullback. If $b=B_i\,\dd x^i+C_a\,\dd z^a$, then $\dd z^a=\psi^a\dd a_{\rad}+a_{\rad}\dd\psi^a$, and $b(R_{\rad})=0$ gives $C_a(X,a_{\rad}\psi)\psi^a=0$. Hence, as forms on $(0,\infty)\times Y$,
$$
 \begin{aligned}
 \ell^*b
    &=B_i(X,a_{\rad}\psi)\,\dd x^i
       +a_{\rad} C_a(X,a_{\rad}\psi)\,\dd\psi^a,\\
 \widehat\eta_\tau
    &=\tau^{-1}
       \bigl(B_i(X,a_{\rad}\psi)\,\dd x^i
             +a_{\rad} C_a(X,a_{\rad}\psi)\,\dd\psi^a\bigr).
\end{aligned}
$$
The terms $\dd\psi^a$ include the base dependence of the frame but have no $\dd\tau$ component. The product and chain rules now give polynomial bounds, since derivatives of $\tau^{-1}$ are bounded on $\tau\geq1$.

Smooth local sections of $\varrho$ transfer the estimates to $\gamma_\tau$. Exterior differentiation gives the estimates for $\Omega_\tau$ and $\omega_\tau$. The product pullback \eqref{ch03:eq:full-level-pullback} then gives the estimates for $\ell^*\omega_G$. For mixed order $n$, it suffices to use $b$ through order $n+1$, $a_{\rad}$ through order $n+2$, and the frames and local sections through order $n+2$. The induction for the implicit equation uses derivatives of $\mu$ through these finite orders.

The exponential of a two-form on $\PP(\Eg)$ has at most $N-1$ two-form factors. Its derivatives are finite sums of products of the coefficients just controlled. The auxiliary parameter in \eqref{ch03:eq:level-transgression} ranges over $[0,1]$, so integration in that parameter preserves a polynomial bound. In a projective trivialization, choose a finite partition of unity with compact angular supports independent of the base point. Each coefficient of the pushforward is an integral of a coefficient of top vertical degree over these supports. Differentiation under this compact integral proves \eqref{ch03:eq:level-polynomial-bounds}.
\end{proof}

The coefficient bounds also give an integrable majorant in the original linear fiber coordinates.

Write the component of top vertical degree as
$$
 [(-1)^re^{\omega_G}]_{\mathrm{vert}=2r}
 =
 \sum_I \widehat F_I(x,z)\,
 \dd x^I\wedge\dd z^1\wedge\cdots\wedge\dd z^{2r},
 \qquad
 D_I(s,x,z)=e^{-s\mu(x,z)}\widehat F_I(x,z).
$$
Here $I$ is an increasing multi-index of even length between zero and $2d$. These coefficients include the orientation factor: in the notation of Corollary~\ref{ch01:prop:coefficient-growth}, $\widehat F_I=(-1)^rF_{k,I}$ when $|I|=2k$ in the same linear chart.

\begin{prop}\label{ch03:prop:analytic-estimates}
For every finite collection of $\alpha,\beta,j$ and every $0<s_0\leq s_1<\infty$, there are $C,M,a>0$ such that
\begin{equation}\label{ch03:eq:original-density-majorant}
 \left|
 \partial_x^\alpha\partial_z^\beta\partial_s^jD_I(s,x,z)
 \right|
 \leq C(1+|z|)^M e^{-a|z|^{2/3}},
 \qquad s\in[s_0,s_1].
\end{equation}
The right side is integrable on $\R^{2r}$, and
\begin{equation}\label{ch03:eq:original-density-tail}
 \sup_{x,s}\int_{\{\mu(x,z)>\Lambda\}}
 \left|
 \partial_x^\alpha\partial_z^\beta\partial_s^jD_I(s,x,z)
 \right|\,\dd z
 \leq Ce^{-s_0\Lambda/2},
 \qquad \Lambda\geq0.
\end{equation}
\end{prop}

\begin{proof}
By Corollary~\ref{ch01:prop:coefficient-growth}, all prescribed derivatives of $\widehat F_I$ are polynomially bounded; the derivatives of $\mu$ satisfy the same type of bound by Proposition~\ref{ch03:prop:homotopy-estimates}. Each derivative of $D_I$ is $e^{-s\mu}$ times a finite sum of products of these derivatives and powers of $s$ and $\mu$. Hence
$$
 \left|\partial_x^\alpha\partial_z^\beta\partial_s^jD_I\right|
 \leq C(1+\|q\|)^M e^{-s\mu}.
$$
Uniform equivalence of the Hermitian and coordinate norms, together with \eqref{ch01:eq:coercivity}, gives
$$
 e^{-s\mu}
 \leq e^{s_1A_0}e^{-s_0c_K\|q\|^{2/3}}
 \leq C'e^{-a|z|^{2/3}}.
$$
This proves \eqref{ch03:eq:original-density-majorant}.

The majorant is integrable on $\R^{2r}$ because $e^{-a|z|^{2/3}}$ decays faster than every inverse power of $|z|$. On $\{\mu>\Lambda\}$,
$$
 e^{-s\mu}
 =e^{-(s-s_0/2)\mu}e^{-(s_0/2)\mu}
 \leq e^{-s_0\Lambda/2}e^{-(s_0/2)\mu}.
$$
The remaining polynomial factor times $e^{-(s_0/2)\mu}$ has an integrable majorant by the same argument. This proves \eqref{ch03:eq:original-density-tail}.

These bounds hold on slightly larger parameter neighborhoods, so each difference quotient in a base or $s$ coordinate is controlled by the next derivative along a short segment. Dominated convergence therefore permits differentiation under the original fiber integral at every finite order and proves continuity of each derivative. Changing a base lift by a vertical vector inserts $2r+1$ vertical arguments and has no effect. Ordinary change of variables for absolutely integrable densities proves agreement under linear changes of fiber coordinates.
\end{proof}

\subsection{Removal of the moment cutoff}

The bounds in Proposition~\ref{ch03:prop:analytic-estimates} give absolute convergence in linear fiber coordinates, while those in Proposition~\ref{ch03:prop:level-coordinate-estimates} control derivatives in coordinates on the level sets of $\mu$. Together they justify removal of the sublevel cutoff.

Set
$$
 \Phi_s=(-1)^rp_*^\C e^{\omega_G-s\mu},
 \qquad
 T_s=2\pi\int_0^\infty e^{-s\tau}U_\tau\,\dd\tau,
$$
where the first integral is taken coefficient by coefficient in linear fiber charts. We prove convergence of the transgression integral and pass the finite localization identity to the limit. To state its integrated form, let $\beta$ be a smooth compactly supported closed homogeneous form on $\MG$ of real degree $\ell$.

\begin{prop}\label{ch03:prop:relative-localization}
For every $s>0$, the forms $\Phi_s,\Theta_s,T_s$ are jointly smooth in $(s,X)\in(0,\infty)\times\MG$, and
\begin{equation}\label{ch03:eq:relative-localization}
 \Phi_s
 =2\pi\int_0^\infty e^{-s\tau}Q_\tau\,\dd\tau
 =\Theta_s+\dM T_s.
\end{equation}
As $\Lambda\to\infty$, the forms $\Phi_{s,\Lambda}$, $\Theta_{s,\Lambda}$, and $T_{s,\Lambda}$ converge to $\Phi_s$, $\Theta_s$, and $T_s$, respectively, in $C^\infty_{\mathrm{loc}}((0,\infty)\times\MG)$. The forms $\Phi_s$ and $\Theta_s$ are closed.

For the form $\beta$ chosen above,
\begin{equation}\label{ch03:eq:compact-insertion}
 \begin{aligned}
 \int_{\HG}
     p^*\beta\wedge[e^{\omega_G-s\mu}]_{2N-\ell}
 &=
 \int_{\MG,\C}\beta\wedge[\Phi_s]_{2d-\ell}\\
 &=
 \int_{\MG,\C}\beta\wedge[\Theta_s]_{2d-\ell}.
\end{aligned}
\end{equation}
The integral over the total space is absolutely convergent. If $\ell=2d-2k$, $0\leq k\leq d$, its value is
$$
 \sum_{j=0}^k\left(\frac{2\pi}{s}\right)^{r+j}
 \int_{\MG}\beta\wedge
       \frac{\omega_0^{k-j}}{(k-j)!}
       \wedge h_j(\Eg,\nabla_{\refn}).
$$
For odd $\ell$, both selected components vanish. The identities extend to inhomogeneous closed compactly supported forms by linearity.
\end{prop}

\begin{proof}
On a compact base set and a compact interval $[s_0,s_1]\subset(0,\infty)$, the quotients $Q_\tau/\tau^{r-1}$ and $U_\tau/\tau^r$ extend smoothly to $\tau=0$ by Proposition~\ref{ch03:prop:finite-localization}. Every $s$-derivative inserts a power of $-\tau$, so all prescribed mixed derivatives are integrable there. At infinity, \eqref{ch03:eq:level-polynomial-bounds} bounds them by $Ce^{-s_0\tau}(1+\tau)^M$. In particular,
\begin{equation}\label{ch03:eq:laplace-tail}
 \int_\Lambda^\infty e^{-s_0\tau}(1+\tau)^M\,\dd\tau
 \leq
 e^{-s_0\Lambda}(1+\Lambda)^M
 \int_0^\infty e^{-s_0t}(1+t)^M\,\dd t.
\end{equation}
The same bounds on slightly larger parameter neighborhoods justify differentiation under the integrals. Hence the truncated forms converge with all derivatives in the base variables and in~$s$. Equations~\eqref{ch03:eq:reference-level-polynomial} and~\eqref{ch03:eq:laplace-factor} give the stated limit $\Theta_s$, and the finite localization identity passes to the limit.

In a linear bundle chart, each coefficient of $\Phi_{s,\Lambda}$ is the integral of $D_I$ over $\{\mu\leq\Lambda\}$. The tail bound \eqref{ch03:eq:original-density-tail} shows that these coefficients converge uniformly to the absolutely defined fiber integral $\Phi_s$. Thus this fiber integral is the same smooth limit as the integral over level sets of $\mu$, and \eqref{ch03:eq:relative-localization} follows. Closedness passes from the finite forms to their smooth limits.

For compactly supported $\beta$, the majorant \eqref{ch03:eq:original-density-majorant} gives absolute convergence and permits Fubini's theorem on a finite chart cover of its support. The first equality in \eqref{ch03:eq:compact-insertion} follows by the projection formula and the orientation factor $(-1)^r$ from Proposition~\ref{ch01:prop:normal-form}. If $\ell$ is even and $m=2d-\ell\geq1$, then
$$
 \dd\bigl(\beta\wedge[T_s]_{m-1}\bigr)
 =\beta\wedge[\dM T_s]_m.
$$
This form has compact support, so its integral is zero. For $m=0$, the degree-zero component of $\dM T_s$ is already zero. The second equality follows from \eqref{ch03:eq:relative-localization}, and the coefficient formula follows by selecting degree $2k$ in \eqref{ch03:eq:theta-definition}. Odd $\ell$ selects odd components of even forms and gives zero.
\end{proof}

\subsection{A potential for the transgression}

We now use the complex structure to express $T_s$ as a $\dc$-derivative. Let $I$ denote the complex structure of the holomorphic bundle $\mathbb E_3$, and put
$$
 \dc=\ii(\bar\partial-\partial).
$$
For a real function $f$, this means $\dc f=-\dd f\circ I$.

The bundle $\mathbb E_3$ parametrizes the cyclic $\mathrm{SL}(3,\R)$-Higgs bundles in the Hitchin section with quadratic differential zero. The Goldman form has type $(1,1)$ for this bundle complex structure by \cite[Section~7.1, Theorems~7.10 and~7.13(1), Corollary~7.14]{CTW2025}. Rescaling the cubic differential and the invariant pairing does not change this property.

Define the radial potential by
\begin{equation}\label{ch03:eq:radial-potential-definition}
 F_{\rad}(X,q)=\int_0^1\mu(X,aq)\,\frac{\dd a}{a}.
\end{equation}
For $\tau>0$ and a $\bar h$-unit representative $\psi$ of a point $[\psi]\in\PP(\Eg)$, we define
\begin{equation}\label{ch03:eq:legendre-definition}
 \mathcal K_\tau(X,[\psi])
 =
 \frac{F_{\rad}(X,a_{\rad}(\tau,\psi)\psi)}{\tau}
 -\log\frac{a_{\rad}(\tau,\psi)}{\sqrt\tau}-\frac12.
\end{equation}
Circle invariance makes this independent of the unit representative. The transgression potential uses the family
$$
 W_{\tau,v}
 =\varpi^*\omega_0
      +\tau\bigl(\Omega_{\refn}-v\,\dP\dPc\mathcal K_\tau\bigr)
$$
and the even base forms
\begin{equation}\label{ch03:eq:potential-transgression-definition}
 \begin{aligned}
 V_\tau
   &=
   \left[
   (-1)^{r-1}\varpi_*
       \left(\tau\mathcal K_\tau
                    \int_0^1e^{W_{\tau,v}}\,\dd v\right)
   \right]_{\leq2d-2},\\
 V_s&=2\pi\int_0^\infty e^{-s\tau}V_\tau\,\dd\tau.
\end{aligned}
\end{equation}
In this definition, projective integration precedes truncation to base degree at most $2d-2$. The operators $\dP$ and $\dPc$ act on the total space $\PP(\Eg)$, and $\dMc$ acts on $\MG$.

\begin{prop}\label{ch03:prop:transgression-potential}
The function $F_{\rad}$ is smooth on $\mathbb E_3$, invariant under the circle action, and zero on the zero section. It satisfies
\begin{equation}\label{ch03:eq:radial-potential-identities}
 R_{\rad}F_{\rad}=\mu,\qquad
 b=-\dc F_{\rad},\qquad
 \omega_G=p^*\omega_0-\dd\dc F_{\rad},\qquad
 \tfrac12\mu\leq F_{\rad}\leq\tfrac32\mu.
\end{equation}
The functions $\mathcal K_\tau$ extend smoothly to $\tau=0$ locally over $\MG$, with
$$
 \mathcal K_0(X,[\psi])
 =\frac12\log\int_X|\psi|_{\sigma_X}^2\,\dA_{\sigma_X},
$$
and the connection difference in \eqref{ch03:eq:connection-difference} satisfies
\begin{equation}\label{ch03:eq:legendre-connection}
 \gamma_\tau=-\dPc\mathcal K_\tau,\qquad
 \Omega_\tau=\Omega_{\refn}-\dP\dPc\mathcal K_\tau.
\end{equation}

The family $V_\tau/\tau^r$ is smooth at $\tau=0$. Locally over $\MG$, mixed derivatives of $V_\tau$ of any finite order grow polynomially as $\tau\to\infty$. Consequently, the integral defining $V_s$ converges with all derivatives in the base variables and in~$s$ locally on $(0,\infty)\times\MG$, and
\begin{equation}\label{ch03:eq:exact-potential-primitive}
 U_\tau=-\dMc V_\tau,\qquad
 T_s=-\dMc V_s.
\end{equation}
\end{prop}

\begin{proof}
The moment and its first fiber differential vanish at zero. In real linear fiber coordinates $z^1,\ldots,z^{2r}$, Taylor's integral remainder gives
$$
 \mu(X,z)=\sum_{i,j=1}^{2r}z^iz^j m_{ij}(X,z)
$$
with smooth coefficients $m_{ij}$. Substituting $az$ shows that the integrand in \eqref{ch03:eq:radial-potential-definition} and its derivatives of any finite order in $(X,z)$ are $O(a)$ on compact sets. The integral is therefore smooth, including at zero, and derivatives pass under it. Invariance follows from invariance of $\mu$. For $c>0$, a change of variable gives
$$
 F_{\rad}(X,cq)=\int_0^c\mu(X,vq)\,\frac{\dd v}{v}.
$$
Differentiating in $\log c$ at $c=1$ proves $R_{\rad}F_{\rad}=\mu$.

The type-$(1,1)$ property relates the radial primitive to this potential. For a tangent vector $A$,
$$
 (\iota_{R_{\rad}}\omega_G)(A)
 =\omega_G(IR_{\rad},IA)
 =-\omega_G(V_{\rot},IA)
 =\dd\mu(IA)
 =-\dc\mu(A).
$$
Each dilation $m_a$ is holomorphic, so pullback commutes with $\dc$. Integrating this identity in the defining radial homotopy gives $b=-\dc F_{\rad}$. Combining this with \eqref{ch03:eq:radial-identities} gives the formula for $\omega_G$.

Integration of the radial bounds \eqref{ch01:eq:radial} along a nonzero ray gives
\begin{equation}\label{ch03:eq:radial-dilation}
 a^2\mu(X,q)\leq\mu(X,aq)
                   \leq a^{2/3}\mu(X,q),\qquad0<a\leq1.
\end{equation}
The derivative of $\log\mu(X,e^tq)$ lies between $2/3$ and $2$; integration from $\log a$ to zero gives these inequalities. Integrating them against $\dd a/a$ proves the bounds on $F_{\rad}$ in \eqref{ch03:eq:radial-potential-identities}. All terms vanish when $q=0$.

We compute the connection difference in a local holomorphic frame of the tautological line bundle. This calculation gives the connection determined by the radial primitive on positive level sets of $\mu$; compare the reduction of invariant K\"ahler potentials in \cite[Lemma~6.2]{Bielawski2007Potentials}.

Let $e$ be a nonvanishing holomorphic local section of $\OO(-1)$, viewed as a holomorphic choice of a nonzero holomorphic cubic differential over a projective chart. Set $n=\|e\|_{\bar h}$, and let $a_e>0$ solve $\mu(a_e e)=\tau$. The unit representative is $e/n$, and uniqueness of the radius gives $a_{\rad}=a_en$. Thus
$$
 \mathcal K_\tau
 =\frac{F_{\rad}(a_e e)}{\tau}
       -\log a_e-\log n+\frac12\log\tau-\frac12.
$$
With $\tau$ held constant, define $G=F_{\rad}(a_e e)/\tau-\log a_e$. For a tangent vector $A_1$ to $\PP(\Eg)$, the radial terms cancel because $R_{\rad}F_{\rad}=\tau$ on the level:
$$
 \dd G(A_1)
 =\tau^{-1}
    \dd F_{\rad}|_{a_e e}\bigl(\dd m_{a_e}\bigl(\dd e(A_1)\bigr)\bigr).
$$
Here $\dd m_{a_e}$ differentiates the dilation while holding $a_e$ constant: it preserves base components and scales vertical components. The section $e$ and each dilation $m_a$ with $a$ held constant are holomorphic, so $\dd m_{a_e}\circ\dd e$ commutes with $I$. Consequently,
$$
 -\dc G(A_1)
 =\tau^{-1}
   b|_{a_e e}\bigl(\dd m_{a_e}\bigl(\dd e(A_1)\bigr)\bigr).
$$
Pulling $b$ back by the section with radius $a_e$ adds $\dd\log a_e$ times $b(R_{\rad})=0$. Thus the right side is the local pullback of the level connection.

The reference Chern connection fixes the remaining sign. Its connection form in the holomorphic frame $e$ is $\partial\log n^2$. In the unit frame $e/n$, this becomes
$$
 \partial\log n^2-\dd\log n
 =\partial\log n-\bar\partial\log n
 =\ii\dc\log n.
$$
For the clockwise action, the unitary connection form is $-\ii$ times the real circle connection. The pullback of $\alpha_{\refn}$ by $e/n$ is therefore $-\dc\log n$. Subtracting this from the level connection gives
$$
 \gamma_\tau
 =-\dc G+\dc\log n
 =-\dPc\mathcal K_\tau.
$$
Exterior differentiation gives the curvature identity in \eqref{ch03:eq:legendre-connection}.

The limit at $\tau=0$ follows directly from the formulas in terms of the squared radius. For a unit representative $\psi$, $\mu(t\psi)=t^2M(t^2,\psi)$ and $a_{\rad}^2=\tau H(\tau,\psi)$. Changing variable in the radial integral gives
$$
 \frac{F_{\rad}(a_{\rad}\psi)}{\tau}
 =\frac{H}{2}\int_0^1 M(v\tau H,\psi)\,\dd v.
$$
Hence
\begin{equation}\label{ch03:eq:legendre-zero-expression}
 \mathcal K_\tau
 =\frac{H}{2}\int_0^1M(v\tau H,\psi)\,\dd v
           -\frac12\log H-\frac12.
\end{equation}
This is smooth up to $\tau=0$ with all mixed derivatives. Using \eqref{ch03:eq:cubic-zero-end} gives the stated value of $\mathcal K_0$. The connection identity extends to $\tau=0$ by continuity.

The curvature identity also gives the transgression potential. The Chern curvature $\Omega_{\refn}$ and the form $\omega_0$ on the zero section have type $(1,1)$. Equation~\eqref{ch03:eq:legendre-connection} shows that $\Omega_\tau$ and $W_{\tau,v}$ have the same type. A closed form of pure type $(1,1)$ is separately $\partial$- and $\bar\partial$-closed, so $\dPc W_{\tau,v}=0$. Therefore
$$
 \dPc\bigl(\tau\mathcal K_\tau e^{W_{\tau,v}}\bigr)
 =\tau\dPc\mathcal K_\tau\wedge e^{W_{\tau,v}}.
$$
Integration along the compact complex projective fibers commutes with $\partial$, $\bar\partial$, and thus $\dc$. Locally, base derivatives pass under the compact integral, while vertical derivatives integrate to zero on the boundaryless fiber. These commutation identities have positive sign because the fibers have even real dimension. Substitution of $\gamma_\tau=-\dPc\mathcal K_\tau$ into \eqref{ch03:eq:level-transgression} proves $U_\tau=-\dMc V_\tau$. The discarded terms in $V_\tau$ have degree at least $2d$, so their $\dc$-derivatives vanish on the base.

It remains to justify the improper integral and its derivatives. Near $\tau=0$, \eqref{ch03:eq:legendre-zero-expression} is smooth. After pulling out the base exponential, projective vertical degree $2r-2$ requires at least $r-1$ curvature factors, each accompanied by $\tau$. The scalar prefactor in \eqref{ch03:eq:potential-transgression-definition} supplies another $\tau$. Thus $V_\tau/\tau^r$ extends smoothly to $\tau=0$.

On a compact subset of $\MG$ and for $\tau\geq1$, \eqref{ch03:eq:radius-estimates} gives $a_{\rad}\asymp\tau^{3/2}$. The bounds on $F_{\rad}$ imply
$$
 \frac12\leq
 J_\tau:=\frac{F_{\rad}(X,a_{\rad}(\tau,\psi)\psi)}{\tau}
 \leq\frac32,\qquad
 |\mathcal K_\tau|\leq C\bigl(1+\log(1+\tau)\bigr).
$$
The derivatives of $\mathcal K_\tau$ in base and projective coordinates are controlled by \eqref{ch03:eq:legendre-connection} and the polynomial bounds for $\gamma_\tau$. For the $\tau$ derivatives, direct differentiation of \eqref{ch03:eq:legendre-definition} gives
$$
 \partial_\tau\mathcal K_\tau
       =\frac{1/2-J_\tau}{\tau},
 \qquad
 \partial_\tau J_\tau
       =\frac{\partial_\tau a_{\rad}}{a_{\rad}}-\frac{J_\tau}{\tau}.
$$
The quantitative lower bound for $a_{\rad}$, its polynomial derivative bounds, and these two identities give all pure $\tau$ derivative bounds by induction. Mixed derivatives follow by differentiating the connection identity in these coordinates, using the bounds on all mixed derivatives of $\gamma_\tau$.

The finite wedge expansion of $e^{W_{\tau,v}}$, integration over $0\leq v\leq1$, and compact projective integration now give polynomial bounds for every derivative of $V_\tau$ of finite order. The integrand is $O(\tau^r)$ as $\tau\to0$, and the exponential factor gives the majorants of \eqref{ch03:eq:laplace-tail} at infinity. The integral defining $V_s$ therefore converges locally uniformly with derivatives of any finite order in the base variables and $s>0$. Passing $\dMc$ under this integral yields
$$
 -\dMc V_s
 =2\pi\int_0^\infty e^{-s\tau}(-\dMc V_\tau)\,\dd\tau
 =2\pi\int_0^\infty e^{-s\tau}U_\tau\,\dd\tau
 =T_s.
$$
\end{proof}

\section{Estimates in plumbing families}\label{ch04:sec:plumbing}

We estimate the weighted Goldman density as the base approaches an arbitrary stable curve whose nodes smooth at independent rates. Using the bundle $\bE_3$ of regular cubic differentials and residue sequence from Section~\ref{ch02:sec:cubic-residues}, we extend the estimates of Section~\ref{ch01:sec:cubic} to plumbing families and prove absolute integrability on the moduli quotient. We retain the reference metric and Chern connection from Section~\ref{ch03:sec:fiber-integration}, where $\Phi_s=\Theta_s+\dd_{\cM}T_s$ and $T_s=-\dd_{\cM}^cV_s$.

\subsection{Residue coordinates and norms of cubic differentials}\label{ch04:defn:plumbing-setup}

Fix a stable graph $\Gamma$ with $1\leq k\leq d$ edges. On the labeled cover induced by $B_\Gamma\to Q_\Gamma$, choose a compact subset of the stratum of curves with dual graph $\Gamma$. The plumbing construction in \cite[Construction~5.2 and Proposition~5.3]{Finski2020} gives finitely many charts with complex coordinates
$$
 (b,t_1,\ldots,t_k),\qquad b\in\C^{d-k},\qquad z_ew_e=t_e.
$$
Deformations of the normalization have smooth Beltrami representatives supported away from the branch disks. The construction applies to each stable normalization component. For local coordinates on the family and the moduli orbifold, see also \cite[Section~8, Theorems~9.11 and~10.1, Proposition~12.2, and Corollary~12.3]{HubbardKoch2013}.

After passing to relatively compact subcharts, we let $b$ range over compact coordinate sets and use branch disks of common radius $e^{-a}$, with $a>0$. The coordinates on a smooth fiber are
$$
 t_e=e^{-L_e+\ii\phi_e},\qquad
 z_e=e^{-v_e+\ii\theta_e},\qquad
 w_e=e^{-(L_e-v_e)+\ii(\phi_e-\theta_e)}.
$$
The $e$-th collar $\mathcal A_e$ is
$$
 a\leq v_e\leq L_e-a,\qquad
 \theta_e\in\R/(2\pi\Z),
$$
and its cylindrical length is $H_e=L_e-2a$. After shrinking the plumbing charts, we take $L_e\geq L_*>\max\{1,4a\}$, so that $L_e/2\leq H_e\leq L_e$. The lengths vary independently, with $(\phi_1,\ldots,\phi_k)\in(\R/2\pi\Z)^k$, and markings are chosen on finitely many bounded intervals in the twist angles. Unless additional dependence is stated, constants depend only on the genus, the compact chart data, and the indicated finite orders, and are independent of all coefficients of $q$ and all ratios $L_e/L_f$.

On the plumbing charts, the residue maps extend to holomorphic linear functionals $R_e(\cdot)$ on $\bE_3$ and admit a holomorphic frame in which
$$
 q=\sum_{e=1}^kR_e e_e+\sum_{\alpha=1}^{r-k}c_\alpha f_\alpha,\qquad
 R_e(e_f)=\delta_{ef},\qquad
 R_e(f_\alpha)=0.
$$
The value $R_e=R_e(q)$ is the constant Laurent coefficient of $q/(\dd z_e/z_e)^3$ on a smooth collar and restricts to the chosen branch residue on $t_e=0$.

The residue coefficients and the remaining coefficients of $q$ enter the estimates at different scales. We use the coefficient norms
\begin{equation}\label{ch04:eq:cubic-coefficient-scales}
 W^2=\sum_e|R_e|^2+\sum_\alpha|c_\alpha|^2,\qquad
 J^2=\sum_\alpha|c_\alpha|^2+\sum_eL_e^3|R_e|^2.
\end{equation}

\begin{prop}\label{ch04:prop:residue-frame}
Uniformly for $b$ in the chosen compact set, $L_e\geq L_*$, and all angular parameters, the flat area of Definition~\ref{ch01:def:flat-area} satisfies
\begin{equation}\label{ch04:eq:flat-scales}
 cJ^{2/3}\leq A_{\flatarea}(q)\leq CJ^{2/3},
\end{equation}
and
$$
 \left|
 \int_{\mathcal A_e}|q|^{2/3}
       -2\pi H_e|R_e|^{2/3}
 \right|\leq CW^{2/3}.
$$
For the hyperbolic metric on the smooth closed fiber,
\begin{equation}\label{ch04:eq:cubic-hyperbolic-bounds}
 c\left(W^2+\sum_eL_e^5|R_e|^2\right)
 \leq \int_X|q|_\sigma^2\,\dA_\sigma
 \leq C\left(W^2+\sum_eL_e^5|R_e|^2\right).
\end{equation}
\end{prop}

\begin{proof}
A local section of the third tensor power of the relative dualizing line bundle near the $e$-th node has the form
$$
 q=F(b,\widehat t_e,z_e,w_e)(\dd z_e/z_e)^3,
 \qquad t_e=z_ew_e,
$$
where $\widehat t_e$ denotes the remaining plumbing parameters. By Proposition~\ref{ch02:prop:cubic-extension}, such sections correspond to sections of the holomorphic direct-image bundle. By Cauchy's estimates, on a smaller polydisc there is a normally convergent expansion
$$
 F=\sum_{m,n\geq0}a_{mn}(b,\widehat t_e)z_e^mw_e^n.
$$
The diagonal terms define
$$
 R_e(q)=\sum_{n\geq0}a_{nn}(b,\widehat t_e)t_e^n.
$$
This is holomorphic in all parameters and equals the constant Laurent coefficient on the smooth annulus. Its restriction to $t_e=0$ is the branch residue.

At $t_1=\cdots=t_k=0$, the residue exact sequence makes the joint map $(R_1,\ldots,R_k)$ surjective, so some $k\times k$ minor remains invertible after shrinking. Its inverse gives holomorphic sections $e_e$, and subtraction of their residue components from the remaining frame vectors gives $f_\alpha$. The residue sequence applies at every stable vertex, including those of genus zero or one.

Regrouping the remaining Taylor terms gives
\begin{equation}\label{ch04:eq:laurent-splitting}
 q=\bigl(R_e+F_{e,+}(z_e)+F_{e,-}(w_e)\bigr)(\dd z_e/z_e)^3,
 \qquad
 |F_{e,+}|+|F_{e,-}|
 \leq CW(|z_e|+|w_e|).
\end{equation}
The branch functions depend holomorphically on the base parameters and satisfy $F_{e,+}(0)=F_{e,-}(0)=0$. On a common smaller polydisc, the frame and its inverse are bounded, giving the displayed estimates uniformly in the plumbing parameters.

We choose a compact subdomain with nonempty interior in every normalization component, disjoint from the branch disks. By the identity theorem, restriction to their union is injective on the space of regular cubic differentials on the central nodal fiber, so both the flat integral and the squared norm for any smooth positive reference metric have positive minima on the unit coefficient sphere. Compactness in $b$ and continuity across the closed parameter polydisc preserve these minima after shrinking. Homogeneity gives the lower bounds $cW^{2/3}$ and $cW^2$ on the thick part, while smoothness gives the upper bounds.

For $p=2/3$,
$$
 \bigl||u|^p-|v|^p\bigr|\leq |u-v|^p
 \qquad(u,v\in\C).
$$
Applying this inequality to \eqref{ch04:eq:laurent-splitting} and integrating in $v_e,\theta_e$ gives the asserted collar error, since
$$
 \int_a^{L_e-a}
 \bigl(e^{-pv_e}+e^{-p(L_e-v_e)}\bigr)\,\dd v_e\leq C.
$$
The collars are disjoint. Thus
$$
 A_{\flatarea}\leq
 C\left(W^{2/3}+\sum_eL_e|R_e|^{2/3}\right).
$$
The lower bound on the thick part controls $W^{2/3}$ by the flat area. The collar estimate then controls each $L_e|R_e|^{2/3}$ by a constant times $A_{\flatarea}+W^{2/3}$. Summing over the finitely many edges gives the reverse inequality. Finite-dimensional norm comparison, with $L_e\geq1$, shows that $W^{2/3}+\sum_eL_e|R_e|^{2/3}$ is comparable to $J^{2/3}$, proving \eqref{ch04:eq:flat-scales}.

We next compare the hyperbolic norms. On $\mathcal A_e$, set
$$
 \zeta_e=v_e-\ii\theta_e,\qquad
 \rho_{L_e}(v_e)=\frac{\pi}{L_e}
                   \csc\left(\frac{\pi v_e}{L_e}\right),\qquad
 g_{\mathrm{mod},e}=\rho_{L_e}^2|\dd \zeta_e|^2.
$$
The regularity and grafted-metric construction of~\cite[Theorem~4, Proposition~8, and~(5.1)]{MelroseZhu2017} apply to this plumbing family, since each punctured normalization component is stable. On a compact subset of the metric resolution, the hyperbolic and grafted metrics are positive Hermitian metrics on the rescaled vertical line bundle of~\cite[Lemma~6]{MelroseZhu2017}, so their ratio and its reciprocal are bounded. The model $g_{\mathrm{mod},e}$, written in (5.1) of that reference, is also positive on this line bundle near the resolved node, while the transition regions have bounded widths and lie in compact subsets of the resolution. Consequently,
\begin{equation}\label{ch04:eq:hyperbolic-model-comparison}
 C^{-1}g_{\mathrm{mod},e}\leq\sigma\leq Cg_{\mathrm{mod},e}
 \quad\text{on }\mathcal A_e,
\end{equation}
and $\sigma$ is uniformly comparable to a smooth reference metric on the thick part, with constants independent of the rates at which the lengths tend to infinity.

For a cubic differential $F(\dd \zeta_e)^3$, the density of its squared norm in the model metric is
$$
 |F|^2\rho_{L_e}^{-4}\,\dd v_e\,\dd\theta_e.
$$
The cubic differential with constant coefficient on the complete model cylinder has integral
$$
 2\pi|R_e|^2\left(\frac{L_e}{\pi}\right)^4
       \int_0^{L_e}\sin^4\left(\frac{\pi v}{L_e}\right)\,\dd v
 =\frac{3|R_e|^2L_e^5}{4\pi^3}.
$$
Removing the two intervals of length $a$ changes this norm by at most $C|R_e|^2$. The terms with positive and negative Laurent exponents in \eqref{ch04:eq:laurent-splitting} are orthogonal to the constant term and to each other after angular integration. Moreover, $\rho_{L_e}^{-1}\leq C\min(v_e,L_e-v_e)$, so their squared integrals are at most $CW^2$. Hence
\begin{equation}\label{ch04:eq:collar-cubic-norm-asymptotic}
 \int_{\mathcal A_e}|q|_{g_{\mathrm{mod},e}}^2\,
       \dA_{g_{\mathrm{mod},e}}
 =\frac{3|R_e|^2L_e^5}{4\pi^3}+O(W^2),
\end{equation}
with an absolute remainder bounded by $CW^2$. The lower bound on the thick part, angular orthogonality, and \eqref{ch04:eq:hyperbolic-model-comparison} prove \eqref{ch04:eq:cubic-hyperbolic-bounds}.
\end{proof}

\subsection{The Wang equation and its linearization on a closed surface}

On each smooth fiber, write
$$
 I_q=\int_XQ\,\dA_\sigma,
 \qquad V_u=2e^u+4Qe^{-2u},
 \qquad L_u=\Delta_\sigma+V_u.
$$

\begin{prop}\label{ch04:prop:wang-comparison}
On every smooth fiber,
\begin{equation}\label{ch04:eq:wang-bounds}
 \begin{gathered}
 0\leq u,\qquad \max(1,Q^{1/3})\leq e^u,
 \qquad 0\leq Qe^{-2u}\leq e^u,\\
 2e^u\leq V_u\leq6e^u,
 \qquad V_u\geq4Qe^{-2u}.
 \end{gathered}
\end{equation}
The moment satisfies
\begin{equation}\label{ch04:eq:moment-comparison}
 \mu=\int_XQe^{-2u}\,\dA_\sigma\leq I_q,
 \qquad
 \max(0,cJ^{2/3}-A_0)\leq\mu\leq CJ^{2/3}.
\end{equation}
The constants are independent of the coefficients of $q$ and of all ratios $L_e/L_f$.
\end{prop}

\begin{proof}
By Propositions~\ref{ch01:prop:wang} and~\ref{ch01:prop:blaschke-comparison}, $u\geq0$ and $Qe^{-3u}\leq1$, which imply \eqref{ch04:eq:wang-bounds}. Equation~\eqref{ch01:eq:area} gives the expression for $\mu$. The bounds
$$
 A_{\flatarea}\leq A_B\leq A_{\flatarea}+A_0
$$
and \eqref{ch04:eq:flat-scales} then give \eqref{ch04:eq:moment-comparison}.
\end{proof}

Since $V_u\geq2$, the $L^\infty$ and $L^2$ inverse bounds of Lemma~\ref{ch01:lem:linearization} hold on every smooth plumbing fiber with the same constant $1/2$. For smooth $v$ satisfying $L_uv=f$, integration by parts gives
$$
 \int_X\bigl(|\dd v|_\sigma^2+V_uv^2\bigr)\,\dA_\sigma
 =\int_Xfv\,\dA_\sigma.
$$

We denote by $\dnu_{\coeff}$ the Lebesgue density in the real and imaginary parts of the adapted coefficients $R_e,c_\alpha$. It serves as the coordinate density when we express differential forms in this frame.

\subsection{Vertical variations and the symplectic fiber measure}

For cubic differentials $p_1=f_1\dd z^3$ and $p_2=f_2\dd z^3$, the pairing is
$$
 \pair{p_1}{p_2}_\sigma=\overline{f_1}f_2\rho_\sigma^{-6},
$$
which is linear in the second argument. On a smooth fiber, the variation of the solution $u$ of Wang's equation in a real vertical direction $p$ is
$$
 u_p=\left.\frac{\dd}{\dd\epsilon}\right|_{\epsilon=0}
           u(q+\epsilon p).
$$
Set
$$
 \mathcal J_q(p)=\int_Xe^{-2u}|p|_\sigma^2\,\dA_\sigma,\qquad
 E_q(v)=\int_X\bigl(|\dd v|_\sigma^2+V_uv^2\bigr)\,\dA_\sigma,
 \qquad
 \Omega_X=\omega_G|_{(\mathbb E_3)_X}.
$$

\begin{lem}\label{ch04:lem:vertical-variation}
For every holomorphic cubic differential $q$ and every real vertical tangent $p$,
\begin{equation}\label{ch04:eq:vertical-linearization}
 L_u u_p=4e^{-2u}\Re\pair q p_\sigma,\qquad
 E_q(u_p)\leq4\mathcal J_q(p).
\end{equation}
The variation $\alpha_p$ of the Blaschke connection in the direction $p$ is
\begin{equation}\label{ch04:eq:vertical-connection-variation}
 \begin{aligned}
 \alpha_p(U)V={}&
 \frac12\bigl(\dd u_p(U)V+\dd u_p(V)U
               -\sigma(U,V)\grad_\sigma u_p\bigr)\\
 &+\frac1{\sqrt2}e^{-u}\sigma^{-1}\Re(p-u_pq)(U,V),
 \end{aligned}
\end{equation}
and
\begin{equation}\label{ch04:eq:vertical-connection-energy}
 \int_X|\alpha_p|_\sigma^2\,\dA_\sigma
 +2\int_X(e^u+Qe^{-2u})u_p^2\,\dA_\sigma
 =2\mathcal J_q(p).
\end{equation}
The tensor norm uses the metric $\sigma$ and the Hilbert--Schmidt norm on $\End(TX)$.

\end{lem}

\begin{proof}
With $X,\sigma$, and the frame of $\mathbb E_3$ fixed, the Wang solution is differentiable by Lemma~\ref{ch01:lem:linearization}. The derivative of $Q$ is $2\Re\pair q p_\sigma$. Differentiating \eqref{ch01:eq:wang} gives the first equation in \eqref{ch04:eq:vertical-linearization}. The equation holds on the smooth closed fiber, so we may test it against $u_p$.

Multiplication by $u_p$ and integration give
$$
 E_q(u_p)=4\int_Xe^{-2u}\Re\pair q p_\sigma u_p\,\dA_\sigma.
$$
By Cauchy--Schwarz and $V_u\geq4Qe^{-2u}$,
$$
 E_q(u_p)\leq
 4\left(\int_XQe^{-2u}u_p^2\,\dA_\sigma\right)^{1/2}
 \mathcal J_q(p)^{1/2}
 \leq2E_q(u_p)^{1/2}\mathcal J_q(p)^{1/2}.
$$
This proves the energy bound.

For $h=e^u\sigma$, the conformal connection difference is
$$
 (\nabla^{h}-\nabla^\sigma)_UV
 =\frac12\bigl(\dd u(U)V+\dd u(V)U
                   -\sigma(U,V)\grad_\sigma u\bigr).
$$
Differentiation of this expression and of $e^{-u}\sigma^{-1}\Re q/\sqrt2$ gives the formula for $\alpha_p$.

We compute the tensor norms in an oriented $\sigma$-orthonormal frame $e_1,e_2$, writing
$$
 J_0=\begin{pmatrix}0&-1\\1&0\end{pmatrix},\qquad
 a=\dd u_p(e_1),\qquad b=\dd u_p(e_2).
$$
The two connection matrices of the conformal variation are
$$
 D_p(e_1)=\frac12(a\Id-bJ_0),\qquad
 D_p(e_2)=\frac12(b\Id+aJ_0).
$$
For $F=A+\ii B$, the two matrices associated to the real part of the cubic differential $F(e^1+\ii e^2)^3$ are
$$
 S_1(F)=
 \begin{pmatrix}A&-B\\-B&-A\end{pmatrix},\qquad
 S_2(F)=
 \begin{pmatrix}-B&-A\\-A&B\end{pmatrix}.
$$
These symmetric trace-free matrices are orthogonal to the span of $\Id,J_0$, and their squared Hilbert--Schmidt norms sum to $4|F|^2$, whereas those of the two $D_p$ matrices sum to $a^2+b^2$. It follows that
\begin{equation}\label{ch04:eq:pointwise-connection-norm}
 |\alpha_p|_\sigma^2
 =|\dd u_p|_\sigma^2+2e^{-2u}|p-u_pq|_\sigma^2.
\end{equation}

Write
$$
 H_p=\int_XQe^{-2u}u_p^2\,\dA_\sigma,\qquad
 M_p=\int_Xe^uu_p^2\,\dA_\sigma.
$$
Expanding $|p-u_pq|_\sigma^2$ and using the energy identity gives
$$
 \int_Xe^{-2u}|p-u_pq|_\sigma^2\,\dA_\sigma
 =\mathcal J_q(p)-\frac12E_q(u_p)+H_p.
$$
Since $E_q(u_p)=\int|\dd u_p|^2+2M_p+4H_p$, substitution proves \eqref{ch04:eq:vertical-connection-energy}. It also gives $\int_Xe^{-2u}|p-u_pq|_\sigma^2\,\dA_\sigma\leq \mathcal J_q(p)$.
\end{proof}

\begin{prop}\label{ch04:prop:vertical-estimates}
For every holomorphic cubic differential $q$ and all real vertical tangents $p,p'$,
\begin{equation}\label{ch04:eq:vertical-pairing}
 \Omega_X(p,p')=
 -2\int_Xe^{-2u}
     \Im\pair{p-u_pq}{p'-u_{p'}q}_\sigma\,\dA_\sigma,
\end{equation}
and
\begin{equation}\label{ch04:eq:vertical-taming}
 \begin{split}
 |\Omega_X(p,p')|&\leq2\sqrt{\mathcal J_q(p)\mathcal J_q(p')},\\
 \frac23\mathcal J_q(p)&\leq-\Omega_X(p,\ii p)\leq2\mathcal J_q(p).
 \end{split}
\end{equation}
\end{prop}

\begin{proof}
Let $S(F)$ denote the matrix-valued one-form whose components are $S_1(F)$ and $S_2(F)$, as in the proof of Lemma~\ref{ch04:lem:vertical-variation}. Direct multiplication with the matrices $D_p$ gives
$$
 \tr_2(D_p\wedge D_{p'})=0,\qquad
 \tr_2(D_p\wedge S)=\tr_2(S\wedge D_{p'})=0,
$$
and
$$
 \tr_2(S(F)\wedge S(G))(e_1,e_2)
 =-4\Im(\overline FG).
$$
Substitution into \eqref{ch01:eq:trace-pairing} gives \eqref{ch04:eq:vertical-pairing}. The weighted $L^2$ bound for $p-u_pq$ from Lemma~\ref{ch04:lem:vertical-variation}, followed by Cauchy--Schwarz, gives the first inequality in \eqref{ch04:eq:vertical-taming}.

To prove the lower bound in \eqref{ch04:eq:vertical-taming}, set
$$
 v=u_p,\quad w=u_{\ii p},\quad
 a_0=\Re\pair q p_\sigma,\quad
 b_0=\Re\pair q{\ii p}_\sigma.
$$
The pairing is linear in its second argument, so
$$
 \Im\pair{p-vq}{\ii p-wq}_\sigma
 =|p|_\sigma^2-va_0-wb_0.
$$
Hence
$$
 -\frac12\Omega_X(p,\ii p)
 =\mathcal J_q(p)-\frac14\bigl(E_q(v)+E_q(w)\bigr).
$$
For a real solution $L_u z=f$, we obtain by weighted Cauchy--Schwarz in the energy identity
$$
 E_q(z)\leq
 \left(\int_X\frac{f^2}{V_u}\,\dA_\sigma\right)^{1/2}
 E_q(z)^{1/2}.
$$
Thus $E_q(z)\leq\int_X f^2/V_u\,\dA_\sigma$. Applying this to $v,w$, and using $a_0^2+b_0^2=Q|p|_\sigma^2$, yields
$$
 E_q(v)+E_q(w)
 \leq\int_X\frac{16Qe^{-4u}}{V_u}|p|_\sigma^2\,\dA_\sigma.
$$
By \eqref{ch04:eq:wang-bounds}, $c_0=Qe^{-3u}$ belongs to $[0,1]$. The last integrand is
$$
 \frac{16c_0}{2+4c_0}\,e^{-2u}|p|_\sigma^2
 \leq\frac83e^{-2u}|p|_\sigma^2.
$$
This proves the lower bound in \eqref{ch04:eq:vertical-taming}. Nonnegativity of both energies gives the upper bound.
\end{proof}

The estimates imply that the fibers of $\mathbb E_3$ are symplectic; their vertical volume form is
$$
 \dV_X=(-1)^r\frac{\Omega_X^r}{r!}.
$$
For $p=\sum_eR_{p,e}e_e+\sum_\alpha c_{p,\alpha}f_\alpha$, put
$$
 W_p^2=\sum_e|R_{p,e}|^2+\sum_\alpha|c_{p,\alpha}|^2.
$$

\begin{cor}\label{ch04:cor:vertical-density}
For every holomorphic cubic differential $q$, the form $\Omega_X$ is symplectic, and $\dV_X$ is positive in the complex fiber orientation. For every real vertical tangent $p$,
$$
 \mathcal J_q(p)\leq C\left(W_p^2+\sum_eL_e^5|R_{p,e}|^2\right).
$$
Thus the connection variation $\alpha_p$ has $L^2$ norm at most $CL_e^{5/2}$ when $p=e_e$ or $p=\ii e_e$, and at most $C$ when $p=f_\alpha$ or $p=\ii f_\alpha$.
\end{cor}

\begin{proof}
The form $\Omega_X$ is closed and has type $(1,1)$. Since $\mathcal J_q(p)>0$ for $p\ne0$, \eqref{ch04:eq:vertical-taming} makes $-\Omega_X$ positive definite. Thus $\Omega_X$ is symplectic and $\dV_X=(-\Omega_X)^r/r!$ is positive in the complex orientation.

Finally, $u\geq0$ gives $\mathcal J_q(p)\leq\int_X|p|_\sigma^2\,\dA_\sigma$. Applying \eqref{ch04:eq:cubic-hyperbolic-bounds} to $p$ and then \eqref{ch04:eq:vertical-connection-energy} proves the norm estimates.
\end{proof}

\begin{cor}\label{ch04:cor:vertical-mass}
For every smooth $X$ and every $s>0$,
\begin{equation}\label{ch04:eq:vertical-mass}
 \int_{(\mathbb E_3)_X}e^{-s\mu}\,\dV_X=\left(\frac{2\pi}{s}\right)^r.
\end{equation}
On the plumbing charts, for every $A\geq0$,
$$
 \int_{\{J\geq A\}}e^{-s\mu}\,\dV_X
 \leq\left(\frac{4\pi}{s}\right)^r
       \exp\left(\frac{sA_0}{2}-\frac{scA^{2/3}}2\right).
$$
The tail is uniform over the plumbing parameters and locally uniform for positive $s$.
\end{cor}

\begin{proof}
At each smooth $X$, the fiber integral converges absolutely by Proposition~\ref{ch03:prop:relative-localization}. In degree zero, \eqref{ch03:eq:relative-localization} and \eqref{ch03:eq:theta-definition} read
$$
 (-1)^r\int_{(\mathbb E_3)_X}e^{-s\mu}\frac{\Omega_X^r}{r!}
 =\left(\frac{2\pi}{s}\right)^r.
$$
The density in \eqref{ch04:eq:vertical-mass} is positive by Corollary~\ref{ch04:cor:vertical-density}. On $J\geq A$, split the weight into two factors $e^{-s\mu/2}$. The lower bound for $\mu$ in \eqref{ch04:eq:moment-comparison} controls one factor by the displayed exponential, and \eqref{ch04:eq:vertical-mass} at $s/2$ integrates the other. On $s\in[s_0,s_1]$, the resulting bound is at most
$$
 \left(\frac{4\pi}{s_0}\right)^r
 e^{s_1A_0/2-s_0cA^{2/3}/2}.
$$
\end{proof}

\subsection{Beltrami representatives and the horizontal metric}

We estimate the horizontal metric as the base curve varies, using real stratum coordinates $b^1,\ldots,b^{2d-2k}$.

\begin{lem}\label{ch04:lem:beltrami}
Every real tangent vector to the plumbing base
$$
 A=(A_b,A_{L_1},A_{\phi_1},\ldots,A_{L_k},A_{\phi_k})
$$
has a smooth Beltrami representative $\nu_A$, linear in $A$, such that
\begin{equation}\label{ch04:eq:beltrami-bound}
 \norm{\nu_A}_\infty
 \leq C\left(
 |A_b|+\sum_e\frac{|A_{L_e}|+|A_{\phi_e}|}{L_e}
 \right).
\end{equation}
The representatives of $\partial_{L_e}$ and $\partial_{\phi_e}$ are supported on the $e$-th collar and vanish near both ends.
\end{lem}

\begin{proof}
In the plumbing construction~\cite[Construction~5.2]{Finski2020}, we represent deformations of each normalization component by Beltrami differentials supported away from the branch disks and uniformly bounded on compact stratum coordinate sets. These representatives vanish on every collar and preserve holomorphic branch coordinates; the complex structure on the complementary thick region is independent of all smoothing parameters.

We choose a reference collar of length $H_0=L_0-2a>0$ with coordinates $(\xi_e,\vartheta_e)\in[0,1]\times\R/(2\pi\Z)$, where its cylindrical coordinates are $v_e=a+H_0\xi_e$ and $\theta_e=\vartheta_e$. Let $\kappa\colon[0,1]\to[0,1]$ be smooth and nondecreasing, equal to zero near zero and one near one. After increasing $L_*$ so that $L_e\geq2L_0$, the marking on the $e$-th collar is defined by
$$
 v_e=a+H_0\xi_e+(L_e-L_0)\kappa(\xi_e),\qquad
 \theta_e=\vartheta_e+\phi_e\kappa(\xi_e).
$$
Its radial derivative is $H_0+(L_e-L_0)\kappa'>0$, so it preserves orientation. Near the left end, $z_e$ is independent of $L_e$ and $\phi_e$. Near the right end, their contributions to $w_e=e^{-L_e+\ii\phi_e}/z_e$ cancel, so $w_e$ is also independent of these parameters. The marking therefore joins the thick marking with parameter-independent transitions.

In the current conformal coordinate $\zeta_e=v_e-\ii\theta_e$, the infinitesimal vector coefficients for $\partial_{L_e}$ and $\partial_{\phi_e}$ are $\kappa(\xi_e(v_e))$ and $-\ii\kappa(\xi_e(v_e))$, respectively. Their Beltrami coefficients have absolute value
$$
 \frac{\kappa'(\xi_e)}
      {2\bigl(H_0+(L_e-L_0)\kappa'(\xi_e)\bigr)}
 \leq\frac1{L_e}.
$$
The coefficients vanish near both ends, so extension by zero gives smooth representatives on the whole fiber. Since the marking of one collar is the identity on every other reference collar, real linear combinations give~\eqref{ch04:eq:beltrami-bound}. The construction is defined on bounded intervals in the twist angles, finitely many of which cover the twist torus.
\end{proof}

Let $I$ be the complex structure of the holomorphic bundle $\mathbb E_3$. For a base tangent vector $A$, let $h_A$ be its $\omega_G$-orthogonal lift relative to the vertical tangent space. This lift exists and is unique by Corollary~\ref{ch04:cor:vertical-density}.

\begin{prop}\label{ch04:prop:horizontal-metric}
The horizontal distribution is complex, and
$$
 g_H(A,A)=\omega_G(h_A,Ih_A)
$$
is positive definite. There is a constant $C$ independent of $X,q$ such that, for every smooth Beltrami representative $\nu$ of $A$,
\begin{equation}\label{ch04:eq:horizontal-comparison}
 g_H(A,A)\leq C\int_X|\nu|^2\,\dA_{h}.
\end{equation}
Consequently,
\begin{equation}\label{ch04:eq:horizontal-quadratic}
 g_H\leq C(A_0+\mu)
 \left(
 \sum_{\alpha=1}^{2d-2k}(\dd b^\alpha)^2+
 \sum_e\frac{(\dd L_e)^2+(\dd \phi_e)^2}{L_e^2}
 \right).
\end{equation}
\end{prop}

\begin{proof}
The affine connection determines the harmonic reduction and its Higgs field. Write $h=e^\psi|\dd z|^2$ and $q=f\,\dd z^3$. Relative to the metric $h\oplus1$ on $TX\oplus\R$, the affine connection of Definition~\ref{ch01:def:affine-connection} splits as
$$
 D_{\mathrm{aff}}=A_h+\Psi,\qquad A_h=\nabla^{h}\oplus\dd,
$$
where
$$
 \Psi_U(Z,a)=\left(\frac1{\sqrt2}h^{-1}\Re q(U,Z)+aU,
                        \ h(U,Z)\right).
$$
The tensor $\Psi$ is self-adjoint. In the complex frame
$$
 \bigl(2e^{-\psi}\partial_{\bar z},\,1,\,\partial_z\bigr),
$$
the Hermitian metric and the $(1,0)$ part $\varphi=\Psi^{1,0}$ are
$$
 H=\operatorname{diag}(2e^{-\psi},1,e^\psi/2),\qquad
 \varphi=
 \begin{pmatrix}
 0&0&f/(2\sqrt2)\\
 1&0&0\\
 0&1&0
 \end{pmatrix}\dd z.
$$
The frame is holomorphic for $A_h^{0,1}$. Since $f$ is holomorphic, $\bar\partial_{A_h}\varphi=0$. Flatness of the affine connection and its self-adjoint splitting give the Hitchin equation. This is the cyclic Higgs bundle in the stable Hitchin section, including at $q=0$.

The scalar correction passing from $D_{\mathrm{aff}}$ to $D_{\mathrm{vol}}$ is accompanied by multiplication of the metric by $e^{-2f_{\mathrm{vol}}/3}$. It changes the metric connection by $-\frac13\dd f_{\mathrm{vol}}\Id$ and leaves the self-adjoint endomorphism-valued form and its endomorphism norms unchanged. Thus the harmonic reduction represents the corresponding point of the Hitchin component in \eqref{ch01:eq:trace-pairing}.

In the notation of \cite{CTW2025}, the Higgs field is $\Phi=2\ii\varphi$. We use their Hermitian form $h_0$ on the joint moduli space and denote it by $\mathfrak h_{\mathrm{amb}}$. Writing $m=2\ii\nu$ for the variation of the complex structure, the argument of the projection in \cite[Lemma~5.9(2)]{CTW2025} is $(2\ii)^{-1}\Phi m=\varphi m$. Thus their horizontal projection formula reads
\begin{equation}\label{ch04:eq:horizontal-projection}
 -\mathfrak h_{\mathrm{amb}}(w,w)
 =
 \left\|
 P_{\ker D'}
 (\varphi m)
 \right\|_{L^2}^2,
 \qquad
 D'=\partial_{A_h}
       +[\varphi^*,\,\cdot\,].
\end{equation}
The form $\varphi^*$ is the $H$-adjoint, and the bracket is the graded commutator on adjoint-valued forms. The $L^2$ norm in \eqref{ch04:eq:horizontal-projection} uses the invariant pairing of \cite{CTW2025}; on $\mathfrak{sl}_3$ it is a fixed positive multiple of the trace norm used below. The vector $w$ is the orthogonal horizontal lift in the joint moduli space with respect to $\mathfrak h_{\mathrm{amb}}$, and $P_{\ker D'}$ is the orthogonal projection in the positive $L^2$ space of adjoint-valued one-forms. The projection formula holds for every smooth representative of the complex-structure variation by \cite[Propositions~5.3 and~5.7 and Lemma~5.9]{CTW2025}. The form $\mathfrak h_{\mathrm{amb}}$ is positive on the vertical tangent spaces of the joint moduli space by \cite[Theorem~4.6(2)]{CTW2025}. With the bundle metric $H$ unchanged, the $L^2$ inner product on one-forms is conformally invariant on a surface: if $h=e^u\sigma$, then
$$
 |\eta|_{\sigma,H}^2\,\dA_\sigma
 =|\eta|_{h,H}^2\,\dA_{h}.
$$
The projection $P_{\ker D'}$ is therefore independent of the choice between $\sigma$ and $h$, and either metric computes the norm in \eqref{ch04:eq:horizontal-projection}.

On the cyclic locus, the restriction of the two-form on the joint moduli space is a positive numerical multiple of $\omega_G$ by \cite[Theorems~7.10 and~7.13(1), Corollary~7.14]{CTW2025}. The Goldman form has type $(1,1)$, and the vertical tangent space is complex, so its symplectic orthogonal is complex as well.

For the norm on complex-valued one-forms, $|\dd z|_h^2=2e^{-\psi}$; the scalar norm of the cubic differential used throughout the paper is $|q|_h^2=|f|^2e^{-3\psi}$. The matrix for $\varphi$ therefore gives
\begin{equation}\label{ch04:eq:higgs-field-norm}
 |\varphi|_{h,\tr}^2
 =2e^{-\psi}
   \left(\frac{e^\psi}{2}+\frac{e^\psi}{2}
               +\frac{|f|^2e^{-2\psi}}2\right)
 =2+|q|_{h}^2\leq3
\end{equation}
by \eqref{ch04:eq:wang-bounds}.

At $q=0$, circle invariance makes the mixed base--fiber terms zero: the circle acts trivially on the base and by scalar rotation on the vertical tangent. Thus the horizontal metric there is the positive metric associated to $\omega_0=2\omega_{\WP}$. Along $aq$, $0\leq a\leq1$, the total form and its vertical restriction remain nondegenerate. The horizontal restriction is therefore nondegenerate throughout this path, so its Hermitian signature is constant. Hence $g_H$ is positive definite.

In the tangent space to the joint moduli space, $\xi=h_A$ and the $\mathfrak h_{\mathrm{amb}}$-orthogonal horizontal lift $w$ project to the same base tangent vector $A$. Hence $v=\xi-w$ is vertical, and $\xi=w+v$ is the orthogonal splitting defined by $\mathfrak h_{\mathrm{amb}}$. The summands are orthogonal and $\mathfrak h_{\mathrm{amb}}(v,v)\geq0$. The relation
$$
 \mathfrak h_{\mathrm{amb}}
 =2\bigl(\omega_{\mathrm{amb}}(I\,\cdot,\cdot)
                 +\ii\omega_{\mathrm{amb}}\bigr)
$$
therefore gives a positive numerical constant $c$ such that
$$
 \begin{aligned}
 c\,g_H(A,A)
 &=-\mathfrak h_{\mathrm{amb}}(\xi,\xi)
 \leq-\mathfrak h_{\mathrm{amb}}(w,w)\\
 &=\|P_{\ker D'}(\varphi m)\|_{L^2}^2
 \leq C\int_X|\nu|^2|\varphi|_{h,\tr}^2\,\dA_h
 \leq C\int_X|\nu|^2\,\dA_h.
 \end{aligned}
$$
Here the projection has norm at most one and \eqref{ch04:eq:higgs-field-norm} bounds the Higgs field. This proves \eqref{ch04:eq:horizontal-comparison}. Using the Beltrami representatives of Lemma~\ref{ch04:lem:beltrami} and $\Area(h)=A_0+\mu$, we obtain \eqref{ch04:eq:horizontal-quadratic}.
\end{proof}

\subsection{The Schur complement and fiber integration}

The inverse of the vertical symplectic form and the horizontal form give an intrinsic description of every coefficient of top vertical degree. We first establish the block identity and then apply the horizontal metric estimate.

Let $p^1,\ldots,p^{2d}$ be real base coordinates, and let $x^1,\ldots,x^{2r}$ be the real and imaginary parts of the coefficients of $q$ in the adapted frame, ordered in the complex orientation. For lifts $u_i$ of $\partial_{p^i}$, write
$$
 B_{ij}=\omega_G(u_i,u_j),\qquad
 C_{ia}=\omega_G(u_i,\partial_{x^a}),\qquad
 D_{ab}=\Omega_X(\partial_{x^a},\partial_{x^b}).
$$
Write $\dV_X=v_X\,\dnu_{\coeff}$ and
$$
 \frac{\omega_G^N}{N!}
 =F_{\mathrm{top}}\,
   \dd p^1\wedge\cdots\wedge \dd p^{2d}
   \wedge \dd x^1\wedge\cdots\wedge \dd x^{2r}.
$$

\begin{lem}\label{ch04:lem:schur}
The horizontal matrix is
\begin{equation}\label{ch04:eq:schur}
 H=B+CD^{-1}C^{\mathsf T}.
\end{equation}
It is independent of the chosen lifts, and
$$
 F_{\mathrm{top}}=(-1)^r\Pf(H)v_X.
$$
\end{lem}

\begin{proof}
For a base lift $u$, define the unique vertical $w_u$ by
$$
 \Omega_X(w_u,\cdot)
 =-\omega_G(u,\cdot)|_{\mathrm{vert}}.
$$
Then $h_u=u+w_u$ is horizontal. In matrix notation $w_u^{\mathsf T}D=-C_u$, so $w_u=D^{-1}C_u^{\mathsf T}$. Expanding $\omega_G(u+w_u,v+w_v)$ gives \eqref{ch04:eq:schur}. Replacing $u$ by $u+\zeta$, with $\zeta$ vertical, replaces $w_u$ by $w_u-\zeta$, so the horizontal lift is unchanged.

The change from $u_i,\partial_{x^a}$ to $h_i,\partial_{x^a}$ is triangular with determinant one. In the new frame the two-form has blocks $H,D$, with zero mixed block. Its top exterior power, divided by $N!$, is
$$
 \frac{H^d}{d!}\wedge\frac{\Omega_X^r}{r!}.
$$
Both block dimensions are even, and $\Omega_X^r/r!=(-1)^r\dV_X$. This proves the coefficient identity.
\end{proof}

The estimate in Proposition~\ref{ch04:prop:horizontal-metric} controls the Schur complement in the simultaneous plumbing coordinates. Write
$$
 p_L=(b^1,\ldots,b^{2d-2k},L_1,\phi_1,\ldots,L_k,\phi_k).
$$
For an increasing set $I$ of $2m$ base indices, define
$$
 n_e(I)=\mathbf1_{\{L_e\in I\}}+\mathbf1_{\{\phi_e\in I\}},
 \qquad
 w_I=\prod_eL_e^{-n_e(I)}.
$$
Let $F_I^L$ be the coefficient of $\dd p_L^I\wedge \dd x^1\wedge\cdots\wedge \dd x^{2r}$ in $\omega_G^{r+m}/(r+m)!$.

\begin{prop}\label{ch04:prop:density-estimates}
For $0\leq m\leq d$ and every index set $I$ with $|I|=2m$,
\begin{equation}\label{ch04:eq:full-vertical-coefficient}
 F_I^L\,\dnu_{\coeff}
 =(-1)^r\Pf(H[I])\,\dV_X,\qquad
 |F_I^L|\,\dnu_{\coeff}
 \leq C_m(A_0+\mu)^m w_I\,\dV_X.
\end{equation}
The empty Pfaffian is one. In particular,
\begin{equation}\label{ch04:eq:top-density-bound}
 |\Pf(H)|\leq C(A_0+\mu)^d\prod_eL_e^{-2},\qquad
 |F_{\mathrm{top}}^L|\,\dnu_{\coeff}
 \leq C(A_0+\mu)^d\prod_eL_e^{-2}\,\dV_X.
\end{equation}

\end{prop}

\begin{proof}
Let $Q_{\mathrm{base}}$ denote the diagonal quadratic form in parentheses in \eqref{ch04:eq:horizontal-quadratic}. By Hermitian Cauchy--Schwarz,
$$
 |H(A,B)|\leq\sqrt{g_H(A,A)g_H(B,B)}.
$$
In a $Q_{\mathrm{base}}$-orthonormal frame, every matrix entry of $H$ is therefore bounded by $C(A_0+\mu)$. The Pfaffian of a $2m\times2m$ skew matrix is a finite sum of products of $m$ entries. Returning to the coordinate frame $p_L$ multiplies its $I$-minor by $w_I$. This proves the Pfaffian estimates.

The triangular change of lifts in Lemma~\ref{ch04:lem:schur} also computes every coefficient of top vertical degree. The term with all $2r$ vertical covectors is the product of $\Omega_X^r/r!$ and the corresponding coefficient of $H^m/m!$. This proves \eqref{ch04:eq:full-vertical-coefficient}. Taking $m=d$ gives \eqref{ch04:eq:top-density-bound}.
\end{proof}

We integrate these coefficient bounds against the positive fiber density and evaluate the resulting moments of the Hamiltonian.

\begin{cor}\label{ch04:cor:density-integration}
For $0\leq m\leq d$, every index set $I$ with $|I|=2m$, every integer $j\geq0$, and every $s>0$,
\begin{equation}\label{ch04:eq:integrated-coefficient}
 \begin{split}
 \int_{(\mathbb E_3)_X}\mu^j e^{-s\mu}|F_I^L|\,\dnu_{\coeff}
 \leq{}&C_mw_I(2\pi)^r
 \sum_{\alpha=0}^m
 \binom m\alpha A_0^{m-\alpha}
 (r)_{j+\alpha}s^{-(r+j+\alpha)},
 \end{split}
\end{equation}
where $(r)_0=1$ and $(r)_n=r(r+1)\cdots(r+n-1)$. For each interval $[s_0,s_1]\subset(0,\infty)$, there are constants $C_{j,m}>0$ and $\eta>0$, depending only on $j,m$, this interval, and the compact chart data, such that for $s\in[s_0,s_1]$ and $M\geq0$,
\begin{equation}\label{ch04:eq:cubic-density-tail}
 \int_{\{J\geq M\}}\mu^j e^{-s\mu}|F_I^L|\,
          \dnu_{\coeff}
 \leq C_{j,m}w_Ie^{-\eta M^{2/3}}.
\end{equation}
Replacing the domain by $\{\mu\geq\Lambda\}$ gives the same bound with $e^{-s_0\Lambda/2}$, for $\Lambda\geq0$.

The forms
$$
 [\Phi_s]_{2m}
 =(-1)^r p_*^\C
       \left(e^{-s\mu}\frac{\omega_G^{r+m}}{(r+m)!}\right)
$$
are defined by absolutely convergent coefficient integrals. Their $j$-th $s$-derivatives are obtained by insertion of $(-\mu)^j$, with the preceding absolute bounds.
\end{cor}

\begin{proof}
For every integer $n\geq0$, positivity of $\dV_X$, $\mu\geq0$, and \eqref{ch04:eq:vertical-mass} give
\begin{equation}\label{ch04:eq:fiber-moments}
 \int_{(\mathbb E_3)_X}\mu^n e^{-s\mu}\,\dV_X
 =(2\pi)^r(r)_n s^{-(r+n)}.
\end{equation}
To justify differentiation of the mass identity, $\mu^ne^{-s\mu}$ is bounded for $s\in[s_0,s_1]\subset(0,\infty)$ by $C_{n,s_0}e^{-s_0\mu/2}$. The latter has finite mass by \eqref{ch04:eq:vertical-mass}. Dominated differentiation proves \eqref{ch04:eq:fiber-moments}. Expanding $(A_0+\mu)^m$ in \eqref{ch04:eq:full-vertical-coefficient} and applying \eqref{ch04:eq:fiber-moments} proves \eqref{ch04:eq:integrated-coefficient}.

On $J\geq M$, the moment lower bound gives
$$
 e^{-s\mu}\leq
 e^{s_1A_0/2-s_0cM^{2/3}/2}\,e^{-s\mu/2}.
$$
The moment formula at $s/2$ integrates the remaining polynomial factor and proves \eqref{ch04:eq:cubic-density-tail}. On $\mu\geq\Lambda$, the corresponding split yields $e^{-s_0\Lambda/2}$. The same estimates apply to the absolute values of all $s$-derivatives. By Proposition~\ref{ch03:prop:relative-localization}, these coefficient integrals are the smooth forms $[\Phi_s]_{2m}$ on $\cM_g$.
\end{proof}

The reciprocal length coordinates make the integrability of the top coefficient explicit.

\begin{rem}\label{ch04:rem:reciprocal-rates}
Put $x_e=L_e^{-1}$. For an absolute fiber integral of a coefficient of $[\Phi_s]_{2m}$, its $s$-derivatives, or its tail on $\{J\geq M\}$, the factor in the bound contributed by the $e$-th neck is given by the following power of $x_e$:
$$
 \begin{array}{c|c}
 \text{covectors from the $e$-th neck}&\text{rate}\\ \hline
 \text{neither }\dd x_e\text{ nor }\dd \phi_e&1\\
 \dd \phi_e\text{ only}&x_e\\
 \dd x_e\text{ only}&x_e^{-1}\\
 \dd x_e\text{ and }\dd \phi_e&1.
 \end{array}
$$
These factors multiply over the edges.

Since $\dd L_e=-x_e^{-2}\dd x_e$, each occurrence of $\dd L_e$ contributes a factor $x_e^{-2}$. Multiplying this factor with $L_e^{-n_e(I)}=x_e^{n_e(I)}$ gives the four rates in the table. A term containing $\dd x_e$ without $\dd \phi_e$ is bounded by $O(x_e^{-1})$. In the top-degree density, both normal coordinate differentials occur, and their combined factor is bounded.
\end{rem}

\begin{cor}\label{ch04:cor:corner-integrability}
Over a compact set of tangential and angular parameters, the weighted top-degree density is integrable over the fiber of $\mathbb E_3$ and all $L_e\geq L_*$. With insertion $\mu^j$, its mass on $\{L_e\geq\ell_e\text{ for all }e,\ J\geq M\}$ is at most
\begin{equation}\label{ch04:eq:corner-tail}
 C_j\prod_e\ell_e^{-1}\,e^{-\eta M^{2/3}},
 \qquad \ell_e\geq L_*,
\end{equation}
locally uniformly for positive $s$. The mass on the union of the deleted regions $\{L_e\geq\ell_e\}$ is at most $C_j\sum_e\ell_e^{-1}$. Rectangular base cutoffs and cutoffs in the fibers of $\mathbb E_3$, with bounds tending independently to infinity, can therefore be removed successively or jointly. Fubini's theorem, dominated convergence for bounded measurable tests, and every derivative with respect to $s$ are justified for the integrals of these top-degree densities.
\end{cor}

\begin{proof}
In the length variables, the bound of Corollary~\ref{ch04:cor:density-integration} is $C_j\prod_eL_e^{-2}e^{-\eta M^{2/3}}$. The separate integrals are
$$
 \int_{\ell_e}^\infty L_e^{-2}\,\dd L_e=\ell_e^{-1}.
$$
The estimate \eqref{ch04:eq:corner-tail} follows by Tonelli's theorem. Applying the same calculation to each deleted region and summing gives the union bound, with all remaining length integrals starting at $L_*$. The positive integrable majorants, available for every $j$, give the stated limit and differentiation assertions.
\end{proof}

\begin{lem}\label{ch04:lem:descent}
The scales $J$ are uniformly equivalent on relatively compact adapted chart overlaps. The bounds of Proposition~\ref{ch04:prop:density-estimates} and Corollaries~\ref{ch04:cor:density-integration} and~\ref{ch04:cor:corner-integrability} descend under branch exchanges, graph automorphisms, local uniformizing groups, and finite orbifold covers. Integration uses reciprocal stabilizer orders and covering degrees.
\end{lem}

\begin{proof}
The intrinsic estimate $A_{\flatarea}\asymp J^{2/3}$ proves equivalence of the scales. The coordinate transformations also preserve the weighted coefficient estimates, as follows.

After matching the labels of the boundary divisors, the normal coordinates on an overlap have the form
$$
 t'_e=u_e t_{\sigma(e)},
$$
where $u_e$ is a nowhere-vanishing holomorphic function. Thus $L'_e=L_{\sigma(e)}+O(1)$. The two residue functionals agree up to branch sign on the entire divisor $t_{\sigma(e)}=0$, including its partial smoothings. Their difference is therefore $t_{\sigma(e)}$ times a bounded holomorphic fiber functional. In particular,
$$
 (L'_e)^{3/2}
 |R'_e(q)\mp R_{\sigma(e)}(q)|
 \leq C L_{\sigma(e)}^{3/2}e^{-L_{\sigma(e)}}W,
$$
which is controlled by $CJ$ independently of all other lengths. The inverse coordinate change gives the reverse comparison.

Differentiation of $t'_e=u_et_{\sigma(e)}$ shows that $\dd L'_e/L'_e$ and $\dd \phi'_e/L'_e$ have bounded coefficients in the coframe
$$
 \left(\dd b^\alpha,\frac{\dd L_f}{L_f},
                         \frac{\dd \phi_f}{L_f}\right).
$$
The additional normal coefficients are bounded by $CL_fe^{-L_f}/L'_e$. Tangential coordinate changes have bounded tangential coefficients and normal coefficients bounded by $CL_fe^{-L_f}$. The inverse change obeys the same bounds, so the weighted coefficient estimates agree on overlaps.

A base-dependent complex linear change of frame in $\mathbb E_3$ changes a base lift by a vertical tangent. The horizontal form is unchanged by Lemma~\ref{ch04:lem:schur}. The vertical density transforms by the absolute real determinant of the fiber change, and the coordinate transformation has the matching triangular Jacobian. Thus the coefficient identities describe the same intrinsic forms.

Under a branch exchange, $\dd z_e/z_e=-\dd w_e/w_e$, so its cubic residue changes sign. Graph automorphisms act by the signed permutations proved in Proposition~\ref{ch02:prop:cubic-residues}. Markings on overlapping angular sectors differ by mapping classes, which preserve the geometric forms. Hence the estimates cover the twist torus.

On an orbifold chart, the lifted integral is divided by the order of the local uniformizing group; on the quotient by $\Aut(\Gamma)$, it is divided by $|\Aut(\Gamma)|$, with the vertex stabilizer weights retained. Branch exchanges are already included in these actions, and integration on a finite orbifold cover is divided by its stack degree.
\end{proof}

\subsection{Finiteness on the moduli quotient}

A finite atlas of the compactification by stable curves turns the local estimates into convergence of $Z_g(s)$ in the Goldman symplectic orientation.

\begin{prop}\label{ch04:prop:global-finiteness}
For every $s>0$, $Z_g(s)$ is finite and strictly positive. For every integer $j\geq0$,
$$
 \int_{\cH_g}\mu^je^{-s\mu}\frac{\omega_G^N}{N!}<\infty,
 \qquad
 \partial_s^jZ_g(s)
 =(-1)^j\int_{\cH_g}\mu^je^{-s\mu}\frac{\omega_G^N}{N!}.
$$
Thus $Z_g\in C^\infty(0,\infty)$. Moreover,
$$
 [\Phi_s]_{2d}\in L^1(\cM_g),\qquad
 Z_g(s)=\int_{\cM_g}[\Phi_s]_{2d}.
$$
All differentiations and base--fiber interchanges are absolutely dominated locally uniformly for positive $s$.

\end{prop}

\begin{proof}
The proper smooth stack $\bM_g$ has compact analytic underlying space \cite[Theorem~109.25.3, Tag~0E9C]{StacksR1}. Choose a finite cover by relatively compact interior orbifold charts and plumbing charts around stable curves. In each plumbing chart, the stratum parameters have compact closure and no additional nodes occur after shrinking.

By Corollary~\ref{ch04:cor:density-integration}, for every $j\geq0$ the absolute fiber integral of the top coefficient is bounded by
$$
 C_j\prod_eL_e^{-2}\,
 |\dd b^1\cdots\dd b^{2d-2k}|
 \prod_e|\dd L_e\,\dd\phi_e|,
$$
uniformly on compact positive $s$-intervals. This is integrable because the tangential and angular domains are bounded and $\int_{L_*}^\infty L^{-2}\,\dd L=L_*^{-1}$. On interior charts, the fiber integrals have an integrable majorant by Proposition~\ref{ch03:prop:analytic-estimates}. Moment insertions are controlled there by
$$
 \mu^j e^{-s\mu}\leq C_{j,s_0}e^{-s_0\mu/2},
 \qquad s\geq s_0>0.
$$
Summing with a finite partition of unity and the stabilizer weights proves absolute convergence for every $j$.

These majorants justify Fubini's theorem and all $s$-derivatives, locally uniformly for $s>0$. The orientation factor in $\Phi_s$ is $(-1)^r$ by Proposition~\ref{ch01:prop:normal-form}, so its top component is integrable and its integral equals $Z_g(s)$. The Goldman density is positive in its symplectic orientation; integration over any nonempty relatively compact open set therefore gives $Z_g(s)>0$. These integrals are compatible with the quotient integrations by Lemma~\ref{ch04:lem:descent}.
\end{proof}

The global moment bounds imply convergence under bounded measurable cutoffs and exponential decay of the moment tails.

\begin{cor}\label{ch04:cor:global-cutoffs}
For every integer $j\geq0$, the following assertions hold. If $|\rho_n|\leq1$ are measurable and $\rho_n\to1$ almost everywhere on $\cH_g$, then the integrals with insertion $\rho_n\mu^j$ converge to the corresponding moment, locally uniformly for positive $s$. For $s\in[s_0,s_1]\subset(0,\infty)$,
$$
 \int_{\{\mu\geq\Lambda\}}\mu^je^{-s\mu}
              \frac{\omega_G^N}{N!}
 \leq C_je^{-s_0\Lambda/2},\qquad\Lambda\geq0.
$$
\end{cor}

\begin{proof}
For the measurable cutoff assertion, the absolute error is bounded, uniformly on $[s_0,s_1]$, by
$$
 \int_{\cH_g}|\rho_n-1|\mu^je^{-s_0\mu}
                     \frac{\omega_G^N}{N!}.
$$
Its integrand tends to zero almost everywhere and is bounded by twice an integrable density. The locally uniform convergence follows by dominated convergence. Products of exhausting base and fiber cutoffs satisfy this assertion whenever they tend to one almost everywhere, whether the cutoffs are removed independently, successively, or jointly.

Finally, on $\mu\geq\Lambda$ and $s\geq s_0$,
$$
 e^{-s\mu}\leq e^{-s_0\Lambda/2}e^{-s_0\mu/2}.
$$
The finite global $j$-th moment at $s_0/2$ gives the last estimate.
\end{proof}

\section{Vanishing of the boundary integrals}
\label{ch05:sec:boundary}

We prove that the transgression constructed in Section~\ref{ch03:sec:fiber-integration} has vanishing boundary integrals along every exhaustion of $\cM_g$ by compact smooth orbifold domains with smooth boundary. Starting from the integrability established in Proposition~\ref{ch04:prop:global-finiteness}, we control its potential and its integrals over products of dyadic annuli in the plumbing coordinates. A smooth representative of $[\omega_0]$ on the compactification provides the correction in the global Stokes argument.

\subsection{The transgression and plumbing data}
\label{ch05:sec:imports}
\label{ch05:defn:standing}

We use the metric $\bar h$, its Chern connection $\bar\nabla$, and the projective bundle from Section~\ref{ch03:sec:fiber-integration}, with $D=d+r-1=N-1$ and $\Delta=\bM_g\setminus\cM_g$. The radial function is $a_{\rad}$, and $J$ is the coefficient norm in~\eqref{ch04:eq:cubic-coefficient-scales}; all plumbing lengths vary independently as in Subsection~\ref{ch04:defn:plumbing-setup}.

We estimate the transgression $T_s=-\dc_{\cM}V_s$ from Proposition~\ref{ch03:prop:transgression-potential}. For its interpolating form $W_{\tau,v}$, the connection identity gives
\begin{equation}\label{ch05:eq:interpolation}
 f_\tau=\tau\mathcal K_\tau,\qquad
 W_{\tau,v}=\varpi^*\omega_0+\tau\Omega_{\refn}
                         -v\dd\dPc f_\tau,
 \qquad 0\leq v\leq1.
\end{equation}
The background curvature form $\Omega_{\refn}$ is defined in Section~\ref{ch03:sec:fiber-integration}.

Let
$$
 H_s=\sum_{a=0}^d\left(\frac{2\pi}{s}\right)^{r+a}
 h_a(\mathbb E_3,\bar\nabla).
$$
Thus $\Theta_s=[e^{\omega_0}H_s]_{\leq2d}$ is the form in~\eqref{ch03:eq:theta-definition}, and $\Phi_s=\Theta_s+\dd_{\cM}T_s$ by Proposition~\ref{ch03:prop:relative-localization}. The defining integrals are smooth with all derivatives locally on $(0,\infty)\times\cM_g$, and the top-degree component of $\Phi_s$ and its $s$-derivatives are absolutely integrable by Proposition~\ref{ch04:prop:global-finiteness}. We estimate the boundary integral of $T_s$ and the correction associated with a smooth representative of $[\omega_0]$ on $\bM_g$.

For the scalar estimate, we work in a chart of Subsection~\ref{ch04:defn:plumbing-setup} with $k\geq1$ nodes and write $L_{\max}=\max_eL_e$. In the adapted frame of Proposition~\ref{ch04:prop:residue-frame}, with $I_q$ as in Section~\ref{ch04:sec:plumbing}, equations~\eqref{ch04:eq:cubic-hyperbolic-bounds} and~\eqref{ch04:eq:moment-comparison} give
\begin{equation}\label{ch05:eq:quadratic-comparison}
 \mu(X,q)\leq I_q
 \leq C\left(W(q)^2+\sum_eL_e^5|R_e(q)|^2\right).
\end{equation}
The constants are independent of the ratios $L_e/L_f$.

For each point $[\psi]$ in a projective fiber, choose a $\bar h$-unit representative $\psi$ and use the radius $a_{\rad}(\tau,X,\psi)$ of Corollary~\ref{ch01:cor:levels}. The value at $\tau=0$ is defined by the smooth extension of $\mathcal K_\tau$ in Proposition~\ref{ch03:prop:transgression-potential}.

\begin{prop}\label{ch05:prop:scalar-potential-bound}
Uniformly over the projective fibers and all independent plumbing lengths,
$$
 |\mathcal K_\tau|
 \leq C\bigl(1+\log L_{\max}+\log(1+\tau)\bigr),
 \qquad \tau\geq0.
$$
\end{prop}

\begin{proof}
The metric $\bar h$ is uniformly equivalent to the coefficient norm $W$ in the frame extending across the divisor. A $\bar h$-unit vector $\psi$ therefore satisfies
$$
 c\leq W(\psi)\leq C,\qquad
 J(\psi)\geq c,\qquad I_\psi\leq CL_{\max}^5.
$$
For $\tau>0$, put $a=a_{\rad}(\tau,X,\psi)$. Equation~\eqref{ch05:eq:quadratic-comparison} gives
$$
 \tau=\mu(X,a\psi)\leq a^2I_\psi\leq CL_{\max}^5a^2,
$$
and hence $a/\sqrt\tau\geq cL_{\max}^{-5/2}$.

For the upper bound, let $a_1=a_{\rad}(1,X,\psi)$. The lower bound in \eqref{ch04:eq:moment-comparison} and $J(\psi)\geq c$ imply $a_1\leq C$. If $0<\tau\leq1$, radial monotonicity gives $a\leq a_1$, and \eqref{ch03:eq:radial-dilation} yields
$$
 \tau=\mu(X,a\psi)\geq(a/a_1)^2.
$$
Thus $a/\sqrt\tau\leq a_1\leq C$. If $\tau\geq1$, the same moment comparison gives
$$
 a\leq\frac{C(A_0+\tau)^{3/2}}{J(\psi)},
 \qquad a/\sqrt\tau\leq C(1+\tau).
$$
We have proved
$$
 cL_{\max}^{-5/2}
 \leq\frac{a_{\rad}}{\sqrt\tau}
 \leq C(1+\tau).
$$
On the level set~$\mu^{-1}(\tau)$, \eqref{ch03:eq:radial-potential-identities} gives $1/2\leq F_{\rad}/\tau\leq3/2$. Substitution in \eqref{ch03:eq:legendre-definition} proves
$$
 -C-\log(1+\tau)\leq\mathcal K_\tau
 \leq C+\frac52\log L_{\max}.
$$
The estimate at $\tau=0$ follows from the smooth extension in Proposition~\ref{ch03:prop:transgression-potential}.
\end{proof}

\subsection{A smooth representative of $[\omega_0]$ on $\bM_g$}
\label{ch05:sec:reference}

We construct a smooth representative $\bar\Omega_0$ on $\bM_g$ and a primitive of $\bar\Omega_0|_{\cM_g}-\omega_0$ whose growth is controlled near every intersection of boundary divisors. Let
$$
 \mathcal L=
 \bigl(\det R\bar\pi_*\OO_{\bC_g}\bigr)^{-12}
 \otimes\OO_{\bM_g}(\Delta),
$$
where $\det R\bar\pi_*\OO=\det R^0\bar\pi_*\OO\otimes (\det R^1\bar\pi_*\OO)^{-1}$. By \cite[Corollary~1.16 and equations~(1.17)--(1.18)]{Finski2020}, applied to the unmarked family with $2g-2>0$, there are a smooth Hermitian metric $h_{\sm}$ on $\mathcal L$ and a global real one-form $\beta$ on $\cM_g$ such that
$$
 \omega_{\mathrm{K}}=-\pi^2c_1(\mathcal L,h_{\sm})+\dd\beta.
$$
Both $\beta$ and $\dd\beta$ have log-log growth along $\Delta$, and the construction is compatible with the orbifold charts.

Let $h_{\hyp}$ be the Hermitian metric on $\omega_{\bar\pi}$ induced by the fiberwise hyperbolic metrics over $\cM_g$. By \cite[Theorem~5.1 and equation~(5.6)]{Finski2020}, $\omega_{\mathrm{K}}=\pi^2\bar\pi_*c_1(\omega_{\bar\pi},h_{\hyp})^2$ on $\cM_g$. In the normalization of \cite[discussion after equation~(7) and Section~5]{Wolpert2007Cusps}, $\omega_{\WP}$ is twice this K\"ahler form. Thus
$$
 \omega_0=2\omega_{\WP}=4\omega_{\mathrm{K}}.
$$
The squared first Chern form is unchanged by constant metric rescaling or dualization, so the smooth representative and its primitive are
\begin{equation}\label{ch05:eq:wp-reference-data}
 \bar\Omega_0=-4\pi^2c_1(\mathcal L,h_{\sm}),\qquad
 \alpha_{\WP}=-4\beta,\qquad
 \Omega_0=\bar\Omega_0|_{\cM_g},
\end{equation}
and they satisfy $\dd\alpha_{\WP}=\Omega_0-\omega_0$. By \cite[Corollary~1.12, equation~(1.16)]{Finski2020}, $c_1(\mathcal L)=-\kappa_1$, so $[\bar\Omega_0]=4\pi^2\kappa_1$. These forms determine the correction to the relative localization identity.

\begin{defn}\label{ch05:defn:correction}
For these choices, the reference forms are
$$
 \bar H_s=\sum_{a=0}^d
 \left(\frac{2\pi}{s}\right)^{r+a}h_a(\bE_3,\bar\nabla),
 \qquad
 \bar\Theta_s=[e^{\bar\Omega_0}\bar H_s]_{\leq2d},
 \qquad
 \widetilde\Theta_s=\bar\Theta_s|_{\cM_g}.
$$
The closed even form $E_s$ is defined by
$$
 E_s=
 \left[
  \left(\int_0^1e^{(1-v)\omega_0+v\Omega_0}\,\dd v\right)H_s
 \right]_{\leq2d-2},
$$
and the correction and its components are
\begin{equation}\label{ch05:eq:corrected-transgression-definition}
 \begin{aligned}
 C_s&=[\alpha_{\WP}E_s]_{\mathrm{odd},\,\leq2d-1},
 &\qquad \widetilde T_s&=T_s-C_s,\\
 A_s&=[\widetilde T_s]_{2d-1},
 & F_s&=[\Phi_s-\widetilde\Theta_s]_{2d}.
 \end{aligned}
\end{equation}
\end{defn}

The form $\bar\Theta_s$ is smooth, real, and closed by Chern--Weil theory. Since $[\bar\Omega_0]=4\pi^2\kappa_1$, its integral is
$$
 \int_{\bM_g}[\bar\Theta_s]_{2d}=P_g(s),
$$
where $P_g(s)$ is defined in \eqref{ch02:eq:compactified-fixed-degree}. We next bound the growth of $C_s$ near $\Delta$ and describe its extension as a current.

\begin{prop}\label{ch05:prop:reference-control}
For every $s>0$,
\begin{equation}\label{ch05:eq:corrected-localization}
 \dd C_s=\widetilde\Theta_s-\Theta_s,\qquad
 \Phi_s=\widetilde\Theta_s+\dd\widetilde T_s,\qquad
 \dd A_s=F_s.
\end{equation}
On every relatively compact plumbing chart there is an integer $P\geq0$ such that the coefficients of
$$
 \alpha_{\WP},\quad \omega_0,\quad
 \partial_s^jE_s,\quad \partial_s^jC_s,\quad
 \dd\partial_s^jC_s
$$
are bounded by
\begin{equation}\label{ch05:eq:reference-coefficients}
 C_j\prod_{e=1}^k(1+\log L_e)^P L_e^{-n_e},
\end{equation}
where $n_e$ counts the occurrences of $\dd L_e$ and $\dd\phi_e$ in the corresponding differential monomial. This holds for every $j\geq0$, uniformly for $s$ in compact subsets of $(0,\infty)$ and with all lengths varying independently.

Each component of $C_s$ and $\dd C_s$ is locally integrable in smooth coordinates on the compactification. For such a form $\eta$, let $\mathcal T(\eta)$ denote its trivial extension as a current on $\bM_g$, defined by integration of $\eta\wedge\theta$ against smooth compactly supported test forms $\theta$ of complementary degree. In each degree,
\begin{equation}\label{ch05:eq:reference-current}
 \dd\mathcal T(C_s)=\mathcal T(\dd C_s),
\end{equation}
and the same identity holds after every $s$-derivative. Moreover,
$$
 \int_{\cM_g}\dd\partial_s^j[C_s]_{2d-1}=0,\qquad
 \int_{\cM_g}\partial_s^j[\Theta_s]_{2d}
 =\partial_s^jP_g(s).
$$
All these integrals are absolutely convergent.
\end{prop}

\begin{proof}
The correction identity follows by differentiating the interpolating exponential. Since $\dd\alpha_{\WP}=\Omega_0-\omega_0$, the closed two-forms in the interpolation commute and
$$
 \frac{\partial}{\partial v}
 e^{\omega_0+v\dd\alpha_{\WP}}
 =\dd\alpha_{\WP}\,e^{\omega_0+v\dd\alpha_{\WP}}.
$$
Therefore
$$
 \dd\left(
 \alpha_{\WP}\int_0^1e^{\omega_0+v\dd\alpha_{\WP}}\,\dd v
 \right)
 =e^{\Omega_0}-e^{\omega_0}.
$$
Multiplication by the closed even form $H_s$, followed by degree truncation, proves the first identity in \eqref{ch05:eq:corrected-localization}. The other two follow from \eqref{ch03:eq:relative-localization}. All forms remain jointly smooth on $(0,\infty)\times\cM_g$.

We next express the log-log growth condition in plumbing coordinates. In an adapted chart, log-log growth is measured using the normal differentials
$$
 \frac{\dd t_e}{t_e\log|t_e|}
 =\frac{\dd L_e-\ii\,\dd\phi_e}{L_e}
$$
and their complex conjugates, together with smooth tangential differentials; see \cite[Definitions~2.10, 2.12 and Notation~2.11]{Finski2020}. The coefficients of a form of log-log growth in this frame are bounded by a finite power of $\prod_e(1+\log L_e)$. Thus $\beta$ and $\dd\beta$ satisfy \eqref{ch05:eq:reference-coefficients}. A form smooth on the compactification has a factor $e^{-L_e}$ for each normal coordinate differential in these coordinates, because $\dd t_e=t_e(-\dd L_e+\ii\,\dd\phi_e)$. Since $\alpha_{\WP}=-4\beta$, this proves its coefficient bound, and $\omega_0=\Omega_0-\dd\alpha_{\WP}$ satisfies it as well.

In a nonzero wedge monomial, the numbers of occurrences of normal coordinate differentials add. The finite sums of wedge products defining $E_s$ and $C_s$ therefore satisfy \eqref{ch05:eq:reference-coefficients}, after increasing $P$. Differentiation with respect to $s$ acts on the coefficients $(2\pi/s)^{r+a}$, so the exponent $P$ is independent of $j$. The identity
$$
 \dd E_s=0,\qquad
 \dd C_s=[(\Omega_0-\omega_0)E_s]_{\leq2d}
$$
gives the coefficient bound for $\dd C_s$.

We now integrate these bounds with independent truncation parameters. Fix lower bounds $L_{0,e}>1$, and write
$$
 Q_e(A)=\frac{(1+\log A)^P}{A},\qquad
 I_e=\int_{L_{0,e}}^\infty
       \frac{(1+\log u)^P}{u^2}\,\dd u.
$$
Each $I_e$ is finite. Indeed, the substitution $u=Ae^w$ gives
\begin{equation}\label{ch05:eq:logarithmic-tail}
 \int_A^\infty\frac{(1+\log u)^P}{u^2}\,\dd u
 =\frac1A\int_0^\infty(1+\log A+w)^Pe^{-w}\,\dd w
 \leq C_P Q_e(A).
\end{equation}
Let
$$
 R(\mathbf A)=
 B\times\prod_e[L_{0,e},A_e]\times(S^1)^k
$$
be the truncated plumbing chart, where $B$ is a relatively compact coordinate domain in the stratum. Let $\chi$ be smooth on the compactification, with compact support in the chart away from its outer coordinate boundaries. On the boundary face $\{L_e=A_e\}$, every nonzero term in the restriction of a $(2d-1)$-form contains $\dd\phi_e$, both normal coordinate differentials for every $f\neq e$, and every tangential covector. Consequently
\begin{equation}\label{ch05:eq:reference-face}
 \int_{\{L_e=A_e\}}
 \left|\chi\,\partial_s^j[C_s]_{2d-1}\right|
 \leq C_j\|\chi\|_{C^0}Q_e(A_e)\prod_{f\neq e}I_f.
\end{equation}
The sum of these face integrals tends to zero whenever every $A_e$ tends to infinity.

A nonzero differential monomial of top degree contains both normal coordinate differentials for every neck. The resulting absolute density is bounded by
$$
 C_j\prod_e(1+\log L_e)^P L_e^{-2}
 \left|\dd b\,\dd L_1\,\dd\phi_1\cdots
       \dd L_k\,\dd\phi_k\right|.
$$
This applies to $\chi\,\dd\partial_s^j[C_s]_{2d-1}$. It also applies to $\dd\chi\wedge\partial_s^j[C_s]_{2d-1}$, because a smooth complementary differential supplies the missing normal factor with an $e^{-L_e}$ coefficient. Fubini's theorem and \eqref{ch05:eq:logarithmic-tail} bound the absolute tails on the union $L_e\geq A_e$ by
$$
 C_j\|\chi\|_{C^1}
 \sum_e Q_e(A_e)\prod_{f\neq e}I_f.
$$
For the term involving $\dd C_s$, the $C^0$ norm suffices. The same bounds hold after wedging a component of $C_s$ with a smooth form of complementary degree.

The coefficient bounds also justify the current identity. For each degree $a$, let $\theta$ be a smooth form of degree $2d-a-1$ on the compactification chart, compactly supported in its interior. We choose the outer truncation so that $\theta$ vanishes near $\partial B$ and every face $L_e=L_{0,e}$. The form $[C_s]_a\wedge\theta$ satisfies the face bound \eqref{ch05:eq:reference-face}. Pairings with smooth complementary forms also show that the coefficients of $[C_s]_a$ and $\dd [C_s]_a$ are locally integrable in smooth coordinates: each missing normal differential contributes a factor $e^{-L_e}$. Applying Stokes' theorem on the compact domain with corners $R(\mathbf A)$, with the boundary orientation obtained by placing the outward normal first, we obtain
$$
 \int_{\partial R(\mathbf A)}[C_s]_a\wedge\theta
 =
 \int_{R(\mathbf A)}
 \left(\dd [C_s]_a\wedge\theta+
       (-1)^a[C_s]_a\wedge\dd\theta\right).
$$
The sum of the integrals over the boundary faces~$\{L_e=A_e\}$ tends to zero, and both volume terms converge absolutely, giving~\eqref{ch05:eq:reference-current}.

On overlaps, the boundary branches are permuted and $t'_e=u_e t_{\sigma(e)}$, where $u_e$ is a nowhere-vanishing holomorphic function. Hence $L'_e=L_{\sigma(e)}-\log|u_e|$, with a bounded smooth remainder on smaller charts. The forms $\dd t_e/(t_e\log|t_e|)$ transform into linear combinations with bounded coefficients, plus smooth differentials divided by $L'_e$. The bounds therefore persist under adapted coordinate changes. Finite invariant partitions of unity glue the current identity on $\bM_g$, with the quotient integration of Lemma~\ref{ch04:lem:descent}.

Pairing~\eqref{ch05:eq:reference-current} in top degree with the constant function on the compact orbifold gives $\int_{\cM_g}\dd\partial_s^j[C_s]_{2d-1}=0$. Since $\Theta_s=\widetilde\Theta_s-\dd C_s$, its top component is integrable, and the divisor has zero measure for the smooth density $[\bar\Theta_s]_{2d}$, so integration proves the last assertion. The majorants are uniform for $s$ in compact subsets of $(0,\infty)$ and therefore justify differentiation.
\end{proof}

For an orbifold domain $B\Subset\cM_g$ with compact closure and smooth boundary, the corrected localization identity gives
\begin{equation}\label{ch05:eq:compact-stokes}
 \int_{p^{-1}(B)}
       e^{-s\mu}\frac{\omega_G^N}{N!}
 =
 \int_B[\widetilde\Theta_s]_{2d}
 +\int_{\partial B}A_s.
\end{equation}
Here $\partial B$ has the boundary orientation obtained by placing the outward normal first. The absolute fiber bounds in Propositions~\ref{ch03:prop:analytic-estimates} and~\ref{ch03:prop:relative-localization} justify Fubini's theorem and all $s$-derivatives, locally uniformly for $s>0$. The factor $(-1)^r$ in $\Phi_s=(-1)^rp_*^\C e^{\omega_G-s\mu}$ converts the product complex orientation to the Goldman orientation, so the left side is $\int_B[\Phi_s]_{2d}$. The formula follows by applying Stokes' theorem to \eqref{ch05:eq:corrected-localization}.

\subsection{Mass and energy estimates}
\label{ch05:sec:mass}

We control the wedge powers of $W_{\tau,v}$ by estimates for semipositive forms relative to a smooth reference form. The proof uses the integration by parts argument for the Chern--Levine--Nirenberg inequalities and the estimates for products involving first derivatives in \cite[Section~1, (1.1), (1.3), and Remark~(1.5)]{DemaillyCLN}.

Choose bounded domains $U\Subset U'\subset\C^d$ and $D+2$ nested intermediate domains. On $U'\times\mathbb P^{r-1}$, let $G=G_b+G_f$, where $G_b$ is the standard Euclidean K\"ahler form and $G_f$ is a Fubini--Study form. The constants below depend on these domains and background forms.

\begin{lem}\label{ch05:lem:mass-energy}
The following estimates hold for smooth closed semipositive $(1,1)$-forms $P,P_0$ and a smooth real function $f$ satisfying
$$
 P=P_0+\dd\dc f,\qquad
 0\leq P_0\leq CaG,\qquad
 a+\|f\|_{L^\infty(U'\times\mathbb P^{r-1})}\leq B,\qquad B\geq1.
$$
For $0\leq m\leq D$,
\begin{equation}\label{ch05:eq:mass}
 \int_{U\times\mathbb P^{r-1}}P^m\wedge G^{D-m}\leq C_mB^m,
\end{equation}
and
\begin{equation}\label{ch05:eq:energy}
 \int_{U\times\mathbb P^{r-1}}
       \dd f\wedge\dc f\wedge P^{D-1}\leq CB^{D+1}.
\end{equation}
If $Q$ is a product of $D-1$ semipositive $(1,1)$-forms, each bounded above by $P$, then every real one-form $\lambda$ on $U\times\mathbb P^{r-1}$ satisfies
\begin{equation}\label{ch05:eq:gradient-pairing}
 \int_{U\times\mathbb P^{r-1}}
       |\lambda\wedge\dc f\wedge Q|
 \leq C\|\lambda\|_{L^\infty(G)}B^D.
\end{equation}
\end{lem}

\begin{proof}
Choose nonnegative base cutoffs equal to one on each smaller domain and supported in the next, with derivatives bounded relative to $G$. The projective factor is compact. The integration by parts argument for the Chern--Levine--Nirenberg estimate \cite[Section~1, (1.1) and~(1.3)]{DemaillyCLN} gives
$$
 \int hP^m\wedge G^{D-m}
 =\int hP_0\wedge P^{m-1}\wedge G^{D-m}
   +\int f\,\dd\dc h\wedge P^{m-1}\wedge G^{D-m}
 \leq CB\int_{\text{next domain}}P^{m-1}\wedge G^{D-m+1}.
$$
Induction proves \eqref{ch05:eq:mass}. The identity
$$
 \dd\dc(f^2)=2\dd f\wedge\dc f+2f(P-P_0)
$$
then yields
$$
 \int h\,\dd f\wedge\dc f\wedge P^{D-1}
 =\frac12\int f^2\dd\dc h\wedge P^{D-1}
   -\int hf(P-P_0)\wedge P^{D-1}
 \leq CB^{D+1},
$$
using the mass estimates on the larger domains.

Since each factor of $Q$ is bounded by $P$, the difference $P^{D-1}-Q$ is a sum of products of semipositive forms. Thus the energy estimate also holds with $Q$. Put $\lambda^c=-\lambda\circ I$. By the Hermitian Cauchy--Schwarz inequality associated with the positive form $Q$, followed by scalar Cauchy--Schwarz,
$$
 \begin{aligned}
 \int|\lambda\wedge\dc f\wedge Q|
 &\leq C
 \left(\int\lambda\wedge\lambda^c\wedge Q\right)^{1/2}
 \left(\int\dd f\wedge\dc f\wedge Q\right)^{1/2}\\
 &\leq C\|\lambda\|_{L^\infty(G)}
          (B^{D-1})^{1/2}(B^{D+1})^{1/2}.
 \end{aligned}
$$
Here $\lambda\wedge\lambda^c\leq C\|\lambda\|_{L^\infty(G)}^2G$. This proves \eqref{ch05:eq:gradient-pairing}; compare the mixed products in \cite[Remark~(1.5)]{DemaillyCLN}.
\end{proof}

\subsection{Estimates on products of dyadic annuli}
\label{ch05:sec:annuli}

We apply Lemma~\ref{ch05:lem:mass-energy} in logarithmic coordinates on uniformly bounded domains. The angular covering multiplicities supply one inverse length for every node.

In a chart of Subsection~\ref{ch04:defn:plumbing-setup}, choose a compact stratum subset $K$ and coordinate domains $K\Subset B\Subset B'$, with $B'$ a polydisc. The plumbing chart lies in a larger chart where the estimates of Propositions~\ref{ch04:prop:horizontal-metric} and~\ref{ch05:prop:scalar-potential-bound} hold; we take $L_*$ sufficiently large that both estimates apply whenever every $L_e\geq L_*/2$, and use this threshold in the annular and cutoff arguments below. We use the frame of Proposition~\ref{ch04:prop:residue-frame}, write $b$ for the stratum coordinate and $\dd_B$ for differentiation in $b$, and for $\boldsymbol\ell=(\ell_1,\ldots,\ell_k)$ with $\ell_e\geq L_*$ set
$$
 S(\boldsymbol\ell)
 =K\times\prod_e[\ell_e,2\ell_e]\times(S^1)^k,\qquad
 H(\boldsymbol\ell)=1+\sum_e\log\ell_e.
$$
The exponential parametrization is
\begin{equation}\label{ch05:eq:logarithmic-cover}
 \mathcal F_{\boldsymbol\ell}(b,z,[q])
 =\bigl(b,t_e=e^{-\ell_ez_e},[q]\bigr)
\end{equation}
on
$$
 B'\times
 \prod_e\left\{\frac12<\Re z_e<3,\ |\Im z_e|<2\right\}
 \times\mathbb P^{r-1}.
$$
Its inner domain has $b\in K$, $1\leq\Re z_e\leq2$, and $|\Im z_e|\leq1$. For a real one-form $\lambda$ on $S(\boldsymbol\ell)$, let $K_\lambda$ be the supremum norm of $\mathcal F_{\boldsymbol\ell}^*\lambda$ on this inner domain, in the product Euclidean metric. For a smooth function $\chi=\chi(b)$, the test forms $\lambda=\dd_B\chi$ and $\lambda=\chi\,\dd L_e/\ell_e$ pull back to $\dd_B\chi$ and $\chi\,\dd\Re z_e$, giving the bounds $K_\lambda\leq C\|\chi\|_{C^1}$ and $K_\lambda\leq C\|\chi\|_{C^0}$, respectively.

Fix $j\geq0$ and $0<s_0\leq s_1<\infty$. The estimates below are uniform for $s\in[s_0,s_1]$, $M\geq0$, and all independent lengths $\ell_e\geq L_*$. Define the nonnegative integral
$$
\begin{aligned}
 \mathcal I_{j,s}(\lambda;\boldsymbol\ell,M)
 ={}&2\pi\int_0^\infty\tau^je^{-s\tau}\int_0^1
 \int_{\substack{\varpi^{-1}S(\boldsymbol\ell)\\
                 J(a_{\rad}\psi)\geq M}}
 \left|
  \varpi^*\lambda\wedge\dPc f_\tau
  \wedge\frac{W_{\tau,v}^{D-1}}{(D-1)!}
 \right|\,\dd v\,\dd\tau.
\end{aligned}
$$
The representative $\psi$ has $\bar h$-norm one. Multiplication of $\psi$ by a unit complex number preserves both the integration domain and the integrand.

For the estimate on level sets of $\mu$, we replace the restriction $J(a_{\rad}\psi)\geq M$ by $\tau\geq\Lambda$, where $\Lambda\geq1$. The next estimate combines the potential bound with the angular covering multiplicities.

\begin{prop}\label{ch05:prop:annular-estimate}
For the fixed interval $[s_0,s_1]\subset(0,\infty)$, there are constants $C_j,c>0$, depending on this interval and the compact chart data, such that
\begin{equation}\label{ch05:eq:annular-estimate}
 \mathcal I_{j,s}(\lambda;\boldsymbol\ell,M)
 \leq
 C_jK_\lambda H(\boldsymbol\ell)^D
 \prod_e\ell_e^{-1}e^{-cM^{2/3}}.
\end{equation}
For the moment restriction, the factor $e^{-cM^{2/3}}$ is replaced by $e^{-s_0\Lambda/2}$.
\end{prop}

\begin{proof}
By Proposition~\ref{ch05:prop:scalar-potential-bound}, on each projective fiber and for every $\tau\geq0$,
\begin{equation}\label{ch05:eq:joint-scalar}
 |\mathcal K_\tau|
 \leq C\bigl(1+\log L_{\max}+\log(1+\tau)\bigr).
\end{equation}
The bound is uniform as the plumbing lengths vary independently.

We next construct semipositive forms on the domain of the map \eqref{ch05:eq:logarithmic-cover}. Let $G_b$ be the standard Euclidean K\"ahler form on the base parameter domain, let $G_f$ be the Fubini--Study form of the holomorphic trivialization, and put $G=G_b+G_f$. The complex dimension of the total projective domain is $D$.

The form $\omega_G$ has type $(1,1)$ by Proposition~\ref{ch03:prop:transgression-potential}, and its vertical Hermitian form is negative definite by \eqref{ch04:eq:vertical-taming}.

For $\tau>0$, the radial and circle directions span a negative complex line. Its orthogonal complement in the vertical tangent space is complex and tangent to the level set~$\mu^{-1}(\tau)$, because $IR_{\rad}=-V_{\rot}$ and $\dd\mu=-\iota_{V_{\rot}}\omega_G$. The map $q\mapsto[q]$ sends this complement complex linearly to the projective tangent space. Thus the reduced vertical form is negative definite. An $\omega_G$-horizontal lift is also tangent to this level set. Its image under reduction is orthogonal to the projective fibers and retains the horizontal metric. Hence the vertical block of $\omega_\tau$ is negative definite, and the Schur complement of this block is positive and bounded above by $C(A_0+\tau)G_b$. The estimate \eqref{ch04:eq:horizontal-quadratic} follows from Proposition~\ref{ch04:prop:horizontal-metric} and pulls back uniformly because
$$
 \mathcal F_{\boldsymbol\ell}^*
 \left(\frac{(\dd L_e)^2+(\dd\phi_e)^2}{L_e^2}\right)
 =\frac{(\dd\Re z_e)^2+(\dd\Im z_e)^2}{(\Re z_e)^2}.
$$

Set
$$
 \omega_{\refn}=\varpi^*\omega_0+\tau\Omega_{\refn}.
$$
In the horizontal splitting determined by the Chern connection, its vertical block is a negative multiple $-\tau$ of a uniformly positive Fubini--Study form. In the holomorphic trivialization its mixed block is $O(\tau)$, and its horizontal block is $O(A_0+\tau)$. The coefficients of forms smooth on the compactification are uniformly bounded after pullback: differentiation in a normal variable gives a factor $\ell_e e^{-\ell_e\Re z_e}$, which is bounded on the larger domain. The case $q=0$ of \eqref{ch04:eq:horizontal-quadratic} controls $\omega_0$. Thus the Schur complement of the vertical block of $\omega_{\refn}$ is also bounded above by $C(A_0+\tau)G_b$.

The Schur complement bound implies semipositivity after subtracting either $\omega_\tau$ or $\omega_{\refn}$ from a suitable base form. For $\tau>0$, the Hermitian matrix associated with the pullback by $\mathcal F_{\boldsymbol\ell}$ of either real $(1,1)$-form, in a splitting with negative vertical block, is
$$
 \begin{pmatrix} B&C\\ C^*&-K\end{pmatrix},
 \qquad K>0.
$$
The Schur complement of $-K$ is $B+CK^{-1}C^*$. For $a_*=C_*(A_0+\tau)$, with sufficiently large $C_*$, the Hermitian quadratic form associated with $a_*G_b-\mathcal F_{\boldsymbol\ell}^*\omega$, evaluated on a horizontal vector $u$ and a vertical vector $w$, is
\begin{equation}\label{ch05:eq:schur-square}
 \bigl\langle K(w-K^{-1}C^*u),w-K^{-1}C^*u\bigr\rangle
 +\bigl\langle
 (a_*G_b-B-CK^{-1}C^*)u,u
 \bigr\rangle.
\end{equation}
Both terms are nonnegative. Consequently
$$
 \mathcal P_\tau=a_*G_b-\mathcal F_{\boldsymbol\ell}^*\omega_\tau,\qquad
 \mathcal P_{\refn}=a_*G_b-\mathcal F_{\boldsymbol\ell}^*\omega_{\refn}
$$
are closed semipositive $(1,1)$-forms. Both forms extend smoothly to $\tau=0$ by Proposition~\ref{ch03:prop:transgression-potential}; their semipositivity at $\tau=0$ follows by continuity.

The potential identity in \eqref{ch05:eq:interpolation} gives
$$
 \mathcal P_\tau=\mathcal P_{\refn}
       +\dd\dc(\mathcal F_{\boldsymbol\ell}^*f_\tau).
$$
Define
$$
 P=\mathcal P_\tau+\mathcal P_{\refn}+a_*G_b,\qquad
 P_0=2\mathcal P_{\refn}+a_*G_b.
$$
Then
$$
 P=P_0+\dd\dc(\mathcal F_{\boldsymbol\ell}^*f_\tau),\qquad
 0\leq P_0\leq Ca_*G.
$$
On the larger domain, \eqref{ch05:eq:joint-scalar} implies
$$
 a_*+\|\mathcal F_{\boldsymbol\ell}^*f_\tau\|_\infty\leq B_*,
 \qquad
 B_*=C(A_0+\tau)
 \left(
  H(\boldsymbol\ell)+\log(1+\tau)
 \right),
$$
after enlarging $C$ so that $B_*\geq1$.

The forms $P,P_0$ and the function $\mathcal F_{\boldsymbol\ell}^*f_\tau$ satisfy the hypotheses of Lemma~\ref{ch05:lem:mass-energy}. For $\mathcal P_v=(1-v)\mathcal P_{\refn}+v\mathcal P_\tau$, we have
\begin{equation}\label{ch05:eq:semipositive-domination}
 \mathcal F_{\boldsymbol\ell}^*W_{\tau,v}=a_*G_b-\mathcal P_v,\qquad
 0\leq a_*G_b\leq P,\qquad 0\leq\mathcal P_v\leq P.
\end{equation}
Expansion of $\mathcal F_{\boldsymbol\ell}^*(W_{\tau,v}^{D-1})$ expresses it as a finite sum of products of these semipositive forms, with signs. Applying \eqref{ch05:eq:gradient-pairing} to each product gives
\begin{equation}\label{ch05:eq:fixed-cover-estimate}
 \int_{\mathrm{inner}}
 \left|
  \mathcal F_{\boldsymbol\ell}^*
  \left(\varpi^*\lambda\wedge\dPc f_\tau
  \wedge\frac{W_{\tau,v}^{D-1}}{(D-1)!}\right)
 \right|
 \leq CK_\lambda B_*^D.
\end{equation}
The estimate is uniform in $v\in[0,1]$.

We now pass from the parameter domain to the twist torus. The pulled-back forms and the domains defined by $J(a_{\rad}\psi)\geq M$ are periodic in $\Im z_e$ with period $2\pi/\ell_e$, because translation by this period preserves $t_e$ and the projective point in the holomorphic frame. Once $L_*$ is large, $[-1,1]$ contains at least $c\ell_e$ complete periods, so the integral of a nonnegative periodic density over this interval is at least $c\ell_e$ times its integral over one period. Successive application in all $k$ angular variables gives at least $c\prod_e\ell_e$ fundamental domains, each mapping once to the twist torus. Thus~\eqref{ch05:eq:fixed-cover-estimate} yields
\begin{equation}\label{ch05:eq:shell-fixed-moment}
 \int_{\varpi^{-1}S(\boldsymbol\ell)}
 \left|
   \varpi^*\lambda\wedge\dPc f_\tau
  \wedge\frac{W_{\tau,v}^{D-1}}{(D-1)!}
 \right|
 \leq CK_\lambda B_*^D\prod_e\ell_e^{-1}.
\end{equation}

Equation~\eqref{ch03:eq:exact-potential-primitive} gives the component of $T_s$ controlled by this integral. Every $W_{\tau,v}$ is closed and of type $(1,1)$, and is therefore $\dc$-closed. Projective integration over the compact complex fibers commutes with $\dc$. The component of $T_s$ of degree $2d-1$ is therefore
\begin{equation}\label{ch05:eq:top-primitive}
 [T_s]_{2d-1}
 =
 -2\pi(-1)^{r-1}\int_0^\infty e^{-s\tau}
 \varpi_*\left(
  \int_0^1\dPc f_\tau
  \wedge\frac{W_{\tau,v}^{D-1}}{(D-1)!}\,\dd v
 \right)\dd\tau.
\end{equation}
The integrand has degree $2D-1$, and integration over the projective fiber lowers the degree by $2r-2$, giving $2d-1$. The $j$-th $s$-derivative inserts $(-\tau)^j$.

The resulting Laplace integral satisfies
\begin{equation}\label{ch05:eq:laplace-majorant}
 \int_0^\infty\tau^je^{-s_0\tau}B_*^D\,\dd\tau
 \leq C_jH(\boldsymbol\ell)^D.
\end{equation}
The factor $B_*^D$ is bounded as $\tau\downarrow0$. As $\tau\to\infty$, the exponential dominates the polynomial in $A_0+\tau$ and the powers of $\log(1+\tau)$.

By Tonelli's theorem and \eqref{ch05:eq:shell-fixed-moment}--\eqref{ch05:eq:laplace-majorant}, \eqref{ch05:eq:annular-estimate} holds when $M=0$. These estimates also bound the absolute value of the projective pushforward and the Laplace integral in \eqref{ch05:eq:top-primitive}. For $\tau\geq\Lambda$,
$$
 e^{-s\tau}\leq e^{-s_0\Lambda/2}e^{-s_0\tau/2}.
$$
The majorant with $s_0/2$ proves the moment tail. Let $c_{\mu}>0$ be the lower-bound constant in \eqref{ch04:eq:moment-comparison}, which is independent of $s_0$. On the set $J(a_{\rad}\psi)\geq M$, we have $\tau\geq c_{\mu}M^{2/3}-A_0$, and hence
$$
 e^{-s\tau}
 \leq
 e^{s_0A_0/2-s_0c_{\mu}M^{2/3}/2}e^{-s_0\tau/2}.
$$
Removing the restricted set from the remaining nonnegative integral proves the cubic tail, with $c=s_0c_{\mu}/2$ in \eqref{ch05:eq:annular-estimate}.
\end{proof}

We apply this bound to $\lambda\wedge\partial_s^j[T_s]_{2d-1}$ and sum over dyadic annuli to obtain the cutoff estimate near intersections of boundary divisors.

\begin{cor}\label{ch05:cor:annular-consequences}
The right side of \eqref{ch05:eq:annular-estimate} with $M=0$ bounds
$$
 \int_{S(\boldsymbol\ell)}
 \left|\lambda\wedge\partial_s^j[T_s]_{2d-1}\right|.
$$
If $\chi$ is smooth on the compactification and has compact support in the chart, then
$$
 \left|\dd\chi\wedge\partial_s^j[T_s]_{2d-1}\right|
$$
is integrable near the divisor intersection. For independent thresholds $A_e\geq L_*$, its tail on the union of the regions $L_e\geq A_e$ is at most
\begin{equation}\label{ch05:eq:side-tail}
 C_j\|\chi\|_{C^1}
 \sum_e\frac{(1+\log A_e)^D}{A_e}.
\end{equation}
This bound also holds before projective and Laplace integration, with the tail factors for $J$ and $\mu$ in Proposition~\ref{ch05:prop:annular-estimate}.
\end{cor}

\begin{proof}
The bound for $\lambda\wedge\partial_s^j[T_s]_{2d-1}$ follows from \eqref{ch05:eq:top-primitive} and Proposition~\ref{ch05:prop:annular-estimate}. For $\chi$ smooth on the compactification, the pullback of $\dd\chi$ also has uniformly bounded norm, because $\dd t_e=-\ell_e e^{-\ell_ez_e}\dd z_e$. It remains to sum the shell estimates.

Decompose the corner into boxes with $\ell_e=L_*2^{n_e}$, $n_e\in\Z_{\geq0}$. Up to a constant independent of the indices, the estimate on each box is
$$
 \left(1+\sum_en_e\right)^D\prod_e2^{-n_e}.
$$
Its sum is finite: expansion of the integer power reduces it to products of convergent series $\sum_{n\geq0}(1+n)^D2^{-n}$. Every box intersecting $L_e\geq A_e$ has $n_e\geq\log_2(A_e/L_*)-1$. Summing first over the other indices and then over this tail gives
$$
 C\frac{(1+\log A_e)^D}{A_e}.
$$
The union bound proves \eqref{ch05:eq:side-tail} with independent $A_e$. The summation preserves each exponential tail factor.
\end{proof}

\subsection{The global divisor cutoff}
\label{ch05:sec:global-cutoff}

We sum the annular estimates over a global divisor cutoff to prove the boundary limit for every smooth exhaustion. We also prove a localized limit at a smooth boundary divisor in Appendix~\ref{ch05:sec:local-flux}, including the term involving $\dd_B\chi$ in Stokes' formula.

Choose a smooth Hermitian metric on $\OO(\Delta)$, and let $\sigma_\Delta$ be its canonical section. Choose $c_0$ so that
$$
 H_\Delta=c_0-\log\|\sigma_\Delta\|^2\geq1
 \quad\text{on }\cM_g.
$$
The function $H_\Delta$ is smooth and proper, since the section norm is positive on $\cM_g$, bounded above on $\bM_g$, and tends to zero at $\Delta$. In an adapted chart,
\begin{equation}\label{ch05:eq:divisor-exhaustion}
 H_\Delta=2\sum_eL_e+h(b,t),
\end{equation}
where $h$ is smooth with bounded derivatives on a slightly larger compactification chart.

Fix a smooth function $\eta\colon[0,\infty)\to[0,1]$ equal to one on $[0,1]$ and zero on $[2,\infty)$, and put $\eta_R=\eta(H_\Delta/R)$.

We choose $P$ for a finite compactification atlas as in Proposition~\ref{ch05:prop:reference-control}. Fix $j\geq0$ and $0<s_0\leq s_1<\infty$. The estimates below are uniform for $s\in[s_0,s_1]$. The restrictions on level sets of $\mu$ and on cubic differentials in each chart are those used in Proposition~\ref{ch05:prop:annular-estimate}.

\begin{prop}\label{ch05:prop:divisor-cutoff}
For sufficiently large $R$,
\begin{align}
 \int_{\cM_g}
 \left|\dd\eta_R\wedge\partial_s^j[T_s]_{2d-1}\right|
 &\leq C_j\frac{(1+\log R)^D}{R},
 \label{ch05:eq:global-transgression-cutoff}\\
 \int_{\cM_g}
 \left|\dd\eta_R\wedge\partial_s^j[C_s]_{2d-1}\right|
 &\leq C_j\frac{(1+\log R)^{P+1}}{R}.
 \label{ch05:eq:global-reference-cutoff}
\end{align}
The first estimate also holds for the absolute defining integral before projective pushforward and Laplace integration, with an additional factor $e^{-s_0\Lambda/2}$ for $\tau\geq\Lambda\geq1$, and with an additional $e^{-cM^{2/3}}$ for the restrictions on cubic differentials in each chart.
\end{prop}

\begin{proof}
Choose a finite nonnegative smooth orbifold partition of unity subordinate to boundary compactification charts and finitely many relatively compact interior orbifold charts. For large $R$, $\dd\eta_R$ vanishes on all the interior charts. It is enough to estimate its product with a function in the partition on each boundary chart.

We first prove \eqref{ch05:eq:global-transgression-cutoff}. We decompose each chart into the dyadic boxes with $\ell_e=L_*2^{n_e}$, $n_e\in\Z_{\geq0}$. A box intersecting $\supp\dd\eta_R$ contains a point with $R\leq H_\Delta\leq2R$. The expression \eqref{ch05:eq:divisor-exhaustion} and boundedness of $h$ then imply
$$
 cR\leq\sum_e\ell_e\leq CR,\qquad
 \max_e\ell_e\asymp R.
$$
On the inner parameter domain,
$$
 \mathcal F_{\boldsymbol\ell}^*\dd\eta_R
 =
 \frac{\eta'(H_\Delta/R)}{R}
 \left(
  2\sum_e\ell_e\,\dd\Re z_e+
  \mathcal F_{\boldsymbol\ell}^*\dd h
 \right).
$$
The smooth term is uniformly bounded, and each contributing $\ell_e$ is at most $CR$. Therefore the pulled-back norm is uniformly bounded. Multiplication by a function in the partition of unity preserves this bound.

Let $n_{\max}=\max_en_e$. On a contributing box, $2^{n_{\max}}\asymp R$. The comparability restricts $n_{\max}$ to a uniformly bounded number of integers. We assign each box to an index $e$ at which $n_e=n_{\max}$. Consequently,
$$
\begin{aligned}
 &\sum_{\text{contributing boxes}}
 \left(1+\sum_fn_f\right)^D\prod_f2^{-n_f}\\
 &\qquad\leq
 C(1+\log R)^D
 \sum_e\sum_{\substack{n_e:\ 2^{n_e}\asymp R\\
                      0\leq n_f\leq n_e,\ f\neq e}}
 2^{-n_e}\prod_{f\neq e}2^{-n_f}
 \leq C\frac{(1+\log R)^D}{R}.
\end{aligned}
$$
For the last inequality, the factor corresponding to the largest dyadic index is $O(R^{-1})$ and every remaining geometric series has bounded sum. The estimate \eqref{ch05:eq:global-transgression-cutoff} now follows from Corollary~\ref{ch05:cor:annular-consequences}. The absolute majorant in Proposition~\ref{ch05:prop:annular-estimate} proves the same assertion before projective and Laplace integration. The moment tail is preserved by the sum. For the restriction on cubic differentials, we take the minimum of the finitely many positive coercivity constants associated with the charts.

We prove \eqref{ch05:eq:global-reference-cutoff} directly from \eqref{ch05:eq:reference-coefficients}. On the support of $\dd\eta_R$, every $L_e\leq CR$. The $\dd h/R$ contribution in \eqref{ch05:eq:divisor-exhaustion} is $O(R^{-1})$, since its wedge with the reference form is integrable by Proposition~\ref{ch05:prop:reference-control}. Only the terms of $[C_s]_{2d-1}$ that omit $\dd L_e$ contribute to the wedge product with $2\dd L_e/R$. Its absolute integral is at most
$$
 \frac{C_j}{R}
 \left(\int_{L_{0,e}}^{CR}
       \frac{(1+\log u)^P}{u}\,\dd u\right)
 \prod_{f\neq e}I_f
 \leq C_j\frac{(1+\log R)^{P+1}}{R}.
$$
Summing over $e$ and the finite atlas proves \eqref{ch05:eq:global-reference-cutoff}.
\end{proof}

We use the cutoff estimates to prove that $\dd A_s$ has zero total integral, then apply Stokes' theorem to obtain the boundary limit for any smooth exhaustion. Let $K_n$ be an exhaustion of $\cM_g$ by compact smooth orbifold domains with smooth boundary, oriented by placing the outward normal first.

\begin{thm}\label{ch05:thm:global-flux}
For every $j\geq0$,
\begin{equation}\label{ch05:eq:signed-boundary-limit}
 \lim_{n\to\infty}\int_{\partial K_n}\partial_s^jA_s=0.
\end{equation}
The same limit holds separately for $\partial_s^j[T_s]_{2d-1}$ and $\partial_s^j[C_s]_{2d-1}$. In particular,
\begin{equation}\label{ch05:eq:global-stokes}
 Z_g(s)=P_g(s),\qquad
 \partial_s^jZ_g(s)
 =\int_{\bM_g}\partial_s^j[\bar\Theta_s]_{2d}.
\end{equation}
All limits and identities are locally uniform for positive $s$.
\end{thm}

\begin{proof}
The estimates in Proposition~\ref{ch05:prop:divisor-cutoff} give
$$
 \int_{\cM_g}|\dd\eta_R\wedge\partial_s^jA_s|
 \longrightarrow0.
$$
The form $\eta_R\partial_s^jA_s$ has compact support. Since $\dd A_s=F_s$, we obtain by Stokes' theorem for compactly supported forms
\begin{equation}\label{ch05:eq:cutoff-stokes}
 \int_{\cM_g}\eta_R\,\partial_s^jF_s
 =
 -\int_{\cM_g}\dd\eta_R\wedge\partial_s^jA_s.
\end{equation}
By Proposition~\ref{ch04:prop:global-finiteness}, the form $\partial_s^j[\Phi_s]_{2d}$ has an $L^1$ majorant. Exhausting cutoffs can be removed locally uniformly by Corollary~\ref{ch04:cor:global-cutoffs}. The form $\partial_s^j[\widetilde\Theta_s]_{2d}$ is the restriction of a form smooth on $\bM_g$. Thus $\partial_s^jF_s$ is integrable, with domination uniform on each compact positive $s$-interval. By dominated convergence in \eqref{ch05:eq:cutoff-stokes},
\begin{equation}\label{ch05:eq:zero-integrated-differential}
 \int_{\cM_g}\partial_s^jF_s=0.
\end{equation}

Every increasing sequence of regular values $R_n\to\infty$ of $H_\Delta$ gives compact smooth orbifold domains $K_n=\{H_\Delta\leq R_n\}$. Such values exist by Sard's theorem in the finite orbifold charts. For these domains, and for every exhaustion by compact smooth orbifold domains with smooth boundary, we obtain by Stokes' theorem
$$
 \int_{\partial K_n}\partial_s^jA_s
 =
 \int_{K_n}\partial_s^jF_s.
$$
The $L^1$ bound and \eqref{ch05:eq:zero-integrated-differential} show that the right side tends to zero locally uniformly in $s$. This proves \eqref{ch05:eq:signed-boundary-limit}.

For the correction $C_s$, the form $\dd\partial_s^j[C_s]_{2d-1}$ is integrable and has zero total integral by Proposition~\ref{ch05:prop:reference-control}. Applying Stokes' theorem on the exhaustion domains proves that its boundary integral also tends to zero. Since $T_s=\widetilde T_s+C_s$, the uncorrected primitive also has vanishing boundary integrals.

Finally, integration of $F_s=[\Phi_s-\widetilde\Theta_s]_{2d}$ in~\eqref{ch05:eq:zero-integrated-differential} gives~\eqref{ch05:eq:global-stokes}. The integral of $\Phi_s$ is the finite Goldman integral by Proposition~\ref{ch04:prop:global-finiteness}, while the integral of $\widetilde\Theta_s$ equals that of $\bar\Theta_s$, since the divisor has zero measure for a smooth top-degree form. The same dominating functions justify derivatives of every order with respect to $s$.
\end{proof}

\subsection{The weighted integral and intersection numbers}

\begin{proof}[Proof of Theorem~\ref{thm:weighted-volume}]
By Proposition~\ref{ch04:prop:global-finiteness}, $0<Z_g(s)<\infty$ for $s>0$ and, for every integer $j\geq0$,
\begin{equation}\label{eq:final-moment-derivatives}
 \partial_s^jZ_g(s)
 =(-1)^j\int_{\cH_g}\mu^j e^{-s\mu}\,\dd V_G,
\end{equation}
with differentiation under the integral locally uniform on $(0,\infty)$.

The global Stokes identity \eqref{ch05:eq:global-stokes} and the characteristic-class expansion \eqref{ch02:eq:fixed-expansion-kappa} yield
\begin{equation}\label{eq:final-intersection-expansion}
 Z_g(s)=P_g(s)=\sum_{k=0}^d A_{g,k}s^{-(r+k)},
\end{equation}
where the coefficients are those of \eqref{eq:intersection-coefficients}. Differentiating this finite sum and using \eqref{eq:final-moment-derivatives} proves the derivative assertions.
\end{proof}

\section{Exact computations and area distribution}\label{sec:volume-formulas}

We compute marked examples and the unmarked weighted volumes in genera two and three, then recover the area distribution and sublevel volumes by Laplace inversion.

\subsection{Marked polynomials}

The first marked polynomials are
\begin{equation}\label{eq:marked-small-polynomials}
 \begin{gathered}
 \mathcal V_{0,3}=1,\qquad
 \mathcal V_{0,4}=a+3t+\sum_{i=1}^4x_i,\qquad
 \mathcal V_{1,1}=\frac{a+4t+x_1}{24}.
 \end{gathered}
\end{equation}
\begin{equation}\label{eq:marked-genus-one-two}
\begin{split}
 \mathcal V_{1,2}
 ={}&\frac{a^2}{16}+\frac{5at}{12}+\frac{19t^2}{24}\\
 &+\left(\frac a{12}+\frac{7t}{24}\right)(x_1+x_2)
   +\frac{(x_1+x_2)^2}{48}.
\end{split}
\end{equation}
To verify the last formula directly, the required intersections on $\bM_{1,2}$ are
$$
 \int\kappa_1^2=\frac18,\qquad
 \int\kappa_2=\frac1{24},\qquad
 \int\lambda_1\kappa_1=\frac1{24},\qquad
 \int\kappa_1\psi_i=\frac1{12},\qquad
 \int\lambda_1\psi_i=\frac1{24}.
$$
The three pure descendant integrals with exponent tuples $(2,0),(1,1),(0,2)$ are all $1/24$. For example,
$$
 \int\kappa_1^2
 =\int_{\bM_{1,4}}\psi_3^2\psi_4^2
  -\int_{\bM_{1,3}}\psi_3^3
 =\frac16-\frac1{24}=\frac18.
$$
These values follow from the string and dilaton equations and $\int_{\bM_{1,1}}\psi_1 =\int_{\bM_{1,1}}\lambda_1=1/24$. Expanding \eqref{eq:marked-cubic-hodge} through degree two gives \eqref{eq:marked-genus-one-two}.

For $(g,n)=(0,5)$, setting $\mathbf x=\mathbf0$ gives
\begin{equation}\label{eq:marked-zero-five}
 \mathcal V_{0,5}(a,t;\mathbf0)
 =\frac52a^2+15at+25t^2.
\end{equation}
With~$p_1=a+3t$ and~$p_2=-a^2/2-3at-2t^2$, we obtain~$\mathcal V_{0,5}(a,t;\mathbf0)=3p_1^2+p_2$ from the genus-zero descendant formula~\eqref{eq:arithmetic-genus-zero}. The factor~$P_t$ in the insertion series contributes to both coefficients.

\subsection{Unmarked volumes in genera two and three}

For~$g\ge2$, the global Stokes identity and~\eqref{eq:marked-return-P} give
\begin{equation}\label{eq:marked-return-Z}
 Z_g(s)=t^{5g-5}\mathcal V_{g,0}(4\pi^2,t),
 \qquad t=\frac{2\pi}{s},\qquad s>0.
\end{equation}
Every coefficient is therefore determined by~\eqref{eq:marked-EO} and~\eqref{eq:marked-full-assembly}, or by the Hodge recurrences of Appendix~\ref{sec:arithmetic}.

In genus two, the needed intersections are
\begin{equation}\label{eq:g2-intersections}
\begin{aligned}
 \int_{\bM_2}\kappa_1^3&=\frac{43}{2880},&
 \int_{\bM_2}\kappa_1\kappa_2&=\frac1{240},&
 \int_{\bM_2}\kappa_3&=\frac1{1152},\\
 \int_{\bM_2}\lambda_1\kappa_1^2&=\frac{19}{2880},&
 \int_{\bM_2}\lambda_1\kappa_2&=\frac1{480},&
 \int_{\bM_2}\lambda_2\kappa_1&=\frac7{5760}.
\end{aligned}
\end{equation}
We compute these intersections from the recurrences in Appendix~\ref{sec:arithmetic} and the permutation identity~\eqref{eq:marked-cycle-identity}. Here~$\mathsf P_g(\mathbf d)=\int_{\bM_{g,n}}\prod_{i=1}^n\psi_i^{d_i}$ denotes the pure descendant integral, as in~\eqref{eq:arithmetic-pure-descendants}. For example, the values~$\mathsf P_2(4)=1/1152$, $\mathsf P_2(2,3)=29/5760$, and~$\mathsf P_2(2,2,2)=7/240$ give the first row after subtraction of the terms with merged cycles in~\eqref{eq:marked-cycle-identity}.

With $b=a+3t$, the degree-three part of~\eqref{eq:marked-cubic-hodge} gives
\begin{equation}
\begin{aligned}
 \mathcal V_{2,0}(a,t)=\int_{\bM_2}\Bigl[
 &\frac{b^3}{6}\kappa_1^3
 +\frac52bt^2\kappa_1\kappa_2+3t^3\kappa_3\\
 &+t\lambda_1\left(\frac{b^2}{2}\kappa_1^2+\frac52t^2\kappa_2\right)
 +t^2\lambda_2b\kappa_1\Bigr].
\end{aligned}
\end{equation}
Substitution of~\eqref{eq:g2-intersections} yields
\begin{equation}\label{eq:g2}
 \mathcal V_{2,0}(a,t)=\frac{43a^3}{17280}
 +\frac{37a^2t}{1440}+\frac{71at^2}{720}+\frac{67t^3}{480}.
\end{equation}
In the geometric normalization,
\begin{equation}\label{eq:Z2}
 Z_2(s)=t^5\left(
 \frac{43\pi^6}{270}
 +\frac{37\pi^4}{90}t
 +\frac{71\pi^2}{180}t^2
 +\frac{67}{480}t^3\right),
 \qquad t=\frac{2\pi}{s}.
\end{equation}

In genus three, the rational recursion of Appendix~\ref{sec:arithmetic} gives
\begin{equation}\label{eq:g3}
\begin{aligned}
 \mathcal V_{3,0}(a,t)={}&
 \frac{176557}{77414400}a^6
 +\frac{2981}{67200}a^5t
 +\frac{990263}{2580480}a^4t^2\\
 &+\frac{5462533}{2903040}a^3t^3
 +\frac{1519069}{276480}a^2t^4
 +\frac{103999}{11520}at^5
 +\frac{99869}{15360}t^6.
\end{aligned}
\end{equation}
The corresponding genus-three weighted volume is
\begin{equation}\label{eq:Z3}
\begin{aligned}
 Z_3(s)=t^{10}\Bigl[
 &\frac{176557\pi^{12}}{18900}
 +\frac{23848\pi^{10}}{525}t
 +\frac{990263\pi^8}{10080}t^2\\
 &+\frac{5462533\pi^6}{45360}t^3
 +\frac{1519069\pi^4}{17280}t^4
 +\frac{103999\pi^2}{2880}t^5
 +\frac{99869}{15360}t^6\Bigr],
 \quad t=\frac{2\pi}{s}.
\end{aligned}
\end{equation}

\subsection{The Weil--Petersson specialization}

For every stable pair~$(g,n)$, we use the full-stack convention for Weil--Petersson volume:
\begin{equation}\label{eq:WP-stack-definition}
 V^{\WP}_{g,n}(\mathbf L)
 :=\int_{\bM_{g,n}}
       \exp\!\left(2\pi^2\kappa_1+
                    \frac12\sum_{i=1}^nL_i^2\psi_i\right).
\end{equation}
For nonnegative lengths, this is the Weil--Petersson symplectic volume of the moduli stack of hyperbolic surfaces with labeled geodesic boundaries, with length zero interpreted as a cusp and with all stabilizer weights retained, including elements acting trivially on a local chart. For $n\ge1$, the identification follows from Mirzakhani's symplectic reduction and intersection formula~\cite[Theorems~4.3--4.4]{Mirzakhani2007} with these weights; for $n=0$, it follows from $[\omega_{\WP}]=2\pi^2\kappa_1$.

At~$t=0$, $C_{g,n}(0)=1$, and hence
\begin{equation}\label{eq:marked-WP-specialization}
 \mathcal V_{g,n}\left(2\pi^2,0;
                \frac{L_1^2}{2},\ldots,\frac{L_n^2}{2}\right)
 =V^{\WP}_{g,n}(L_1,\ldots,L_n).
\end{equation}
The generic stabilizers for $(g,n)=(1,1)$ and $(2,0)$ are cyclic of order two. Their stack weights are~$1/2$; passage to the effective orbifold doubles the volumes in these two cases~\cite[Section~3.1]{Do2011WP}. Thus
\begin{equation}\label{eq:marked-WP-one-one}
 V^{\WP}_{1,1}(L)=\frac{L^2+4\pi^2}{48}.
\end{equation}
In the same convention, the unmarked coefficients computed above give
\begin{equation}\label{eq:marked-WP-two-three}
 V^{\WP}_{2,0}=\frac{43\pi^6}{2160},\qquad
 V^{\WP}_{3,0}=\frac{176557\pi^{12}}{1209600}.
\end{equation}
For the geometric Goldman substitution~$a=4\pi^2$, homogeneity gives
\begin{equation}\label{eq:marked-WP-leading-factor}
 \mathcal V_{g,0}(4\pi^2,0)
 =2^{3g-3}V^{\WP}_{g,0}.
\end{equation}
This factor is the contribution of~$\omega_0=2\omega_{\WP}$.

The topological recursion and Hodge computations agree directly on the first cases beyond $(0,3)$ and $(1,1)$. For example, the residue coefficients for $(g,n)=(1,2)$ are
\begin{equation}\label{eq:marked-EO-one-two-check}
 \begin{gathered}
 J_{1,2}(0,0)=\frac{19}{24},\qquad
 J_{1,2}(1,0)=J_{1,2}(0,1)=-\frac7{24},\\
 J_{1,2}(2,0)=J_{1,2}(1,1)=J_{1,2}(0,2)=\frac1{24}.
 \end{gathered}
\end{equation}
They give~\eqref{eq:marked-genus-one-two} after the insertion formula. The rational recurrences of Appendix~\ref{sec:arithmetic} reproduce these coefficients and the genus-two and genus-three polynomials. For Hodge integrals at the top Hodge degree, the independent comparison is the formula of Faber and Pandharipande~\cite{FP}:
\begin{equation}\label{eq:marked-lambda-g-comparison}
 \int_{\bM_{g,n}}\lambda_g\prod_i\psi_i^{d_i}
 =\frac{(2g-3+n)!}{\prod_i d_i!}
   \frac{(2^{2g-1}-1)|B_{2g}|}{2^{2g-1}(2g)!},
 \quad g,n\ge1,\quad |\mathbf d|=2g-3+n.
\end{equation}
For~$(g,n)=(2,1)$ and~$d_1=2$, this gives~$\int_{\bM_{2,1}}\lambda_2\psi_1^2=7/5760$. By the projection formula, this integral equals~$\int_{\bM_2}\lambda_2\kappa_1$, in agreement with~\eqref{eq:g2-intersections}.

\subsection{Pushforward of the Goldman measure}

The weighted integrals in Theorem~\ref{thm:weighted-volume} determine the pushforward measure under $\mu$ on $[0,\infty)$. We use the uniqueness statement in \cite[Proposition~1.2, p.~2]{SSV2010}: two finite positive Borel measures on $[0,\infty)$ with equal Laplace transforms at every positive argument are equal.

\begin{proof}[Proof of Corollary~\ref{cor:moment-density}]
The Goldman density, integrated with the stack weights, defines the positive pushforward measure $\nu_g$. For every nonnegative Borel function $f$,
\begin{equation}\label{eq:pushforward-test}
 \int_{[0,\infty)}f(\tau)\,\dd\nu_g(\tau)
 =\int_{\cH_g}f(\mu)\,\dd V_G.
\end{equation}
For $b>0$, we have, by Theorem~\ref{thm:weighted-volume},
$$
 \nu_g([0,R])\leq e^{bR}Z_g(b)<\infty
 \qquad(R<\infty).
$$
Thus $\nu_g$ is locally finite, and $\dd\nu_{g,b}=e^{-b\tau}\dd\nu_g$ is a finite positive measure.

For every integer $m\geq1$, the integral representation of the gamma function gives
$$
 \int_0^\infty e^{-s\tau}\frac{\tau^{m-1}}{(m-1)!}\,\dd\tau=s^{-m}.
$$
Hence the real polynomial $p_g$ in \eqref{eq:moment-density} satisfies
\begin{equation}\label{eq:polynomial-transform}
 \int_0^\infty e^{-s\tau}p_g(\tau)\,\dd\tau=Z_g(s).
\end{equation}
Moreover,
$$
 \int_0^\infty e^{-b\tau}|p_g(\tau)|\,\dd\tau
 \leq\sum_{k=0}^d|A_{g,k}|b^{-(r+k)}<\infty.
$$
Therefore $\dd\zeta_{g,b}=e^{-b\tau}p_g(\tau)\,\dd\tau$ is a finite signed measure. Its Laplace transform agrees with that of $\nu_{g,b}$ at every positive argument, by \eqref{eq:polynomial-transform}. Writing $\zeta_{g,b}=\zeta_{g,b}^+-\zeta_{g,b}^-$, uniqueness for finite positive measures applied to $\nu_{g,b}+\zeta_{g,b}^-$ and $\zeta_{g,b}^+$ yields
$$
 \nu_{g,b}=\zeta_{g,b}.
$$

Removing the positive weight on bounded intervals gives
\begin{equation}\label{eq:final-measure-identification}
 \nu_g(B)=\int_Bp_g(\tau)\,\dd\tau
 \qquad\text{for bounded Borel sets }B\subset[0,\infty).
\end{equation}
Since $\nu_g$ is positive and $p_g$ is continuous, $p_g\geq0$, with $p_g(0)=0$ because $r\geq5$. Monotone convergence then extends~\eqref{eq:final-measure-identification} to all Borel sets $B\subset[0,\infty)$.
\end{proof}

\subsection{Sublevel volumes and the first coefficient}

\begin{proof}[Proof of Corollary~\ref{cor:sublevel-volume}]
By \eqref{eq:pushforward-test} and Corollary~\ref{cor:moment-density}, for $\Lambda\geq0$,
$$
 \Vol_G\{\mu\leq\Lambda\}
 =\nu_g([0,\Lambda])
 =\int_0^\Lambda p_g(\tau)\,\dd\tau
 =\sum_{k=0}^d A_{g,k}\frac{\Lambda^{r+k}}{(r+k)!}.
$$
Since $\mu=A_B-4\pi(g-1)$, the two sublevel sets in \eqref{eq:sublevel-volume} coincide. The largest possible exponents in the volume and density formulas are $r+d=N$ and $r+d-1=N-1$, respectively.
\end{proof}

\begin{proof}[Proof of Corollary~\ref{cor:wp-leading}]
Integration over the compactified moduli stack gives
$$
 V_{g,0}^{\WP}
 =\frac1{d!}\int_{\bM_g}[\omega_{\WP}]^d
 =\frac{(2\pi^2)^d}{d!}
   \int_{\bM_g}\kappa_1^d;
$$
see \cite[Theorem~7 and \S3.1]{Do2011WP}. The term $h_0=1$ in \eqref{eq:intersection-coefficients} gives
\begin{equation}\label{eq:final-leading-coefficient}
 A_{g,0}
 =\frac{(2\pi)^r(4\pi^2)^d}{d!}
   \int_{\bM_g}\kappa_1^d
 =(2\pi)^r2^dV_{g,0}^{\WP}.
\end{equation}
For each $g$, the remaining terms in the finite sum \eqref{eq:sublevel-volume} are $O(\Lambda^{r+1})$ as $\Lambda\downarrow0$.
\end{proof}

\appendix

\section{Evaluation by stable graphs}\label{ch02:sec:cubic-coefficients}

The topological recursion gives the coefficients without expanding boundary products. For comparison, we record the stable-graph evaluation of the same characteristic classes, including the automorphism and excess-intersection factors. For a stable graph $A$ of type $(g,n)$, legs are labeled, and vertex valences include both legs and half-edges. Let
$$
B_A=\prod_v\bM_{g_v,n_v},
\qquad
[A,\theta]=\frac1{|\Aut A|}\xi_{A,*}\theta,
$$
where $\theta$ is a polynomial in vertex logarithmic $\kappa$-classes and cotangent line classes at the markings. The automorphism group fixes every leg label and includes exchanges of the two half-edges of a loop. This is the decorated stratum class of \cite[Notation~2.23]{SchmittVanZelm2018}. We extend the notation linearly.

The required graph identities consist of pullback of decorations, the excess intersection formula, and the contribution from the ordered node cover. We recall each in the form used below.

On~$\bM_{g,n}$, the ordinary classes~$\widetilde\kappa_m=(\pi_{g,n})_*c_1(\omega_{\pi_{g,n}})^{m+1}$ satisfy
$$
 \kappa_m=\widetilde\kappa_m+\sum_{i=1}^n\psi_i^m,
 \qquad m\ge0,
$$
by~\cite[Notation~3.1.1, equations~(12)--(13)]{Chiodo08}. The sum on a vertex factor includes both external markings and half-edge markings.

For a contraction $f\colon\Gamma\to A$, write $\alpha_f$ for its vertex map and $\beta_f$ for the correspondence identifying retained half-edges and legs. The induced clutching map $\xi_f\colon B_\Gamma\to B_A$ satisfies
\begin{equation}
\xi_f^*\kappa_{v,m}
=
\sum_{w\in\alpha_f^{-1}(v)}\kappa_{w,m},
\qquad
\xi_f^*\psi_h=\psi_{\beta_f(h)}.
\label{ch02:eq:graph-decoration-pullback}
\end{equation}
The identities follow from \cite[Lemma~2.27 and Corollary~2.28]{SchmittVanZelm2018}. The pullback of a global logarithmic class is therefore the sum of the vertex logarithmic classes. Writing these in terms of $\widetilde\kappa_m$ retains every half-edge cotangent term. Powers in \eqref{ch02:eq:graph-decoration-pullback} are expanded by the ordinary multinomial formula.

For two stable graphs $A$ and $B$ of type $(g,n)$, the product formula is indexed by isomorphism classes of triples $(\Gamma,f,g')$ with contractions to $A$ and $B$. The contractions include their half-edge correspondences. Let $E_f$ and $E_{g'}$ be the edge subsets retained in the respective targets, and require
$$
E_f\cup E_{g'}=E(\Gamma).
$$
Isomorphisms of triples commute with both contractions. The product rule is
\begin{equation}
\begin{split}
[A,\theta]\,[B,\lambda]
={}&
\frac1{|\Aut A|\,|\Aut B|}
\sum_{(\Gamma,f,g')}
\xi_{\Gamma,*}
\left(
\xi_f^*\theta\,
\xi_{g'}^*\lambda
\prod_{\{h,h'\}\in E_f\cap E_{g'}}
(-\psi_h-\psi_{h'})
\right).
\end{split}
\label{ch02:eq:graph-product}
\end{equation}
The product rule \eqref{ch02:eq:graph-product} is given in \cite[paragraph~2.11 and Theorem~2.29]{SchmittVanZelm2018}. The excess bundle is the sum of the line bundles $L_h^\vee\otimes L_{h'}^\vee$ over the common retained edges, with first Chern classes $-\psi_h-\psi_{h'}$, as in \cite[paragraph~2.18 and Proposition~2.19]{SchmittVanZelm2018}. The moduli stacks here are ordinary stable marked moduli stacks, and the coefficients are rational, as required in the cited formulas. The right-hand pushforwards in \eqref{ch02:eq:graph-product} are unnormalized:
$$
\xi_{\Gamma,*}\eta
=|\Aut\Gamma|[\Gamma,\eta].
$$
Thus conversion back to normalized decorated classes multiplies the coefficient by $|\Aut\Gamma|$.

The contributions from the cover that orders the node branches can be expressed using graphs with one edge. For any polynomial $P$ symmetric in the two branch classes,
\begin{equation}
\frac12j_*P
=
\sum_{\substack{A\ \mathrm{stable}\\ |E(A)|=1}}
[A,P_A],
\label{ch02:eq:ordered-node-graph}
\end{equation}
where $P_A$ is the corresponding decoration on the two half-edges. For a nonseparating edge, the corresponding component of the ordered cover is the genus-$(g-1)$ vertex factor with its two branch markings. The graph has an automorphism of order two exchanging those markings. Its contribution is therefore $\frac12\xi_{A,*}P_A=[A,P_A]$.

For a separating edge with unequal vertex genera, the ordered cover has two components, according to which side is first. The two pushforwards agree because $P$ is symmetric, and the graph automorphism group is trivial. For equal vertex genera, the ordered component distinguishes the two factors, while the graph automorphism group has an automorphism exchanging the two factors. Its pushforward, multiplied by $1/2$, is again $[A,P_A]$. On pointed factors, this count respects the external leg labels. An exchange of the two sides is allowed precisely when it preserves those labels. Accordingly, one either divides by the automorphism of order two or sums over the two components of the ordered cover. These counts continue to hold when vertex curves acquire further nodes, because the node under consideration remains selected.

Exchanging the two arguments of $\gamma_{m-1}$ multiplies it by $(-1)^{m-1}$. For odd $m$, it is therefore symmetric. For positive even $m$, $B_{m+1}=0$. Consequently every node polynomial $B_{m+1}\gamma_{m-1}$ occurring in \eqref{ch02:eq:cubic-grr} or \eqref{ch02:eq:pointed-cubic-grr} is symmetric in the two branch classes and may be replaced using \eqref{ch02:eq:ordered-node-graph}. Formula~\eqref{ch02:eq:ordered-node-graph} already accounts for the factor $1/2$ from branch ordering; no additional factor is needed.

\begin{prop}
\label{ch02:prop:finite-graph-evaluation}
For every integer $k$ with $0\leq k\leq d$, the intersection number
$$
\int_{\bM_g}\kappa_1^{d-k}h_k
$$
is evaluated by a finite sum of products of compactified vertex $\kappa$- and $\psi$-intersection numbers with rational coefficients.
\end{prop}

\begin{proof}
The generating function \eqref{ch02:eq:symmetric-logarithm} gives
\begin{equation}
h_k=
\sum_{\substack{\ell_1,\ldots,\ell_k\geq0\\
\sum_{m=1}^km\ell_m=k}}
\prod_{m=1}^k
\frac{\bigl((m-1)!\ch_m(\bE_3)\bigr)^{\ell_m}}{\ell_m!},
\label{ch02:eq:symmetric-partitions}
\end{equation}
with $h_0=1$. Substitute \eqref{ch02:eq:cubic-grr}, multiply by $\kappa_1^{d-k}$, and use \eqref{ch02:eq:ordered-node-graph} to express the node terms as decorated one-edge classes. Products are then computed by \eqref{ch02:eq:graph-product}, with pullbacks from \eqref{ch02:eq:graph-decoration-pullback}. The conversion $\xi_{\Gamma,*}\eta=|\Aut\Gamma|[\Gamma,\eta]$ preserves the stated normalization.

For example, if $A$ has one vertex and one loop, the part of $[A,1]^2$ in which both contractions retain that same loop is
$$
\frac{2}{2\cdot2}\,\xi_{A,*}(-\psi_+-\psi_-).
$$
The numerator counts the two branch correspondences. Refinements with distinct retained loops contribute the remaining terms of \eqref{ch02:eq:graph-product}.

A decorated graph has codimension $|E(\Gamma)|+\deg(\theta)$. The product formula respects this degree because
$$
|E_f|+|E_{g'}|=|E(\Gamma)|+|E_f\cap E_{g'}|.
$$
We may therefore discard degrees above $d$ after each multiplication. There are finitely many stable graphs of type $(g,n)$ and finitely many contractions between them; the allowed decoration degree is bounded. Thus the procedure gives a finite sum.

Each decoration is a sum of tensor products over the vertices, and
\begin{equation}
\int_{\bM_{g,n}}\left[\Gamma,\bigotimes_v\theta_v\right]
=\frac1{|\Aut\Gamma|}\prod_v\int_{\bM_{g_v,n_v}}\theta_v.
\label{ch02:eq:graph-vertex-degree}
\end{equation}
Only degree $3g_v-3+n_v$ on every vertex factor contributes. All integrals are stack degrees, as in \eqref{ch02:eq:quotient-integration}; this is the intersection procedure of \cite[Theorem~2.29 and Remark~2.30]{SchmittVanZelm2018}.
\end{proof}

\section{Rational recursions for Hodge integrals}\label{sec:arithmetic}

We evaluate the Hodge integrals in~\eqref{eq:marked-Hodge-assembly} by reducing Chern characters to pure descendant intersections, giving a computation over $\mathbb Q$ independent of the spectral residues. Throughout this appendix, negative genus, an unstable pair, a negative descendant exponent, or a violation of the dimension constraint is understood to give value zero. Every splitting refers to labeled entries, even when their numerical values coincide, and its genus variables range over nonnegative integers.

\subsection{Descendant integrals}

For a stable pair $(g,n)$ and nonnegative integers $d_1,\ldots,d_n$, we set
\begin{equation}\label{eq:arithmetic-pure-descendants}
 \mathsf P_g(d_1,\ldots,d_n)
 =\int_{\bM_{g,n}}\prod_{i=1}^n\psi_i^{d_i}.
\end{equation}
The dimension equation is $\sum_i d_i=3g-3+n$. The two numerical initial values are
\begin{equation}\label{eq:arithmetic-seeds}
 \mathsf P_0(0,0,0)=1,\qquad \mathsf P_1(1)=\frac1{24}.
\end{equation}
The initial values in~\eqref{eq:arithmetic-seeds} take precedence over the string and dilaton equations. For a tuple~$\mathbf d$ of length~$n$, the latter equations read
\begin{equation}\label{eq:arithmetic-string-dilaton}
 \mathsf P_g(0,\mathbf d)
 =\sum_{i=1}^n\mathsf P_g(\mathbf d-\mathbf e_i),
 \qquad
 \mathsf P_g(1,\mathbf d)
 =(2g-2+n)\mathsf P_g(\mathbf d).
\end{equation}
For every tuple other than the two initial cases, if an entry is zero or one, we use the corresponding equality in \eqref{eq:arithmetic-string-dilaton}. If all entries are at least two, we select one entry $k$ and index the remaining entries by $I$. By the Witten--Kontsevich theorem in the form of Dijkgraaf, Verlinde, and Verlinde \cite{Kontsevich,DVV},
\begin{equation}\label{eq:arithmetic-DVV}
\begin{aligned}
 (2k+1)!!\,\mathsf P_g(k,\mathbf d_I)
 ={}&\sum_{i\in I}
 \frac{(2k+2d_i-1)!!}{(2d_i-1)!!}
 \mathsf P_g(k+d_i-1,\mathbf d_{I\setminus\{i\}})\\
 &+\frac12\sum_{\substack{a,b\ge0\\a+b=k-2}}
 (2a+1)!!(2b+1)!!
 \Biggl[\mathsf P_{g-1}(a,b,\mathbf d_I)\\
 &\hspace{23mm}
 +\sum_{\substack{g_1+g_2=g\\A\sqcup B=I}}
 \mathsf P_{g_1}(a,\mathbf d_A)
 \mathsf P_{g_2}(b,\mathbf d_B)\Biggr].
\end{aligned}
\end{equation}
The splitting $(g_1,A),(g_2,B)$ is ordered, with entries distinguished by their labels even when their exponents agree, and $(-1)!!=1$.

Every stable integral on the right-hand side of~\eqref{eq:arithmetic-string-dilaton} has smaller negative Euler characteristic~$2g-2+n$ than the integral on the left. The first and nonseparating terms in~\eqref{eq:arithmetic-DVV} decrease this integer by one. In a separating term, the sum of the corresponding integers is one less than the original value, and each stable factor has a positive value; hence both factors have smaller negative Euler characteristic. The initial values in~\eqref{eq:arithmetic-seeds} and the prescribed zero values terminate the recursion. In genus zero, repeated use of the string equation gives

\begin{equation}\label{eq:arithmetic-genus-zero}
 \mathsf P_0(d_1,\ldots,d_n)
 =\frac{(n-3)!}{\prod_i d_i!},
 \qquad \sum_i d_i=n-3.
\end{equation}

\subsection{Chern-character insertions}

For a tuple $K=(k_1,\ldots,k_m)$ of positive odd integers, we define
\begin{equation}\label{eq:arithmetic-H-definition}
 \mathsf H_g(\mathbf d;K)
 =\int_{\bM_{g,n}}
   \prod_{i=1}^n\psi_i^{d_i}
   \prod_{\nu=1}^m\ch_{k_\nu}(\mathbb E).
\end{equation}
Here the dimension equation is $|\mathbf d|+\sum_\nu k_\nu=3g-3+n$. We have $\mathsf H_g(\mathbf d;\varnothing)=\mathsf P_g(\mathbf d)$. For every $k\ge1$, $\ch_{2k}(\mathbb E)=0$; an integral with a positive-degree Chern-character insertion vanishes in genus zero. The Bernoulli numbers are defined by
$$
 \frac{u}{e^u-1}=\sum_{j\ge0}B_j\frac{u^j}{j!}.
$$
We select one labeled entry~$k$ of~$K$ and denote the tuple of remaining entries by~$K'$. We index the entries of~$K'$ by~$\{1,\ldots,|K'|\}$, and~$K'_A$ denotes the subtuple whose labels belong to~$A$.

\begin{prop}\label{prop:arithmetic-character-recursion}
For every stable $(g,n)$ and every admissible tuple $(\mathbf d;k,K')$,
\begin{equation}\label{eq:arithmetic-character-recursion}
\begin{aligned}
 \mathsf H_g(\mathbf d;k,K')
 =\frac{B_{k+1}}{(k+1)!}\Biggl\{
 &\mathsf H_g(\mathbf d,k+1;K')
 -\sum_{i=1}^n\mathsf H_g(\mathbf d+k\mathbf e_i;K')\\
 &+\frac12\sum_{\substack{a,b\ge0\\a+b=k-1}}(-1)^a
 \Biggl[\mathsf H_{g-1}(\mathbf d,a,b;K')\\
 &\quad+
 \sum_{\substack{g_1+g_2=g,\ I\sqcup J=\{1,\ldots,n\}\\
                 A\sqcup B=\{1,\ldots,|K'|\}}}
 \mathsf H_{g_1}(\mathbf d_I,a;K'_A)
 \mathsf H_{g_2}(\mathbf d_J,b;K'_B)
 \Biggr]\Biggr\}.
\end{aligned}
\end{equation}
Together with \eqref{eq:arithmetic-seeds}--\eqref{eq:arithmetic-DVV}, this determines every value of \eqref{eq:arithmetic-H-definition} by finitely many rational operations.
\end{prop}

\begin{proof}
By Mumford's Grothendieck--Riemann--Roch formula \cite{Mumford83}, or \cite[Theorem~2.1]{GLN} with root parameter one and exponent parameter one,
\begin{equation}\label{eq:arithmetic-Mumford}
 \ch_k(\mathbb E)
 =\frac{B_{k+1}}{(k+1)!}
 \left(\kappa_k-\sum_{i=1}^n\psi_i^k
 +\frac12j_*\sum_{a+b=k-1}(-1)^a\psi_+^a\psi_-^b\right).
\end{equation}
The morphism $j$ parametrizes the ordered branches of separating and nonseparating nodes. We multiply this identity by the remaining insertions in \eqref{eq:arithmetic-H-definition} and integrate.

For the $\kappa_k$ term, we use the morphism~$f\colon\bM_{g,n+1}\to\bM_{g,n}$ that forgets the last marking and stabilizes the curve. The identity $\kappa_k=f_*\psi_{n+1}^{k+1}$ and the projection formula pull the remaining insertions to $\bM_{g,n+1}$. The Hodge bundle pulls back under $f$, while each original cotangent line class satisfies $f^*\psi_i=\psi_i-D_{i,n+1}$, where the rational component of $D_{i,n+1}$ carries the markings $i,n+1$ and the attaching node. Since $\psi_{n+1}$ restricts to zero on this three-pointed component, multiplication by $\psi_{n+1}^{k+1}$ kills every correction term and gives $\mathsf H_g(\mathbf d,k+1;K')$. The marking terms in~\eqref{eq:arithmetic-Mumford} give the sum with a minus sign in~\eqref{eq:arithmetic-character-recursion}.

At a nonseparating node, the Hodge bundle restricts in $K$-theory to the lower-genus Hodge bundle plus a trivial line bundle. Its positive-degree Chern characters therefore restrict to those of the lower-genus bundle. At a separating node, the Hodge bundle restricts to the direct sum of the two Hodge bundles. We distribute each remaining labeled Chern-character factor over the two components. The resulting terms are precisely the two boundary contributions in \eqref{eq:arithmetic-character-recursion}. The factor $1/2$ compensates for ordering the two branches.

Every integral on the right-hand side of~\eqref{eq:arithmetic-character-recursion} has fewer Chern-character factors than the integral on the left, even when the first term adds a marking. Together with the strict decrease of~$2g-2+n$ in the pure descendant recursion, this makes~$(|K|,2g-2+n)$ decrease lexicographically for every stable term. The prescribed initial and zero values therefore terminate the calculation, and the finiteness of each sum gives every value by finitely many rational operations.
\end{proof}

\subsection{Hodge classes}

We recover the Chern classes of $\mathbb E$ from its odd Chern characters. Since
$$
 \log c_z(\mathbb E)
 =\sum_{k\ge1}(-1)^{k-1}(k-1)!z^k\ch_k(\mathbb E)
$$
and $\ch_{2k}(\mathbb E)=0$ for $k\ge1$, coefficient extraction gives
\begin{equation}\label{eq:arithmetic-lambda}
 \lambda_j
 =\sum_{\substack{q_1,q_3,q_5,\ldots\ge0\\
                 \sum_{k\ \mathrm{odd}}kq_k=j}}
 \prod_{k\ \mathrm{odd}}
 \frac{\bigl((k-1)!\ch_k(\mathbb E)\bigr)^{q_k}}{q_k!}.
\end{equation}
At $j=0$, the unique empty term equals one; for $j>g$, $\lambda_j=0$ by the rank of $\mathbb E$. For a given $j$, only the odd integers $k\le j$ occur, and each $q_k\le j/k$. Hence \eqref{eq:arithmetic-lambda} is a finite sum. Each integral obtained by multiplying it by descendant classes is evaluated by Proposition~\ref{prop:arithmetic-character-recursion}. For example,
$$
 \lambda_1=\ch_1(\mathbb E),\qquad
 \lambda_2=\frac{\ch_1(\mathbb E)^2}{2},\qquad
 \lambda_3=\frac{\ch_1(\mathbb E)^3}{6}
              +2\ch_3(\mathbb E).
$$

Every integral in \eqref{eq:marked-Hodge-assembly} can therefore be evaluated using \eqref{eq:arithmetic-DVV}, \eqref{eq:arithmetic-character-recursion}, and \eqref{eq:arithmetic-lambda}. The finite dimension bounds in that formula complete an independent rational computation of all marked polynomials $\mathcal V_{g,n}$.

At $a=0$, the insertion series $T_{0,t}(w)$ in \eqref{eq:marked-Hodge-assembly} has only the terms $3tw^2-2t^2w^3$. For $d_0=D_{g,n}-|\mathbf d|\ge0$, this specializes to
\begin{equation}\label{eq:arithmetic-two-insertions}
\begin{aligned}
 F_{g,n}(0,t;\mathbf d)
 =t^{d_0}\sum_{j=0}^{\min(g,d_0)}
 \sum_{\substack{m,q\ge0\\m+2q=d_0-j}}
 \frac{3^m(-2)^q}{m!\,q!}
 \int_{\bM_{g,n+m+q}}\lambda_j
 \prod_{i=1}^n\psi_i^{d_i}
 \prod_{v=1}^{m}\psi_{n+v}^{2}
 \prod_{v=1}^{q}\psi_{n+m+v}^{3}.
\end{aligned}
\end{equation}
Only $\psi^2$ and $\psi^3$ insertions occur at the added markings. The relation $j+m+2q=d_0$ accounts for the common power $t^{d_0}$, including the case $d_0=0$.

\section{Localized boundary integrals}
\label{ch05:sec:local-flux}

At a smooth point of the boundary divisor, we estimate the boundary integral over a level set of the normal coordinate using Proposition~\ref{ch05:prop:annular-estimate}. The localized Stokes formula also contains a term involving the derivative of the cutoff along the stratum.

Let $B$ be a stratum coordinate domain near a smooth divisor point, $t=e^{-L+\ii\phi}$, and $x=L^{-1}$. Choose $\chi=\chi(b)\in C_c^\infty(B,\R)$ and denote its stratum differential by $\dd_B\chi$. The complex orientation is the stratum orientation followed by $\dd x\wedge\dd\phi$. The boundary face $\{x=\epsilon\}$ of $\{x\geq\epsilon\}$ is therefore oriented by minus the stratum orientation followed by increasing $\phi$. For $0<a\leq L_*^{-1}$, write
$$
 \mathcal F_s(a;\chi)
 =\int_{x=a,\mathrm{inner}}\chi A_s.
$$
Choose an integer $P\geq D$ large enough for the reference bounds in Proposition~\ref{ch05:prop:reference-control} on this chart.

\begin{cor}\label{ch05:cor:localized-flux}
For every $j\geq0$, uniformly on compact positive $s$-intervals,
$$
 \left|
 \int_{x=\epsilon,\mathrm{inner}}
 \chi\,\partial_s^j[T_s]_{2d-1}
 \right|
 \leq C_j\|\chi\|_{C^1}
       \epsilon(1+|\log\epsilon|)^D.
$$
The corrected primitive satisfies
\begin{equation}\label{ch05:eq:localized-corrected-flux}
 \left|
 \int_{x=\epsilon,\mathrm{inner}}\chi\,\partial_s^jA_s
 \right|
 \leq C_j\|\chi\|_{C^1}
       \epsilon(1+|\log\epsilon|)^P.
\end{equation}
The integral of the term involving $\dd_B\chi$ is absolutely convergent, with tail
$$
 \int_{B\times\{L\geq\ell\}\times S^1}
 |\dd_B\chi\wedge\partial_s^j A_s|
 \leq C_j\|\chi\|_{C^1}\frac{(1+\log\ell)^P}{\ell}.
$$
Moreover,
\begin{equation}\label{ch05:eq:side-limit}
 \lim_{\epsilon\to0}
 \int_{B\times[\epsilon,a]\times S^1}\dd_B\chi\wedge A_s
 =
 -\mathcal F_s(a;\chi)
 -\int_{B\times(0,a]\times S^1}\chi F_s.
\end{equation}
The last integral is absolutely convergent. The remainder in \eqref{ch05:eq:side-limit} is $O(\|\chi\|_{C^1}\epsilon(1+|\log\epsilon|)^P)$, also after every $s$-derivative.
\end{cor}

\begin{proof}
With $k=1$, we obtain from Corollary~\ref{ch05:cor:annular-consequences}
$$
 \int_{B\times\{L\geq\ell\}\times S^1}
 |\dd_B\chi\wedge\partial_s^j[T_s]_{2d-1}|
 \leq C_j\|\chi\|_{C^1}\frac{(1+\log\ell)^D}{\ell}
$$
by summing dyadic annuli. The averaged integral also satisfies
\begin{equation}\label{ch05:eq:averaged-angular}
 \int_{B\times[\ell,2\ell]\times S^1}
 \left|
 \frac{\dd L}{\ell}\wedge
 \chi\,\partial_s^j[T_s]_{2d-1}
 \right|
 \leq C_j\|\chi\|_{C^0}\frac{(1+\log\ell)^D}{\ell}.
\end{equation}
Both assertions also hold for the absolute defining integrals restricted to the tails in $J$ and $\mu$ in Proposition~\ref{ch05:prop:annular-estimate}.

The coefficient of $[\partial_s^j\Phi_s]_{2d}$ in the coordinates $(b,L,\phi)$ is $O(L^{-2})$ by Corollary~\ref{ch04:cor:density-integration}. This rate also holds for $\partial_s^j\Theta_s$. In fact, the case $q=0$ of \eqref{ch04:eq:horizontal-quadratic} gives rates $1,L^{-1},L^{-2}$ for coefficients of $\omega_0$ with zero, one, or two normal coordinate differentials. The Chern forms smooth on the compactification have a factor $e^{-L}$ for each such differential. Each nonzero wedge product of top degree contains both normal coordinate differentials, which gives the asserted rate.

Set
$$
 G_j(L)=
 \int_{x=1/L,\mathrm{inner}}
 \chi\,\partial_s^j[T_s]_{2d-1}.
$$
Applying Stokes' theorem on $B\times[L_1,L_2]\times S^1$ and using $\dd T_s=\Phi_s-\Theta_s$, we obtain
\begin{equation}\label{ch05:eq:face-integral-cauchy}
 |G_j(L_2)-G_j(L_1)|
 \leq C_j\|\chi\|_{C^1}
       \frac{(1+\log L_1)^D}{L_1},
 \qquad L_2\geq L_1.
\end{equation}
The term from $\chi\,\partial_s^j(\Phi_s-\Theta_s)$ has an absolute $O(L_1^{-1})$ tail. The term from $\dd_B\chi\wedge\partial_s^j[T_s]_{2d-1}$ has the dyadic bound at the beginning of the proof. The functions $G_j$ therefore have finite limits, with this convergence rate and the stated uniformity. The averaged bound \eqref{ch05:eq:averaged-angular} implies
\begin{equation}\label{ch05:eq:averaged-face-limit}
 \frac1\ell\int_\ell^{2\ell}|G_j(L)|\,\dd L
 \leq C_j\|\chi\|_{C^0}
       \frac{(1+\log\ell)^D}{\ell}\longrightarrow0.
\end{equation}
The average on the left tends to the absolute value of the finite limit of $G_j$, so that limit vanishes and the stated estimate for $[T_s]_{2d-1}$ follows.

For the reference correction, \eqref{ch05:eq:reference-coefficients} bounds the coefficients containing $\dd\phi$ by $O((1+\log L)^P/L)$ and those containing both $\dd L$ and $\dd\phi$ by $O((1+\log L)^P/L^2)$. The bounds \eqref{ch05:eq:reference-face} and \eqref{ch05:eq:logarithmic-tail} give the corresponding face and side bounds for $C_s$. Their subtraction proves \eqref{ch05:eq:localized-corrected-flux}.

On $B\times[\epsilon,a]\times S^1$, the induced orientation of the face~$\{x=a\}$ is opposite to its orientation as the boundary of~$\{x\geq a\}$ used in~$\mathcal F_s(a;\chi)$. Since
$$
 \dd(\chi A_s)=\dd_B\chi\wedge A_s+\chi F_s,
$$
we obtain by Stokes' theorem
$$
 \mathcal F_s(\epsilon;\chi)-\mathcal F_s(a;\chi)
 =
 \int_{B\times[\epsilon,a]\times S^1}
       \dd_B\chi\wedge A_s
 +\int_{B\times[\epsilon,a]\times S^1}\chi F_s.
$$
The first boundary term tends to zero by~\eqref{ch05:eq:localized-corrected-flux}, the side term converges absolutely, and the volume term has an absolute $O(\epsilon)$ tail. Taking limits gives~\eqref{ch05:eq:side-limit} with the stated sign, and the same proof applies to every $s$-derivative.

The limits agree on overlaps and are compatible with partitions of unity along the stratum; they descend to the quotient by Lemma~\ref{ch04:lem:descent}.
\end{proof}

\end{document}